\documentclass[12pt]{article}
\usepackage{amsmath,amsthm,amsfonts,amscd,amsxtra,amsopn,amssymb,verbatim,pdfsync}
\usepackage[mathscr]{eucal}
\usepackage{ stmaryrd }
\usepackage{graphicx}
\usepackage{epstopdf}
\usepackage{tikz} 
\usepackage{enumerate}
\usetikzlibrary{arrows,backgrounds,calc} 
\usetikzlibrary{matrix}

\usepackage{psfrag}
\input xy
\xyoption{all}

\newcommand\C{\mathbb C}

\newcommand\F{\mathbb F}

\newcommand\N{\mathbb N}

\newcommand\R{\mathbb R}

\newcommand\Z{\mathbb Z}

\newcommand\cC{\mathcal C}

\newcommand{\ra}{\longrightarrow}

\newcommand{\into}{\hookrightarrow}

\def\doublearrow #1 #2
{\quad\raisebox{.1cm}{$\overset{#1}\ra$}
\hskip -.67cm
\raisebox{-.1cm}{$\underset {#2}\ra$}\quad}

\newcommand\wt{\widetilde}
\newcommand\wh{\widehat}
\renewcommand\t{\tilde}

\newcommand\str{\operatorname{str}}
\newcommand\tr{\operatorname{tr}}

\newcommand\op{\operatorname{op}}

\newcommand{\id}{\operatorname{id}}

\newcommand\p{\partial}

\newcommand\ev{\opn{ev}}
\newcommand\coev{\opn{coev}}
\newcommand\classtr{\opn{cl-tr}}

\newcommand{\one}{I}
\newcommand{\ie}{{i.e.\;}}

\newcommand\Hom{\operatorname{Hom}}

\newcommand\Iff{if and only if }

\newcommand{\nb}{\nobreakdash}

\newcommand\scr{\mathscr}
\renewcommand\={\overset{\text{def}}=}
\newcommand{\opn}[1]{\operatorname{#1}}

\newcommand{\sB}{\text{\sf{B}}}
\newcommand{\sC}{\text{\sf{C}}}
\newcommand{\sD}{\text{\sf{D}}}

\newcommand{\sF}{\text{\sf{F}}}

\newcommand{\Fun}{\text{\sf{Fun}}}

\newcommand{\Set}{\text{\sf{Set}}}

\newcommand\TV{\text{\sf TV}}
\newcommand\Ban{\text{\sf Ban}}
\newcommand{\Vect}{\text{\sf{Vect}}}
\newcommand{\SVect}{\text{\sf{SVect}}}

\newcommand{\RBord} [1]{#1\text{-\hskip -.01in{\sf RBord}}}

\newtheorem{prop}{Proposition}

\newtheorem{thm}[prop]{Theorem}
\newtheorem{cor}[prop]{Corollary}
\newtheorem{lem}[prop]{Lemma}

\newtheorem{ddefn}[prop]{Definition}
\newtheorem{eex}[prop]{Example}
\newtheorem{rrem}[prop]{Remark}
\newtheorem{eexercise}[prop]{Exercise}
\newtheorem{con}[prop]{Conjecture}
\newtheorem{hhome}[prop]{Homework}
\newtheorem{nnumber}[prop]{}

\newenvironment{defn}{\begin{ddefn}\rm}{\end{ddefn}}
\newenvironment{ex}{\begin{eex}\rm}{\end{eex}}
\newenvironment{rem}{\begin{rrem}\rm}{\end{rrem}}

{\catcode`@=11\global\let\c@equation=\c@prop}

\numberwithin{prop}{section}

\title{Traces in monoidal categories}
\author{Stephan Stolz and Peter Teichner}
\begin{document}
\maketitle
\abstract{This paper contains the construction, examples and properties of a trace and a trace pairing for certain morphisms in a monoidal category with switching isomorphisms. Our construction of the categorical trace is a common generalization of the trace for endomorphisms of dualizable objects in a balanced monoidal category and the trace of nuclear operators on a topological vector space with the approximation property. In a forthcoming paper, applications to the partition function of super symmetric field theories will be given.}
\tableofcontents
\section{Introduction}
The results of this paper provide an essential step in our proof that  the partition function of a Euclidean field theory of dimension $2|1$ is an integral modular function \cite{STpartition}. While motivated by field theory, the two main results are the construction of traces and trace pairings for certain morphisms in a monoidal category. 
Let  $\sC$ be a monoidal category with  monoidal unit $\one\in \sC$ \cite{McL}. 
\medskip

\noindent{\bf Question.} What conditions on an endomorphism $f\in \sC(X,X)$ allow us to construct a well-defined trace $\tr(f)\in \sC(\one,\one)$ with the usual properties expected of a trace?

\medskip

Theorem \ref{thm:main1} below provides an answer to this question. Our construction is a common generalization of the following two well-known classical cases:
\begin{enumerate}
\item If $X\in \sC$ is a dualizable object (see Definition \ref{def:dualizable}), then every endomorphism $f$ has a well-defined trace \cite[Proposition 3.1]{JSV}.
\item If $f\colon X\to X$ is a nuclear operator (see Definitions \ref{def:nuclear_Ban} and \ref{def:nuclear_TV}) on a topological vector space $X$, then $f$ has a well-defined trace provided $X$ has the {\em approximation property}, i.e., the identity operator on $X$ can be approximated by finite rank operators in the compact open topology \cite{Li}. 
\end{enumerate}
Let $\TV$ be the category of topological vector spaces (more precisely, these are assumed to be locally convex, complete, Hausdorff), equipped with the monoidal structure given by the projective tensor product (see section \ref{subsec:TV}). Then an object $X\in \TV$ is dualizable \Iff $X$ is finite dimensional, whereas every Hilbert space has the approximation property. Hence extending the trace from endomorphisms of dualizable objects of $\sC$ to more general objects is analogous to extending the notion of trace from endomorphisms of finite dimensional vector spaces to certain infinite dimensional topological vector spaces. In fact, our answer will involve analogues of the notions {\em nuclear} and {\em approximation property} for general monoidal categories which we now describe. 

The following notion is our analogue of a nuclear morphism.

\begin{defn}\label{def:thick} A morphism $f\colon X\to Y$ in a monoidal category $\sC$ is {\em thick} if it can be factored in the form
\begin{equation}\label{eq:factorization}
\xymatrix{
X\cong \one\otimes X\ar[rr]^{t\otimes \id_X}&&Y\otimes Z\otimes X\ar[rr]^{\id_Y\otimes b}&&Y\otimes\one\cong Y
}
\end{equation}
for morphisms $t\colon \one\to Y\otimes Z$, $b\colon Z\otimes X\to \one$.
\end{defn}

As explained in the next section, the terminology is motivated by considering the bordism category.  In the category $\Vect$ of vector spaces, with monoidal structure given by the tensor product, a morphism $f\colon X\to Y$ is thick \Iff it has finite rank (see Theorem \ref{thm:vect}).  In the category $\TV$ a morphism is thick \Iff it is nuclear (see Theorem \ref{thm:TV}).

If $f\colon X\to X$ is a thick endomorphism with a factorization as above, we attempt to define its {\em categorical trace} $\tr(f)\in\sC(\one,\one)$ to be the composition
\begin{equation}\label{eq:trace}
\xymatrix{
\one\ar[r]^<>(.5)t&X\otimes Z\ar[rr]^{s_{X,Z}}&&Z\otimes X\ar[r]^<>(.5)b&\one
}.
\end{equation}
This categorical trace depends on the choice of a natural family of isomorphisms $s=\{s_{X,Y}\colon X\otimes Y\to Y\otimes X\}$ for $X,Y\in \sC$. We don't assume that $s$ satisfies the relations \eqref{eq:relations} required for the braiding isomorphism of a braided monoidal category. Apparently lacking an established name, we will refer to $s$ as {\em switching isomorphisms}. We would like to thank Mike Shulman for this suggestion.

  For the monoidal category $\Vect$, equipped with the standard switching isomorphism $s_{X,Y}\colon X\otimes Y\to Y\otimes X$, $x\otimes y\mapsto y\otimes x$, the categorical trace of a finite rank (i.e. thick) endomorphism $f\colon X\to X$ agrees with its classical trace (see Theorem \ref{thm:vect}). More generally,  if $X$ is a dualizable object of a monoidal category $\sC$,  the above definition agrees with the classical definition of the trace in that situation (Theorem \ref{thm:dualizable}). In general, the above trace is {\em not well-defined}, since it might depend on the factorization of $f$ given by the triple $(Z,b,t)$ rather than just the morphism $f$. As we will see in Section~\ref{subsec:Ban}, this happens for example in the category of Banach spaces.

To understand the problem with defining $\tr(f)$, let us write $\wh\tr(Z,t,b)\in \sC(\one,\one)$ for the composition \eqref{eq:trace}, and $\Psi(Z,t,b)\in \sC(X,Y)$ for the composition \eqref{eq:factorization}. 
There is an equivalence relation on these triples (see Definition \ref{def:thickened}) such that $\wh\tr(Z,t,b)$ and $\Psi(Z,t,b)$ depend only on the equivalence class $[Z,t,b]$. In other words, there are well-defined maps
\[
\wh\tr\colon \wh\sC(X,X)\ra \sC(\one,\one)
\qquad
\Psi\colon \wh\sC(X,Y)\ra \sC(X,Y)
\]
where $\wh\sC(X,Y)$ denotes the  equivalence classes of triples $(Z,t,b)$ for fixed $X,Y\in\sC$. We note that by construction the image of $\Psi$ consists of the thick morphisms from $X$ to $Y$. We will call elements of $\wh\sC(X,Y)$  {\em thickened morphisms}. If $\wh f\in \wh\sC(X,Y)$ with $\Psi(\wh f)=f\in \sC(X,Y)$, we say that $\wh f$ is a {\em thickener} of $f$. 

Using the notation $\sC^{tk}(X,Y)$ for the set of thick morphisms from $X$ to $Y$, it is clear that there is a well-defined trace map $\tr\colon \sC^{tk}(X,X)\to \sC(\one,\one)$ making the diagram 
\begin{equation}\label{eq:trace_diagram}
\xymatrix{
\sC^{tk}(X,X)\ar@{-->}[rr]^{\tr}&&\sC(\one,\one)\\
&\wh\sC(X,X)\ar@{->>}[lu]^{\Psi}\ar[ru]_{\wh\tr}
}
\end{equation}
commutative \Iff $X$ has the following property:

\begin{defn} An object $X$ in a monoidal category $\sC$ with switching isomorphisms has the {\em trace property} if the map $\wh\tr$ is constant on the fibers of $\Psi$.
\end{defn}

For the category $\Ban$ of Banach spaces and continuous maps, we will show in Section~\ref{subsec:Ban} that the map $\Psi$ can be identified with the homomorphism
\begin{equation}
\Phi\colon Y\otimes X'\ra \Ban(X,Y)
\qquad
w\otimes f\mapsto (v\mapsto w f(v))
\end{equation}
where $X'$ is the Banach space of continuous linear maps $f\colon X\to \C$ equipped with the operator norm and $\otimes$ is the projective tensor product. 
Operators in the image of $\Phi$ are referred to as {\em nuclear operators}, and hence a morphism in $\Ban$ is thick \Iff it is nuclear. It is a classical result  that the trace property for a Banach space $X$ is equivalent to the injectivity of the map $\Phi$ which in turn is equivalent to the {\em approximation property} for $X$: the identity operator of $X$ can be approximated by finite rank operators in the compact-open topology, see e.g.\ \cite[\S 43.2(7)]{Ko}. Every Hilbert space has the approximation property, but deciding whether a  Banach space has this property is surprisingly difficult. Grothendieck asked this question in the fifties, but the first example of a Banach space {\em without} the approximation property was found by Enflo only in 1973 \cite{En}. Building on Enflo's work, Szankowski showed in 1981 that the Banach space of bounded operators on an (infinite dimensional) Hilbert space does not have the approximation property \cite{Sz}.

\begin{thm}\label{thm:main1} Let $\sC$ be a monoidal category with switching isomorphisms, i.e. $\sC$ comes equipped with a family of natural isomorphisms $s_{X,Y}\colon X\otimes Y\to Y\otimes X$. If $X\in\sC$ is an object with the trace property, then the above categorical trace $\tr(f)\in \sC(\one, \one)$ is well-defined for any thick endomorphism $f\colon X\to X$. This compares to the two classical situations mentioned above as follows:

\begin{enumerate}[(i)]
\item If $X$ is a dualizable object, then $X$ has the trace property and any endomorphism $f$ of $X$ is thick. Moreover, the categorical trace of $f$ agrees with its classical trace.
\item In the category $\TV$ of  topological vector spaces (locally convex, complete, Hausdorff), a morphism  is thick \Iff it is nuclear, and the approximation property of an object $X\in \TV$ implies the trace property. Moreover, if $f\colon X\to X$ is a nuclear endomorphism of an object with the approximation property, then the categorical trace of $f$ agrees with its classical trace.
\end{enumerate}
\end{thm}

The first part sums up our discussion above. Statements (i) and (ii) appear as Theorem \ref{thm:dualizable} respectively Theorem \ref{thm:TV} below. 
It would be interesting to find an object in $\TV$ which has the trace property but not the approximation property.
 
To motivate our second main result, Theorem \ref{thm:main2}, we note that a monoidal functor $F\colon \sC\to \sD$ preserves thick and thickened morphisms and gives commutative diagrams for the map $\Psi$  from \eqref{eq:trace_diagram}. If $F$ is compatible with the switching isomorphisms then it also commutes with $\wh\tr$.  
However, the {\em trace property} is {\em not} functorial in the sense that if some object $X\in \sC$ has the trace property then  it is {\em not} necessarily inherited by $F(X)$ (unless $F$ is essentially surjective and full or has some other special property). In particular, when the functor $F$ is a field theory, then, as explained in the next section, this non-functoriality causes a problem for calculating the partition function of $F$.

We circumvent this problem by replacing the {\em trace} by a closely related {\em trace pairing}
\begin{equation}\label{eq:trace_pairing}
\tr\colon \sC^{tk}(X,Y)\times \sC^{tk}(Y,X)\ra \sC(\one,\one)
\end{equation}
 for objects $X$, $Y$ of a monoidal category $\sC$ with switching isomorphisms. Unlike the trace map 
$
\tr\colon \sC^{tk}(X,X)\ra \sC(\one,\one)
$
 discussed above, which is only defined if $X$ has the trace property, {\em no condition on $X$ or $Y$} is needed to define this trace pairing $\tr(f,g)$ as follows. Let $\wh f\in \wh\sC(X,Y)$, $\wh g\in \wh \sC(Y,X)$ be thickeners of $f$ respectively $g$ (i.e.,  $\Psi(\wh f)=f$ and $\Psi(\wh g)=g$). We will show that elements of $\wh \sC(X,Y)$ can be pre-composed or post-composed with ordinary morphisms in $\sC$ (see Lemma \ref{lem:hat_composition1}). This composition gives elements $\wh f\circ g$ and $f\circ \wh g$ in $\wh \sC(Y,Y)$ which we will show to be equal in Lemma \ref{lem:hat_composition2}. Hence the trace pairing defined by
\[
\tr(f,g):=\wh\tr(\wh f\circ g)=\wh\tr(f\circ \wh g)\in \sC(\one,\one)
\]
is {\em independent} of the choice of $\wh f$ and $\wh g$. 
We note that $\Psi(\wh f\circ g)=\Psi(f\circ \wh g)=f\circ g\in\sC^{tk}(Y,Y)$ and hence if 
$Y$ has the trace property, then 
\begin{equation}\label{eq:trace-pairing_trace}
\tr(f, g)=\tr(f\circ g)
\qquad\text{for $f\in \sC^{tk}(X,Y)$, $g\in \sC^{tk}(Y,X)$}.
\end{equation}
In other words, the trace pairing $\tr(f,g)$ is a generalization of the categorical trace of $f\circ g$, defined in situations where this trace might not be well-defined.

The trace pairing has the following properties that are analogous to properties one expects to hold for a trace. We note that the relationship \eqref{eq:trace-pairing_trace} immediately implies these properties for our trace defined for thick endomorphism of objects satisfying the trace property. 

\begin{thm}\label{thm:main2} Let $\sC$ be a monoidal category with switching isomorphisms. Then the trace pairing \eqref{eq:trace_pairing} is functorial and has the following properties:
\begin{enumerate}
\item $\tr(f,g)=\tr(g,f)$ for thick morphisms $f\in \sC^{tk}(X,Y)$, $g\in \sC^{tk}(Y,X)$. If $Y$ has the trace property then $\tr(f, g)=\tr(f\circ g)$ and symmetrically for $X$.
\item If $\sC$ is an additive category with distributive monoidal structure (see Definition \ref{def:distributive}), then the trace pairing is a bilinear map.
\item  $\tr(f_1\otimes f_2,g_1\otimes g_2)=\tr(f_1,g_1)\tr(f_2,g_2)$ 
for $f_i\in \sC^{tk}(X_i,Y_i)$, $g_i\in \sC^{tk}(Y_i,X_i)$, provided $s$ gives $\sC$ the structure of a {\em symmetric} monoidal category. More generally, this property holds if $\sC$ is a {\em balanced monoidal category}.
\end{enumerate}
\end{thm}

We recall that a balanced monoidal category is a braided monoidal category equipped with a natural family of isomorphisms $\theta=\{\theta_X\colon X\to X\}$ called {\em twists} satisfying a compatibility condition (see Definition \ref{def:balanced}). Symmetric monoidal categories are balanced monoidal categories with $\theta\equiv\id$. For a balanced monoidal category $\sC$ with braiding isomorphism $c_{X,Y}\colon X\otimes Y\to Y\otimes X$ and twist $\theta_X\colon X\to X$, one defines the {\em switching isomorphism} $s_{X,Y}\colon X\otimes Y\to Y\otimes X$ by 
$s_{X,Y}:=(\id_Y\otimes\theta_X)\circ c_{X,Y}$.

There are interesting examples of balanced monoidal categories that are not symmetric monoidal, e.g., categories of bimodules over a fixed von Neumann algebra (monoidal structure given by Connes fusion) or categories of modules over quantum groups. Traces in the latter are used to produce polynomial invariants for knots. Originally, we only proved the multiplicative property of our trace pairing for symmetric monoidal categories.
We are grateful to Gregor Masbaum for pointing out to us the classical definition of the trace of an endomorphism of a dualizable object in a balanced monoidal category which involves using the twist (see \cite{JSV}).

The rest of this paper is organized as follows. In section 2 we explain the motivating example: we consider the $d$\nb-dimensional Riemannian bordism category, explain what a {\em thick} morphism in that category is,  and show that the partition function of a $2$\nb-dimensional Riemannian field theory can be expressed as the relative trace of the thick operators that a field theory associates to annuli. Section 2 is motivational and can be skipped by a reader who wants to see the precise definition of $\wh \sC(X,Y)$,  the construction of $\wh \tr$ and a statement of the properties of $\wh\tr$ which are presented in section 3. In section 4 we discuss thick morphisms and their traces in various categories. In section 5 we prove the properties of $\wh \tr$ and deduce the corresponding properties of the trace pairing stated as Theorem \ref{thm:main2} above. 

\thanks{Both authors were partially supported by NSF grants. They would like to thank the referee for many valuable suggestions. The first author visited the second author at the Max-Planck-Institut in Bonn during the Fall of 2009 and in July 2010. He would like to thank the institute for the support and for the stimulating atmosphere.}

\section{Motivation via field theories}\label{sec:motivation}

A well-known aximatization of field theory is due to Graeme Segal \cite{Se} who defines a field theory as a monoidal functor from a bordism category to the category $\TV$ of  topological vector spaces. The precise definition of the bordism category depends on the type of field theory considered; for a $d$-dimensional {\em topological} field theory, the objects are closed $(d-1)$-dimensional manifolds and morphisms are $d$\nb-dimensional bordisms (more precisely, equivalence classes of bordisms where we identify bordisms if they are diffeomorphic relative boundary). Composition is given by gluing of bordisms, and the monoidal structure is given by disjoint union. 

For other types of field theories, the manifolds constituting the objects and morphisms in the bordism category come equipped with an appropriate geometric structure; e.g., a {\em conformal} structure for {\em conformal} field theories, a {\em Riemannian} metric for {\em Riemannian} field theories, or a {\em Euclidean} structure ($=$ Riemannian metric with vanishing curvature tensor) for a {\em Euclidean} field theory. In these cases more care is needed in the definition of the bordism category to ensure the existence of a well-defined composition and the existence of identity morphisms.

Let us consider the Riemannian bordism category $\RBord{d}$. The objects of $\RBord{d}$ are closed Riemannian $(d-1)$\nb-manifolds. A morphism from $X$ to $Y$ is a $d$\nb-dimensional Riemannian bordism $\Sigma$ from $X$ to $Y$, that is, a Riemannian $d$\nb-manifold $\Sigma$ with boundary and an isometry $X\amalg Y\to \p\Sigma$. More precisely, a morphism is an equivalence class of Riemannian bordisms, where two bordisms $\Sigma$, $\Sigma'$ are considered equivalent  if there is an isometry $\Sigma\to \Sigma'$ compatible with the boundary identifications. In order to have a well-defined compositon by gluing Riemannian bordisms we require that all metrics are product metrics near the boundary. To ensure the existence of identity morphisms, we enlarge the set of morphisms from $X$ to $Y$ by also including all isometries $X\to Y$. Pre- or post-composition of a bordism with an isometry is the given bordism with boundary identification modified by the isometry. In particular, the identity isometry $Y\to Y$ provides the identity morphism for $Y$ as object of the Riemannian bordism category $\RBord{d}$.

A more sophisticated way to deal with the issues addressed above  was developed in our paper \cite{STsurvey}. There we don't require the metrics on the bordisms to be a product metric near the boundary; rather, we have more sophisticated objects consisting of a closed $(d-1)$\nb-manifold equipped with a Riemannian collar. Also, it is technically advantageous not to mix Riemannian bordisms and isometries. This is achieved in that paper by constructing a suitable double category (or equivalently, a category internal to categories), whose vertical morphisms are isometries and whose horizontal morphisms are bordisms between closed $(d-1)$\nb-manifolds equipped with Riemannian collars. The $2$\nb-morphisms are isometries of such bordisms, relative boundary. When using the results of the current paper in \cite{STpartition}, we translate between the approach here using categories versus the approach via internal categories used in \cite{STsurvey}.
 
Let $E$ be $d$\nb-dimensional Riemannian field theory, that is, a symmetric monoidal functor
\[
E\colon \RBord{d}\ra \TV.
\]
For the bordism category $\RBord{d}$ the symmetric monoidal structure  is given by disjoint union; for the category $\TV$ it is given by the projective tensor product. Let $X$ be a closed Riemannian $(d-1)$\nb-manifold and $\Sigma$ be a Riemannian bordism from $X$ to itself. Let $\Sigma_{gl}$ be the closed Riemannian manifold obtained by gluing the two boundary pieces (via the identity on $X$). Both $\Sigma$ and $\Sigma_{gl}$ are morphisms in $\RBord{d}$:
\[
\Sigma\colon X\ra X
\qquad
\Sigma_{gl}\colon \emptyset\ra \emptyset
\]
We note that $\emptyset$ is the monoidal unit in $\RBord{d}$, and hence the vector space $E(\emptyset)$ can be identified with $\C$, the monoidal unit in $\TV$. In particular, $E(\Sigma_{gl})\in \Hom(E(\emptyset),E(\emptyset))=\Hom(\C,\C)=\C$ is a complex number.

\medskip

\noindent{\bf Question.} How can we calculate $E(\Sigma_{gl})\in\C$ in terms of the operator $E(\Sigma)\colon E(X)\to E(X)$?

\medskip

We would like to say that $E(\Sigma_{gl})$ is the {\em trace} of the operator $E(\Sigma)$, but to do so we need to check that the conditions guaranteeing a well-defined trace are met. For a {\em topological} field theory $E$ this is easy: In the topological bordism category every object $X$ is dualizable (see Definition \ref{def:dualizable}), hence $E(X)$ is dualizable in $\TV$ which is equivalent to $\dim E(X)<\infty$. By contrast, for a {\em Euclidean} field theory the vector space $E(X)$ is typically infinite dimensional, and hence to make sense of the trace of the operator $E(\Sigma)$ associated to a bordism $\Sigma$ from $X$ to itself, we need to check that the operator $E(\Sigma)$ is thick and that the vector space $E(X)$ has the trace property. 

It is easy to prove (see Theorem \ref{thm:bord}) that every object $X$ of the bordism category $\RBord{d}$ has the trace property and that among the morphisms of $\RBord{d}$ (consisting of Riemannian bordisms and isometries), exactly the bordisms are thick. The latter characterization motivated the adjective `thick', since we think of isometries as `infinitely thin' Riemannian bordisms. It is straightforward to check that being {\em thick} is a {\em functorial property} in the sense that the thickness of $\Sigma$ implies that $E(\Sigma)$ is thick. Unfortunately, as already mentioned in the introduction, the trace property is not functorial and we cannot conclude that $E(X)$ has the trace property.

Replacing the problematical trace by the well-behaved {\em trace pairing} leads to the following result. It is applied in \cite{STpartition} to prove the modularity and integrality of the partition function of a supersymmetric Euclidean field theory of dimension $2$. 
 
\begin{thm} Suppose $\Sigma_1$ is a Riemannian bordism of dimension $d$ from $X$ to $Y$, and $\Sigma_2$ is a Riemannian bordism from $Y$ to $X$. Let $\Sigma=\Sigma_1\circ \Sigma_2$ be the bordism from $Y$ to itself obtained by composing the bordisms $\Sigma_1$ and $\Sigma_2$, and let $\Sigma_{gl}$ be the closed Riemannian $d$\nb-manifold obtained from $\Sigma$ by identifying the two copies of $Y$ that make up its boundary. If $E$ is $d$\nb-dimensional Riemannian field theory, then 
\[
E(\Sigma_{gl})=\tr(E(\Sigma_2),E(\Sigma_1)).
\]
\end{thm}

\begin{proof} 
By Theorem \ref{thm:bord} the bordisms $\Sigma_1\colon X\to Y$ and $\Sigma_2\colon Y\to X$ are  thick  morphisms in $\RBord{d}$, and hence the morphism $\tr(\Sigma_1,\Sigma_2)\colon \emptyset\to \emptyset$ is defined. Moreover, every object $X\in\RBord{d}$ has the trace property (see Theorem \ref{thm:bord}) and hence
\[
\tr(\Sigma_1,\Sigma_2)=\tr(\Sigma_2\circ \Sigma_1)=\tr(\Sigma)
\]
In part (3) of Theorem \ref{thm:bord} we will show that $\tr(\Sigma)=\Sigma_{gl}$. Then functoriality of the construction of the trace pairing implies
\[
E(\Sigma_{gl})=E(\tr(\Sigma_1,\Sigma_2))=\tr(E(\Sigma_1),E(\Sigma_2))
\]
\end{proof}

\section{Thickened morphisms and their traces} \label{sec:thick}
In this section we will define the  {\em thickened morphisms} $\wh\sC(X,Y)$ and the trace $\wh\tr(\wh f)\in \sC(\one,\one)$ of thickened endomorphisms $\wh f\in \wh\sC(X,X)$ for a monoidal category $\sC$ equipped with a natural family of isomorphisms $s_{X,Y}\colon X\otimes Y\to Y\otimes X$. 

We recall that a {\em monoidal category} is a category $\sC$ equipped with a functor 
\[
\otimes\colon \sC\times\sC\to\sC
\]
called the {\em tensor product}, a distinguished element $\one\in \sC$ and natural isomorphisms
\begin{align*}
\alpha_{X,Y,Z}&\colon (X\otimes Y)\otimes Z\overset\cong\ra X\otimes(Y\otimes Z)
&&\text{(associator)}\\
\ell_X&\colon \one\otimes X\overset\cong\ra X
&&\text{(left unit constraint)}\\
r_X&\colon X\otimes \one\overset\cong\ra X
&&\text{(right unit constraint)}
\end{align*}
for objects $X,Y,Z\in \sC$. These natural isomorphisms are required to make two diagrams (known as the {\em associativity pentagon} and the {\em triangle for unit}) commutative, see \cite{McL}.

It is common to use diagrams to represent morphisms in $\sC$ (see for example \cite{JS1}). The pictures 

 
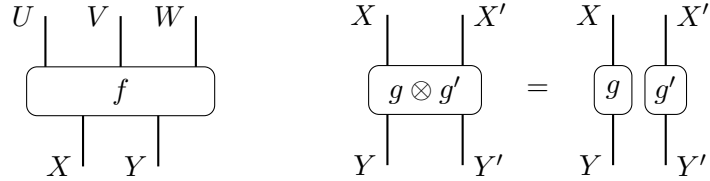
\begin{center}
\begin{tikzpicture}
	\draw[thick] (0,0)  to (0,1)node[left]{$U$};
	\draw[thick] (1,0)  to (1,1)node[left]{$V$};
	\draw[thick] (2,0)  to (2,1)node[left]{$W$};
	\draw[thick] (.5,0)  to (.5,-1)node[left]{$X$};
	\draw[thick] (1.5,0)  to (1.5,-1)node[left]{$Y$};
	\node[rectangle,rounded corners,fill=white,draw,minimum height=.65cm,minimum width=2.5cm] at (1,0){$f$};
\end{tikzpicture}
\qquad\qquad
\begin{tikzpicture}
	\draw[thick] (0,0)  to (0,1)node[left]{$X$};
	\draw[thick] (1,0)  to (1,1)node[right]{$X'$};
	\draw[thick] (0,0)  to (0,-1)node[left]{$Y$};
	\draw[thick] (1,0)  to (1,-1)node[right]{$Y'$};
	\node[rectangle,rounded corners,fill=white,draw,minimum height=.65cm,minimum width=1.5cm] at (.5,0){$g\otimes g'$};
	\node(=) at (2,0){$=$};
	\draw[thick] (3,0)  to (3,1)node[left]{$X$};
	\draw[thick] (3.7,0)  to (3.7,1)node[right]{$X'$};
	\draw[thick] (3,0)  to (3,-1)node[left]{$Y$};
	\draw[thick] (3.7,0)  to (3.7,-1)node[right]{$Y'$};
	\node[rectangle,rounded corners,fill=white,draw,minimum height=.65cm,minimum width=.5cm] at (3,0){$g$};
	\node[rectangle,rounded corners,fill=white,draw,minimum height=.65cm,minimum width=.5cm] at (3.7,0){$g'$};
\end{tikzpicture}
\end{center}
represent a morphism $f\colon U\otimes V\otimes W\to X\otimes Y$ and the tensor product of morphisms $g\colon X\to Y$ and $g'\colon X'\to Y'$, respectively. 
The composition $h\circ g$ of morphisms $g\colon X\to Y$ and $h\colon Y\to Z$ is represented by the picture
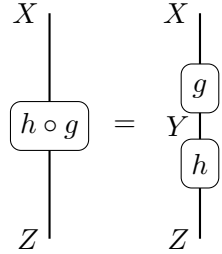
\begin{center}
\begin{tikzpicture}
	\draw[thick] (0,-.5)  to (0,1)node[left]{$X$};
	\draw[thick] (0,-.5)  to (0,-2)node[left]{$Z$};
	\node[rectangle,rounded corners,fill=white,draw,minimum height=.65cm,minimum width=1cm] at (0,-.5){$h\circ g$};
	\node(=) at (1,-.5){$=$};
	\draw[thick] (2,0)  to (2,1)node[left]{$X$};
	\draw[thick] (2,0)  to node[left]{$Y$}(2,-1);
	\draw[thick] (2,-1)  to (2,-2)node[left]{$Z$};
		\node[rectangle,rounded corners,fill=white,draw,minimum height=.65cm,minimum width=.5cm] at (2,0){$g$};
	\node[rectangle,rounded corners,fill=white,draw,minimum height=.65cm,minimum width=.5cm] at (2,-1){$h$};
\end{tikzpicture}
\end{center}

With tensor products being represented by juxtaposition of pictures, the isomorphisms $X\cong X\otimes\one\cong \one \otimes X$ suggest to delete edges labeled by the monoidal unit $\one$ from our picture. E.g., the pictures
\begin{equation}
\begin{tikzpicture}
	\draw[thick] (0,0) -- +(0,-1.5) node[left]{$Y$};
	\draw[thick] (1,0) -- +(0,-1.5) node[right]{$Z$};
	\draw[thick] (5,-1) -- +(0,1.5) node[left]{$Z$} ;
	\draw[thick] (6,-1) -- +(0,1.5) node[right]{$X$} ;
\node[rectangle,rounded corners,fill=white,draw,minimum height=.5cm,minimum width=1.5cm] at (.5,0){$t$};
\node[rectangle,rounded corners,fill=white,draw,minimum height=.5cm,minimum width=1.5cm] at (5.5,-1){$b$};
\end{tikzpicture}
\end{equation}
represent morphisms $t\colon \one\to Y\otimes Z$ and $b\colon Z\otimes X\to \one$, respectively.
Rephrasing Definition \ref{def:thick} of the introduction in our pictorial notation, a morphism $f\colon X\to Y$ in $\sC$ is {\em thick} if is can be factored in the form
\begin{equation}\label{eq:up_down}
\begin{tikzpicture}
	\draw[thick] (2,0) -- +(0,-2) node[left]{$Y$};
	\draw[thick] (3,0)  -- node[left]{$Z$} +(0,-1.5) ;
	\draw[thick] (4,-1.5) -- +(0,2) node[right]{$X$} ;
	\draw[thick] (0,.5)node[left]{$X$} -- +(0,-2.5) node[left]{$Y$};
\node at (1,-.75){$=$};
\node[rectangle,rounded corners,fill=white,draw,minimum height=.5cm,minimum width=1.5cm] at (2.5,0){$t$};
\node[rectangle,rounded corners,fill=white,draw,minimum height=.5cm,minimum width=1.5cm] at (3.5,-1.5){$b$};
\node[rectangle,rounded corners,fill=white,draw,minimum height=.5cm,minimum width=.5cm] at (0,-.75){$f$};
\end{tikzpicture}
\end{equation}
Here $t$ stands for {\em top} and $b$ for {\em bottom}. We will use the notation $\sC^{tk}(X,Y)\subset \sC(X,Y)$ for the subset of thick morphisms.

\subsection{Thickened morphisms}
It will be convenient for us to characterize the  thick morphisms as the image of a map 
\[
\Psi\colon \wh \sC(X,Y)\ra\sC(X,Y),
\]
the domain of which we refer to as  {\em thickened morphisms}. 

\begin{defn}\label{def:thickened}
Given objects $X,Y\in\sC$, a {\em thickened morphism} from $X$ to $Y$ is an equivalence class of  triples $(Z,t,b)$ consisting of   an object $Z\in \sC$, and morphisms 
\[
t\colon \one\ra Y\otimes Z
\qquad
b\colon Z\otimes X\ra \one
\]
To describe the equivalence relation, it is useful to think of these triples as objects of a category, and to define 
a morphism from $(Z,t,b)$ to  $(Z',t',b')$ to be a morphism $g\in \sC( Z,Z')$ such that
\begin{equation}\label{eq:equivalence}
\begin{tikzpicture}
	\def\hor{0}
	\def\vert{0}
\node at ($(0,-1)+(\hor,\vert)$){$t'$};
\node at ($(1,-1)+(\hor,\vert)$){$=$};
	\def\hor{2}
	\def\vert{0}
\draw[thick]($(1,0)+(\hor,\vert)$) --node[right]{$Z$}+(0,-1.5);
\draw[thick]($(0,0)+(\hor,\vert)$) --+(0,-2.5)node[left]{$Y$};
\draw[thick]($(1,-1.5)+(\hor,\vert)$) --+(0,-1)node[right]{$Z'$};
\node[rectangle,rounded corners,fill=white,draw,minimum height=.5cm,minimum width=1.5cm] at ($(.5,0)+(\hor,\vert)$){$t$};
\node[rectangle,rounded corners,fill=white,draw,minimum height=.5cm,minimum width=.5cm] at ($(1,-1.5)+(\hor,\vert)$){$g$};
\node at (5,-1){and};
\node at (7,-1){$b$};
\node at (8,-1){$=$};
	\def\hor{9}
	\def\vert{0}
\draw[thick]($(0,0)+(\hor,\vert)$)node[left]{$Z$} --+(0,-1);
\draw[thick]($(0,-1)+(\hor,\vert)$) --node[left]{$Z'$}+(0,-1.5);
\draw[thick]($(1,0)+(\hor,\vert)$)node[right]{$X$} --+(0,-2.5);
\node[rectangle,rounded corners,fill=white,draw,minimum height=.5cm,minimum width=1.5cm] at ($(.5,-2.5,0)+(\hor,\vert)$){$b'$};
\node[rectangle,rounded corners,fill=white,draw,minimum height=.5cm,minimum width=.5cm] at ($(0,-1)+(\hor,\vert)$){$g$};
\end{tikzpicture}
\end{equation}
Two triples $(Z,t,b)$ and $(Z',t',b')$ are {\em equivalent} if there is are triples $(Z_i,t_i,b_i)$ for $i=1,\dots, n$ with $(Z,t,b)=(Z_1,t_1,b_1)$ and $(Z',t',b')=(Z_n,t_n,b_n)$ and morphisms $g_i$ between  $(Z_i,t_i,b_i)$ and $(Z_{i+1},t_{i+1},b_{i+1})$ (this means that $g_i$ is either a morphism from $(Z_i,t_i,b_i)$ to  $(Z_{i+1},t_{i+1},b_{i+1})$ or from $(Z_{i+1},t_{i+1},b_{i+1})$ to $(Z_i,t_i,b_i)$). In other words, a thickened morphism is a path component of the category defined above. We write $\wh\sC(X,Y)$ for the thickened morphisms from $X$ to $Y$. 
 \end{defn}

As suggested by the referee, it will be useful to regard $\wh\sC(X,Y)$ as a {\em coend}; this will streamline the proofs of some results. Let us consider the functor
\[
S\colon \sC^{op}\times \sC\ra \Set
\qquad\text{given by}\qquad
(Z',Z)\mapsto \sC(\one, Y\otimes Z)\times \sC(Z'\otimes X,\one)
\]
Then the elements of $\coprod_{Z\in\sC}S(Z,Z)$ are triples $(Z,t,b)$ with $t\in \sC(\one,Y\otimes Z)$ and $b\in \sC(Z\otimes X,\one)$. 
Any morphism $g\colon Z\to Z'$ induces maps
\[
\xymatrix@1{\sC(Z',Z)\ar[rr]^{S(g,\id_Z)}&&\sC(Z,Z)}
\qquad\text{and}\qquad
\xymatrix@1{\sC(Z',Z)\ar[rr]^{S(\id_{Z'},g)}&&\sC(Z',Z')}
\]
We note that for any $(t,b')\in S(Z',Z)$ the two triples 
\[
(Z,S(g,\id_Z)(t,b'))=(Z,t,b'\circ g)
\qquad\text{and}\qquad
(Z',S(\id_{Z'},g)(t,b'))=(Z',g\circ t,b')
\]
represent the {\em same} element in $\wh C(X,Y)$. In fact, by construction $\wh C(X,Y)$ is the coequalizer
\[
\opn{coequalizer}\left(\xymatrix{
\displaystyle{\coprod_{g\in \sC(Z,Z')}}S(Z',Z)\,\ar@<1.5ex>[rr]^{\coprod S(g,\id_Z)}\ar@<.5ex>[rr]_{\coprod S(\id_{Z'},g)}&&\,\displaystyle{\coprod_{Z\in\sC}} S(Z,Z)
}\right)
\]
i.e., the quotient space of $\coprod_{Z\in\sC}S(Z,Z)$ obtained by identifying all image points of these two maps. 

This coequalizer can be formed for any functor $S\colon \sC^{op}\times \sC\to\Set$; it is called the {\em coend} of $S$, and following \cite[Ch.\ IX, \S  6]{McL}, we will use the integral notation 
\[
\int^{Z\in \sC}S(Z,Z)
\]
for the coend. Summarizing our discussion, we have the following way of expressing  $\wh\sC(X,Y)$ as a coend:
\begin{equation}\label{eq:coend}
\wh\sC(X,Y)=\int^{Z\in\sC}  \sC(I,Y\otimes Z)\times \sC(Z\otimes X,I)
\end{equation}

\begin{lem}\label{lem:Psi_welldefined} Given a triple $(Z,t,b)$ as above, let $\Psi(Z,t,b)\in\sC(X,Y)$ be the composition on the right hand side of equation \eqref{eq:up_down}. 
Then  $\Psi$ only depends on the equivalence class $[Z,t,b]$ of $(Z,t,b)$, \ie the following map is well-defined:
\begin{equation}\label{eq:Psi}
\begin{tikzpicture}
[scale=.7,
big/.style={rectangle,rounded corners,fill=white,draw,minimum height=.3cm,minimum width=1cm},
small/.style={rectangle,rounded corners,fill=white,draw,minimum height=.3cm,minimum width=.3cm}]
\node at (0,-.5){$\Psi\colon \wh \sC(X,Y)\ra\sC(X,Y)$};
\node at (6,-.5){$[Z,t,b]\mapsto\Psi(Z,t,b)=$};
\begin{scope}[shift={(9,0)}]
\draw[thick](0,-.5)--(0,-2.5)node[below]{$Y$};
\draw[thick](1,-.5)--(1,-2);
\node at(1.4,-1.25){$X^{\spcheck}$};
\draw[thick](2,-2)--(2,.5)node[above]{$X$};
\node at (.5,-.5)[big]{$t$};
\node at (1.5,-2)[big]{$\ev$};
\end{scope}
\end{tikzpicture}
\end{equation}
\end{lem}
We note that by construction the image of $\Psi$ is equal to $\sC^{tk}(X,Y)$, the set of thick morphisms from $X$ to $Y$.  As mentioned in the introduction, for $f\in \sC(X,Y)$ we call any $\wh f=[Z,t,b]\in \wh\sC(X,Y)$ with $\Psi(\wh f)=f$ a {\em thickener} of $f$.

\begin{rem}\label{rem:thick}
The difference between an {\em orientable} versus an {\em oriented} manifold is that the former is a property, whereas the latter is an additional structure on the manifold. In a similar vain, being {\em thick} is a property of a morphism $f\in \sC(X,Y)$, whereas a {\em thickener} $\wh f$ is an additional structure. To make the analogy between these situations perfect, we were tempted to introduce the words {\em thickenable} or {\em thickable} into mathematical English. However, we finally decided against it, particularly because the thick-thin distinction for morphisms in the bordism category is just perfectly suited for the purpose.
\end{rem}

\begin{proof} Suppose that $g\colon Z\to Z'$ is an equivalence from $(Z,t,b)$ to $(Z',t',b')$ in the sense of Definition \ref{def:thickened}. Then the following diagram shows that $\Psi(Z',t',b')=\Psi(Z,t,b)$.
\begin{equation*}
\begin{tikzpicture}
	\def\hor{0}
	\def\vert{0}
	\draw[thick] ($(1,0)+(\hor,\vert)$)  -- node[left]{$Z'$} +(0,-2) ;
	\draw[thick] ($(2,-2)+(\hor,\vert)$) -- +(0,2.5) node[right]{$X$} ;
	\draw[thick] ($(0,0)+(\hor,\vert)$) -- +(0,-2.5) node[left]{$Y$};
	\node[rectangle,rounded corners,fill=white,draw,minimum height=.5cm,minimum width=1.5cm] at ($(.5,0)+(\hor,\vert)$){$t'$};
	\node[rectangle,rounded corners,fill=white,draw,minimum height=.5cm,minimum width=1.5cm] at ($(1.5,-2)+(\hor,\vert)$){$b'$};
	\node at ($(3,-1)+(\hor,\vert)$){$=$};
	\def\hor{4}
	\def\vert{0}
	\draw[thick] ($(1,0)+(\hor,\vert)$)  -- node[left]{$Z$} +(0,-1) ;
	\draw[thick] ($(1,-1)+(\hor,\vert)$)  -- node[left]{$Z'$} +(0,-1) ;
	\draw[thick] ($(2,-2)+(\hor,\vert)$) -- +(0,2.5) node[right]{$X$} ;
	\draw[thick] ($(0,0)+(\hor,\vert)$) -- +(0,-2.5) node[left]{$Y$};
	\node[rectangle,rounded corners,fill=white,draw,minimum height=.5cm,minimum width=1.5cm] at ($(.5,0)+(\hor,\vert)$){$t$};
	\node[rectangle,rounded corners,fill=white,draw,minimum height=.5cm,minimum width=1.5cm] at ($(1.5,-2)+(\hor,\vert)$){$b'$};
	\node[rectangle,rounded corners,fill=white,draw,minimum height=.5cm,minimum width=.5cm] at ($(1,-1)+(\hor,\vert)$){$g$};
\node at ($(3,-1)+(\hor,\vert)$){$=$};
	\def\hor{8}
	\def\vert{0}
	\draw[thick] ($(1,0)+(\hor,\vert)$)  -- node[left]{$Z$} +(0,-2) ;
	\draw[thick] ($(2,-2)+(\hor,\vert)$) -- +(0,2.5) node[right]{$X$} ;
	\draw[thick] ($(0,0)+(\hor,\vert)$) -- +(0,-2.5) node[left]{$Y$};
	\node[rectangle,rounded corners,fill=white,draw,minimum height=.5cm,minimum width=1.5cm] at ($(.5,0)+(\hor,\vert)$){$t$};
	\node[rectangle,rounded corners,fill=white,draw,minimum height=.5cm,minimum width=1.5cm] at ($(1.5,-2)+(\hor,\vert)$){$b$};
\end{tikzpicture}
\end{equation*}
\end{proof}

Thickened morphisms can be pre-composed or post-composed with ordinary morphisms to obtain again thickened morphisms as follows:
\begin{align*}
&\xymatrix@1{
\wh\sC(Y,W)\times\cC(X,Y)\ar[r]^<>(.5){\circ}
&\wh\sC(X,W)}
\qquad &([Z,t,b],f)\mapsto [Z,t,b\circ (\id_Z\otimes f)]\\
&\xymatrix{
\sC(Y,W)\times\wh\cC(X,Y)\ar[r]^<>(.5){\circ}
&\wh\sC(X,W)}
\qquad
&(f,[Z,t,b])\mapsto [Z,(f\otimes \id_Z)\circ t,b]
\end{align*}

The proof of the following result is straightforward and we leave it to the reader.
\begin{lem} \label{lem:hat_composition1} The composition of morphisms with thickened morphisms is well-defined and compatible with the usual composition via the map $\Psi$ in the sense that the following diagrams commutes:
\[
\xymatrix{
\wh\sC(Y,W)\times\sC(X,Y)\ar[d]^{\Psi\times\id}\ar[r]^<>(.5){\circ}
&\wh\sC(X,W)\ar[d]^\Psi\\
\sC(Y,W)\times\sC(X,Y)\ar[r]^<>(.5){\circ}&\sC(X,W)
}
\qquad
\xymatrix{
\sC(Y,W)\times\wh\sC(X,Y)\ar[d]^{\id\times\Psi}\ar[r]^<>(.5){\circ}
&\wh\sC(X,W)\ar[d]^{\Psi}\\
\sC(Y,W)\times\sC(X,Y)\ar[r]^<>(.5){\circ}&\sC(X,W)
}
\]
\end{lem}

\begin{cor} If $f\colon X\to Y$ is a thick morphism in a monoidal category $\sC$, then the compositions $f\circ g$ and $h\circ f$ are thick for any morphisms $g\colon W\to X$, $h\colon Y\to Z$.
\end{cor}

\begin{lem} \label{lem:hat_composition2} Let $\wh f_1\in \wh \sC(X,Y)$, $\wh f_2\in \wh \sC(U,X)$ and $f_i=\Psi(\wh f_i)$. Then $\wh f_1\circ f_2=f_1\circ \wh f_2\in \wh\sC(U,Y)$.
\end{lem}

\begin{proof} Let $\wh f_i=[Z_i,t_i,b_i]$. Then
$
\wh f_1\circ f_2=[Z_1,t_1,b_1\circ (\id_{Z_1}\otimes f_2)]
$ and 
\[
\begin{tikzpicture}
[big/.style={rectangle,rounded corners,fill=white,draw,minimum height=.6cm,minimum width=1.5cm},
small/.style={rectangle,rounded corners,fill=white,draw,minimum height=.6cm,minimum width=.6cm}]
\begin{scope}[shift={(1,0)}]
	\draw[thick] (0,0)node[left]{$Z_1$} -- + (0,-3);
	\draw[thick] (1,0)node[right]{$U$} --  +(0,-1.5);
	\draw[thick] (1,-1.5) --  node[right]{$X$}+(0,-1.5);
	\node at (.5,-3)[big]{$b_1$};
	\node at (1,-1.5)[small]{$f_2$};
\end{scope}
\begin{scope}[shift={(4,0)}]
	\draw[thick] (0,0)node[left]{$Z_1$} --  +(0,-3);
	\draw[thick] (1,-1.5) --  node[left]{$X$}+(0,-1.5);
	\draw[thick] (2,-1.5) --  node[right]{$Z_2$}+(0,-1.5);
	\draw[thick] (3,0)node[right]{$U$} --  +(0,-3);
	\node at (.5,-3)[big]{$b_1$};
	\node at (1.5,-1.5)[big]{$t_2$};
	\node at (2.5,-3)[big]{$b_2$};
\end{scope}
\begin{scope}[shift={(9,0)}]
	\draw[thick] (0,0)node[left]{$Z_1$} --  +(0,-1.5);
	\draw[thick] (0,-1.5) --  node[right]{$Z_2$}+(0,-1.5);
	\draw[thick] (1,0)node[right]{$U$} --  +(0,-3);
	\node at (0,-1.5)[small]{$g$};
	\node at (.5,-3)[big]{$b_2$};
\end{scope}
\node at (-2,-1.5) {$b_1\circ (\id_{Z_1}\otimes f_2)=$};
\node at (3,-1.5) {$=$};
\node at (8,-1.5) {$=$};

\end{tikzpicture}
\]
where
\[
\begin{tikzpicture}
[big/.style={rectangle,rounded corners,fill=white,draw,minimum height=.6cm,minimum width=1.5cm},
small/.style={rectangle,rounded corners,fill=white,draw,minimum height=.6cm,minimum width=.6cm}]
\begin{scope}[shift={(0,0)}]
	\draw[thick] (0,0)node[left]{$Z_1$} --  +(0,-1.5);
	\draw[thick] (1,0) --  node[left]{$X$}+(0,-1.5);
	\draw[thick] (2,0) --  +(0,-2)node[right]{$Z_2$};
	\node at (.5,-1.5)[big]{$b_1$};
	\node at (1.5,0)[big]{$t_2$};
\end{scope}
\node at(-2,-.7) {$g$};
\node at (-1,-.7) {$=$};
\end{tikzpicture}
\]
Similarly, $f_1\circ \wh f_2=[Z_2,(f_1\otimes \id_{Z_2})\circ t_2,b_2]$, and 
\[
\begin{tikzpicture}
[big/.style={rectangle,rounded corners,fill=white,draw,minimum height=.6cm,minimum width=1.5cm},
small/.style={rectangle,rounded corners,fill=white,draw,minimum height=.6cm,minimum width=.6cm}]
\begin{scope}[shift={(1,0)}]
	\draw[thick]  (1,0)--  +(0,-3)node[right]{$Z_2$};
	\draw[thick] (0,0) -- node[left]{$X$}  +(0,-1.5);
	\draw[thick] (0,-1.5) --  +(0,-1.5)node[left]{$Y$};
	\node at (.5,0)[big]{$t_2$};
	\node at (0,-1.5)[small]{$f_1$};
\end{scope}
\begin{scope}[shift={(4,0)}]
	\draw[thick] (0,0) --  +(0,-3)node[left]{$Y$};
	\draw[thick] (1,0) --  node[left]{$Z_1$}+(0,-1.5);
	\draw[thick] (2,0) --  node[right]{$X$}+(0,-1.5);
	\draw[thick] (3,0) --  +(0,-3)node[right]{$Z_2$};
	\node at (.5,0)[big]{$t_1$};
	\node at (1.5,-1.5)[big]{$b_1$};
	\node at (2.5,0)[big]{$t_2$};
\end{scope}
\begin{scope}[shift={(9,0)}]
	\draw[thick] (1,0) --  +(0,-3)node[right]{$Z_2$};
	\draw[thick] (0,0)node[right]{$Z_1$} --  +(0,-1.5);
	\draw[thick] (0,-1.5) --  +(0,-1.5)node[left]{$Y$};
	\node at (1,-1.5)[small]{$g$};
	\node at (.5,0)[big]{$t_1$};
\end{scope}
\node at (-1.5,-1.5) {$(f_1\otimes \id_{Z_2})\circ t_2=$};
\node at (3,-1.5) {$=$};
\node at (8,-1.5) {$=$};
\end{tikzpicture}
\]
This shows that $g$ is an equivalence from $(Z_1,t_1,b_1\circ (\id_{Z_1}\otimes f_2))$ to $(Z_2,(f_1\otimes \id_{Z_2})\circ t_2,b_2)$ and hence these triples represent the {\em same} element of $\wh \sC(U,Y)$ as claimed.
\end{proof}
\subsection{The trace of a thickened morphism}

Our next goal is to show that for a monoidal category $\sC$ with switching isomorphism $s_{X,Z}\colon X\otimes Z\to Z\otimes X$ the map
\begin{equation}\label{eq:whtr}
\wh\tr\colon \wh\sC(X,X)\ra \sC(\one,\one)
\qquad
[Z,t,b]\mapsto\wh\tr(Z,t,b)
\end{equation}
is well-defined. We recall from equation \ref{eq:trace} of the introduction that $\wh\tr(Z,t,b)$ is defined by $\wh\tr(Z,t,b)=b\circ s_{X,Z}\circ t$.  In our pictorial notation, we write it as
\begin{equation}\label{eq:whtr2}
\begin{tikzpicture}
[scale=.7,
big/.style={rectangle,rounded corners,fill=white,draw,minimum height=.3cm,minimum width=1cm},
small/.style={rectangle,rounded corners,fill=white,draw,minimum height=.3cm,minimum width=.3cm}]
\node at (0,0){$\wh\tr(Z,t,b)=$};
\begin{scope}[shift={(2.5,0)}]
\draw[thick, rounded corners](0,1.5)--(0,.5)--(1,-.5)--(1,-1.5);
\draw[thick, rounded corners](1,1.5)--(1,.5)--(0,-.5)--(0,-1.5);
\node[big] at (.5,1.5){$t$};
\node[big] at (.5,-1.5){$b$};
\node at (-.3,.5){$X$};
\node at (1.3,.5){$Z$};
\node at (-.3,-.5){$Z$};
\node at (1.3,-.5){$X$};
\end{scope}
\node at (6,0){where};
\begin{scope}[shift={(7.5,0)}]
\begin{scope}[shift={(0,0)}]
\draw[thick, rounded corners](0,1)node[above]{$X$}--(0,.5)--(1,-.5)--(1,-1)node[below]{$X$};
\draw[thick, rounded corners](1,1)node[above]{$Z$}--(1,.5)--(0,-.5)--(0,-1)node[below]{$Z$};
\end{scope}
\node at (4,0){is shorthand for};
\begin{scope}[shift={(7,0)}]
\draw[thick, rounded corners](0,1)node[above]{$X$}--(0,-1)node[below]{$X$};
\draw[thick, rounded corners](1,1)node[above]{$Z$}--(1,-1)node[below]{$Z$};
\node[big] at (.5,0){$s_{X,Z}$};
\end{scope}
\end{scope}
\end{tikzpicture}
\end{equation}
We note that the naturality of $s_{X,Y}$ is expressed pictorially as
\begin{equation}\label{eq:s_naturality}
\begin{tikzpicture}
[scale=.7,
big/.style={rectangle,rounded corners,fill=white,draw,minimum height=.5cm,minimum width=1cm},
small/.style={rectangle,rounded corners,fill=white,draw,minimum height=.5cm,minimum width=.5cm}]
\begin{scope}[shift={(0,0)}]
\draw[thick, rounded corners](0,2.5)node[above]{$X_1$}--(0,.5)--(1,-.5)--(1,-2.5)node[below]{$X_2$};
\draw[thick, rounded corners](1,2.5)node[above]{$Y_1$}--(1,.5)--(0,-.5)--(0,-2.5)node[below]{$Y_2$};
\node[small] at (0,1.5){$g$};
\node[small] at (1,1.5){$h$};
\end{scope}
\node at (2.5,0){$=$};
\begin{scope}[shift={(4,0)}]
\draw[thick, rounded corners](0,2.5)node[above]{$X_1$}--(0,.5)--(1,-.5)--(1,-2.5)node[below]{$X_2$};
\draw[thick, rounded corners](1,2.5)node[above]{$Y_1$}--(1,.5)--(0,-.5)--(0,-2.5)node[below]{$Y_2$};
\node[small] at (1,-1.5){$g$};
\node[small] at (0,-1.5){$h$};
\end{scope}
\end{tikzpicture}
\end{equation}
for morphisms $g\colon X_1\to X_2$, $h\colon Y_1\to Y_2$. 
\begin{lem}
$\wh\tr(Z,t,b)$ depends only on the equivalence class of $(Z,t,b)$. In particular, the map \eqref{eq:whtr} is well-defined.
\end{lem}

\begin{proof}
Suppose that $g\colon Z\to Z'$ is an equivalence from $(Z',t',b')$  to $(Z,t,b)$ (see Definition \ref{def:thickened}). Then
\[
\begin{tikzpicture}
[scale=.7,
big/.style={rectangle,rounded corners,fill=white,draw,minimum height=.3cm,minimum width=1cm},
small/.style={rectangle,rounded corners,fill=white,draw,minimum height=.3cm,minimum width=.3cm}]
\node at (0,0){$\wh\tr(Z',t',b')=$};
\begin{scope}[shift={(2.5,0)}]
\node at (2,0){$=$};
\draw[thick, rounded corners](0,3)--(0,.5)--(1,-.5)--(1,-3);
\draw[thick, rounded corners](1,3)--(1,.5)--(0,-.5)--(0,-3);
\node[big] at (.5,3){$t'$};
\node[big] at (.5,-3){$b'$};
\node at (-.4,2.25){$X$};
\node at (1.4,2.25){$Z'$};
\end{scope}
\begin{scope}[shift={(5.5,0)}]
\node at (2,0){$=$};
\draw[thick, rounded corners](0,3)--(0,.5)--(1,-.5)--(1,-3);
\draw[thick, rounded corners](1,3)--(1,.5)--(0,-.5)--(0,-3);
\node[big] at (.5,3){$t$};
\node[big] at (.5,-3){$b'$};
\node[small] at (1,1.5){$g$};
\node at (-.4,2.25){$X$};
\node at (1.4,2.25){$Z$};
\node at (1.4,.7){$Z'$};
\end{scope}
\begin{scope}[shift={(8.5,0)}]
\node at (2,0){$=$};
\draw[thick, rounded corners](0,3)--(0,.5)--(1,-.5)--(1,-3);
\draw[thick, rounded corners](1,3)--(1,.5)--(0,-.5)--(0,-3);
\node[big] at (.5,3){$t$};
\node[big] at (.5,-3){$b'$};
\node[small] at (0,-1.5){$g$};
\node at (-.4,2.25){$X$};
\node at (1.4,2.25){$Z$};
\node at (-.4,.-2.25){$Z'$};
\end{scope}
\begin{scope}[shift={(11.5,0)}]
\node at (3,0){$=\wh\tr(Z,t,b)$};
\draw[thick, rounded corners](0,3)--(0,.5)--(1,-.5)--(1,-3);
\draw[thick, rounded corners](1,3)--(1,.5)--(0,-.5)--(0,-3);
\node[big] at (.5,3){$t$};
\node[big] at (.5,-3){$b$};
\node at (-.4,2.25){$X$};
\node at (1.4,2.25){$Z$};
\end{scope}
\end{tikzpicture}
\]
In terms of the coend description, this lemma is actually obvious because $\wh\tr$ is the composition
\[
\wh\sC(X,X)=\int^{Z\in\sC}  \sC(I,X\otimes Z)\times \sC(Z\otimes X,I) \to \int^{Z\in\sC}  \sC(I,Z\otimes X)\times \sC(Z\otimes X,I) \to \sC(I,I)
\]
where the first map is given by switching and the second by composition.
\end{proof}

\section{Traces in various categories}
The goal in this section is to describe thick morphisms in various categories in classical terms and relate our {\em categorical trace} to classical notions of trace. At a more technical level, we describe the map $\Psi\colon \wh \sC(X,Y)\to \sC(X,Y)$ and its image $\sC^{tk}(X,Y)$ in these categories.

In the first subsection this is done for the category of vector spaces (not necessarily finite dimensional). In the second subsection we show that $\Psi$ is a bijection if $X$ is dualizable and hence {\em any} endomorphism of a dualizable object has a well-defined trace. Traces in categories for which all objects are dualizable are well-studied \cite{JS2}. In the third subsection we introduce {\em semi-dualizable} objects and describe the map $\Psi$ in more explicit terms. In a closed monoidal category every object is semi-dualizable. In the following subsection we apply these considerations to the category of Banach spaces. Then we discuss the category of topological vector spaces, and finally the Riemannian bordism category.  

\subsection{The category of vector spaces}
\begin{thm} \label{thm:vect}
Let $\sC$ be the monoidal category $\Vect_\F$ of vector spaces (not necessarily finite dimensional) over a field $\F$ equipped with the usual tensor product and the usual switching isomorphism $s_{X,Y}\colon X\otimes Y\to Y\otimes X$, $x\otimes y\mapsto y\otimes x$.  
\begin{enumerate}
\item A morphism $f\colon X\to Y$ is thick \Iff it has finite rank.
\item The map $\Psi\colon \wh\sC(X,Y)\to \sC(X,Y)$ is injective; in particular, every finite rank endomorphism $f\colon X\to X$ has a well-defined categorical trace $\tr(f)\in\sC(\one,\one)\cong\F$.
\item For a finite rank endomorphism $f$, its categorical trace $\tr(f)$ agrees with its usual (classical) trace which we denote by $\classtr(f)$. 
\end{enumerate}
\end{thm}

For a vector space $X$, let $X^*$ be the dual vector space, and define the morphism
\begin{equation}\label{eq:Phi_vect}
\Phi\colon \sC(\F,Y\otimes X^*)\ra \sC(X,Y)
\end{equation}
by sending $t\colon \F\to Y\otimes X^*$ to the composition 
\[
X\cong \F\otimes X\overset{t\otimes\id_X}\ra
Y\otimes X^*\otimes X\overset{\id_Y\otimes\ev}\ra
Y\otimes \F\cong Y.
\]
We note that the set $\sC(\F,Y\otimes X^*)$ can be identified with $Y\otimes X^*$ via the evaluation map $t\mapsto t(1)$. Using this identification, $\Phi$ maps an elementary tensor $y\otimes g\in Y\otimes X^*$ to the linear map $x\mapsto y g(x)$ which is a rank $\le 1$ operator.  Since finite rank operators are finite sums of rank one operators, we see that the image $\Phi$ consists of the finite  rank operators. The classical trace of the rank one operator $\Phi(y\otimes g)$ is defined to be $g(y)$; by linearity this extends to a well-defined trace for all finite rank operators.

The proof of Theorem \ref{thm:vect} is based on the following lemma that shows that $\Psi$ is equivalent to the map $\Phi$.
\begin{lem}\label{lem:Psi_Phi_vect}
For any vector spaces $X$, $Y$ the map
\begin{equation}\label{eq:alpha_vect}
\alpha\colon \sC(\F,Y\otimes X^*)\ra \wh\sC(X,Y)
\qquad
 t\mapsto [X^*,t,\ev]
\end{equation}
is a bijection and $\Psi\circ \alpha=\Phi$. Here $\ev\colon X^*\otimes X\to \F$ is the evaluation map $g\otimes x\mapsto g(x)$.
\end{lem}

In subsection \ref{subsec:semi-dualizable}, we will construct the map $\Phi$ and prove this lemma in the more general context where $X$ is a semi-dualizable object of a monoidal category. 

\begin{proof}[Proof of Theorem \ref{thm:vect}] 
Statement (1) follows immediately from the Lemma.
To prove part (2), it suffices to show that $\Phi$ is injective which is well known and elementary.

For the proof of part (3) let $f\colon X\to X$ be a thick morphism. We recall that its categorical trace $\tr (f)$ is defined by $\tr(f)=\wh \tr(\wh f)$ for a thickener $\wh f\in \wh \sC(X,X)$. So $\Psi(\wh f)=f$ and this is well-defined thanks to the injectivity of $\Psi$. Using that $\alpha$ is a bijection, there is a unique $ t\in \sC(\F,X\otimes X^*)$ with $\alpha(t)=[X^*,t,\ev]=\wh f$. We note that  $\wh\tr(\alpha(t))=\wh\tr([X^*, t,\ev])\in \sC(\F,\F)=\F$ is the composition
\[
\F\overset{ t}\ra X\otimes X^*\overset {s_{X,X^*}}\ra X^*\otimes X\overset{\ev}\ra\F
\]
In particular, if $t(1)=y\otimes g\in X\otimes X^*$, then 
\[
1\mapsto y\otimes g\mapsto g\otimes y\mapsto g(y)
\]
and hence $\tr(f)=g(y)=\classtr(f)$ is the classical trace of the rank  one operator $\Phi(t)=f$  given by $x\mapsto yg(x)$. Since both the categorical trace of $\alpha(t)$ and the classical trace of $\Phi(t)$ depend linearly on $t$, this finishes the proof.
\end{proof}

\begin{rem} We can replace the monoidal category of vector spaces by the monoidal category $\SVect$ of super vector spaces. The objects of this category are just $\Z/2$\nb-graded vector spaces with the usual tensor product $\otimes$. The grading involution on $X\otimes Y$ is the tensor product $\epsilon_X\otimes \epsilon_Y$ of the grading involutions on $X$ respectively\ $Y$. The switching isomorphism $X\otimes Y\cong Y\otimes X$ is given by $x\otimes y\mapsto (-1)^{|x||y|}y\otimes x$ for homogeneous elements $x\in X$, $y\in Y$ of degree $|x|,|y|\in\Z/2$. Then the statements above and their proof work for $\SVect$ as well, except that the categorical trace of a finite rank endomorphism $T\colon X\to X$ is its {\em super trace} $\str(T):=\classtr(\epsilon_XT)$.
\end{rem}

\subsection{Thick morphisms with semi-dualizable domain}\label{subsec:semi-dualizable}

The goal of this subsection is to generalize Lemma \ref{lem:Psi_Phi} from the category of vector spaces to general monoidal categories $\sC$, provided the object $X\in \sC$ satisfies the following condition.

\begin{defn}\label{def:semi-dualizable} An object $X$ of a monoidal category $\sC$ is {\em semi-dualizable} if the functor $\sC^{op}\to \Set$, $Z\mapsto \sC(Z\otimes X,\one)$ is representable, i.e., if there is an object $X^{\spcheck}\in \sC$ and natural bijections
\begin{equation}\label{eq:semi-dualizable}
\sC(Z,X^{\spcheck})\cong \sC(Z\otimes X,\one).
\end{equation}
By Yoneda's Lemma, the object $X^{\spcheck}$ is unique up to isomorphism. It is usually referred to as the {\em (left) internal hom} and denoted by $\underline C(X,\one)$. 
\end{defn}

To put this definition in context, we recall that a monoidal category $\sC$ is {\em closed} if for any $X\in \sC$ the functor $\sC\to\sC$, $Z\mapsto Z\otimes X$ has a right adjoint; i.e., if there is a functor
\[
\underline{\sC}(X,\quad)\colon \sC\ra \sC
\]
and natural bijections
\begin{equation}\label{eq:adjunction}
\sC(Z,\underline \sC(X,Y))\cong \sC(Z\otimes X,Y)
\qquad Z,X,Y\in\sC
\end{equation}
In particular, in a closed monoidal category {\em every} object $X$ is semi-dualizable with $X^{\spcheck}=\underline\sC(X,I)$. The category $\sC=\Vect_\F$ is an example of a closed monoidal category; for vector spaces $X$, $Y$ the internal hom $\underline \sC(X,Y)$ is the vector space of linear maps from $X$ to $Y$. Hence every vector space is semi-dualizable with semi-dual $X^{\spcheck}=\underline\sC(X,I)=X^*$. 

Other examples of closed monoidal categories is the category of Banach spaces (see subsection \ref{subsec:Ban}) and the category of {\em bornological vector spaces} \cite{Me}.
As also discussed in \cite[p.\ 9]{Me}, the symmetric monoidal category $\TV$ of   topological vector spaces with the projective tensor product is {\em not an example} of a closed monoidal category (no matter which topology on the space of continuous linear maps is used). Still, some topological vector spaces are semi-dualizable (e.g., Banach spaces), and so it seems preferable to state our results for semi-dualizable objects rather than objects in closed monoidal categories.

In Lemma \ref{lem:Psi_Phi} we will generalize a statement about the category $\Vect_\F$ to a statement about a general monoidal category $\sC$. To do so, we need to construct the maps $\Phi$ and $\alpha$ in the context of a general monoidal category. This is 
straightforward by using the {\em same definitions} as above, just making the following replacements:
\begin{enumerate}
\item Replace $\F$, the monoidal unit in $\Vect_\F$, by the monoidal unit $I\in \sC$.
\item Replace $X^*$, the vector space dual to $X$, by $X^{\spcheck}$, the semi-dual of $X\in \sC$. Here we need to assume that $X\in \sC$ is semi-dualizable which is automatic for any object of a closed category like $\Vect_\F$. 
\item For a semi-dualizable object $X\in \sC$, the {\em evaluation map} 
\begin{equation}\label{eq:evaluation}
\ev\colon X^{\spcheck}\otimes X\ra \one
\end{equation}
is by definition the morphism that corresponds to the identity morphism $X^{{\spcheck}}\to X^{\spcheck}$ under the bijection \eqref{eq:semi-dualizable} for $Z=X^{\spcheck}$. It is easy to check that this agrees with the usual evaluation map $X^*\otimes X\to \F$ for $\sC=\Vect_\F$.
\end{enumerate}
The map $\Phi\colon \sC(\one,Y\otimes X^{\spcheck})\ra \sC(X,Y)$ is given by the picture
\begin{equation}\label{eq:Phi}
\begin{tikzpicture}
[scale=.7,
big/.style={rectangle,rounded corners,fill=white,draw,minimum height=.3cm,minimum width=1cm},
small/.style={rectangle,rounded corners,fill=white,draw,minimum height=.3cm,minimum width=.3cm}]
\begin{scope}[shift={(0,0)}]
\begin{scope}[shift={(0,0)}]
\draw[thick](0,-.5)--(0,-1.5)node[below]{$Y^{\quad}$};
\draw[thick](1,-.5)--(1,-1.5)node[below]{$X^{\spcheck}$};
\node at (.5,-.5)[big]{$t$};
\end{scope}
\node at (2,-.5){$\longmapsto$};
\begin{scope}[shift={(3,0)}]
\draw[thick](0,-.5)--(0,-2.5)node[below]{$Y$};
\draw[thick](1,-.5)--(1,-2);
\node at(1.4,-1.25){$X^{\spcheck}$};
\draw[thick](2,-2)--(2,.5)node[above]{$X$};
\node at (.5,-.5)[big]{$t$};
\node at (1.5,-2)[big]{$\ev$};
\end{scope}
\end{scope}
\end{tikzpicture}
\end{equation}

The following Lemma is a generalization of Lemma \ref{lem:Psi_Phi_vect}. 

\begin{lem}\label{lem:Psi_Phi}
Let $X$ be a semi-dualizable object of a monoidal category $\sC$. Then for any object $Y\in \sC$ the map 
\begin{equation}\label{eq:alpha}
\alpha\colon \sC(\one,Y\otimes X^{\spcheck})\ra\wh\sC(X,Y)
\qquad
t\mapsto [X^{\spcheck},t,\ev]
\end{equation}
is a natural bijection which makes the diagram
\[
\xymatrix{
\sC(\one,Y\otimes X^{\spcheck})\ar[dr]_{\Phi}\ar[rr]_\alpha^\cong&&\wh\sC(X,Y)\ar[dl]^\Psi\\
&\sC(X,Y)&
}
\]
commutative. In particular, a morphism $f\colon X\to Y$ is thick \Iff it is in the image of the map $\Phi$ (see Equation \eqref{eq:Phi}). 
\end{lem} 

\begin{proof}
The commutativity of the diagram is clear by comparing the definitions of the maps $\Phi$ (see Equation \eqref{eq:Phi}), $\alpha$ (Equation \eqref{eq:alpha}), and $\Psi$ (Equation \eqref{eq:Psi}). 

To see that $\alpha$ is a bijection we factor it in the following form
\begin{multline}\label{eq:proof_Psi_Phi}
\sC(I,Y\otimes X^{\spcheck})\overset\cong\ra \int^{Z\in \sC}\sC(I,Y\otimes Z)\times\sC(Z,X^{\spcheck})\\
\overset\cong\ra\int^{Z\in \sC}\sC(I,Y\otimes Z)\times\sC(Z\otimes X,\one) 
\end{multline}
Here the first map sends $t\in \sC(I,Y\otimes X^{\spcheck})$ to $[X^{\spcheck},t,\id_{X^{\spcheck}}]$. This map is a bijection by the {\em coend form of the Yoneda Lemma} according to which for any functor $F\colon \sC\to\Set$ the map
\begin{equation}\label{eq:Yoneda}
F(W)\ra \int^{Z\in\sC}F(Z)\times\sC(Z,W)
\qquad
t\mapsto [W,t,\id_W]
\end{equation}
is a bijection (see \cite[Equation 3.72]{Ke}. The second map of Equation \eqref{eq:proof_Psi_Phi} is induced by the bijection $\sC(Z,X^{\spcheck})\cong \sC(Z\otimes X,\one)$. By construction the identity map $\id_{X^{\spcheck}}$ corresponds to the evaluation map via this bijection, and hence the composition of these two maps is $\alpha$. 
\end{proof}

\begin{rem} The referee observed that there is a neat interpretation of $\wh\sC(X,Y)$ in terms of the Yoneda embedding
\[
\sC\ra \sF:=\Fun(\sC^{\op},\Set)
\qquad
X\mapsto \sC(-,X)
\]
The monoidal structure $\otimes$ on $\sC$ induces a monoidal structure $*$ on $\sF$, the {\em convolution tensor product} \cite{Day}, defined by 
\[
(M*N)(S):=\int^{V,W}\sC(S,V\otimes W)\times M(V)\times N(W)
\]
for an object $S\in \sC$. Equipped with the convolution product, the functor category $\sF$ is a {\em closed monoidal category} (see Equation \eqref{eq:adjunction}) with internal left hom $\underline \sF(N,L)\in \sF$ given by
\[
\underline \sF(N,L)(S)=\sF(N(-),L(S\otimes-))
\]
We can regard $\sC$ as a full monoidal subcategory of $\sF$ via the Yoneda embedding. In particular, every object $X\in \sC$ has a left semi-dual $X^{\spcheck}=\underline\sF(X,-)\in \sF$ and hence by the previous lemma we have a bijection $\wh\sF(X,Y)\cong\sF(\one,Y*X^{\spcheck})$. Explicitly, the semi-dual $X^{\spcheck}$ is given by 
\[
X^{\spcheck}(S)=\sF(\sC(-,X),\sC(S\otimes-,\one))=\sC(S\otimes X,\one)
\]

The referee observed that the Yoneda embedding induces a bijection
\begin{equation}\label{eq:thick_via_Yoneda}
\wh\sC(X,Y)\ra\wh\sF(X,Y)\cong\sF(\one,Y*X^{\spcheck}).
\end{equation}
In particular, morphism $f\in\sC(X,Y)$ is thick \Iff its image under the Yoneda embedding is thick. 
To see that the above map is a bijection, we evaluate the right hand side
\begin{align*}
\sF(\one,Y*X^{\spcheck})&=(Y*X^{\spcheck})(\one)=
\int^{V,W}\sC(\one,V\otimes W)\times\sC(V,Y)\times X^{\spcheck}(W)\\
&=\int^{W}\sC(\one,Y\otimes W)\times X^{\spcheck}(W)
=\int^{W}\sC(\one,Y\otimes W)\times \sC(W\otimes X,\one)
\end{align*}
which we recognize as the coend description of $\wh\sC(X,Y)$ (Equation \eqref{eq:coend}). The second equality is a consequence of (the coend form of) the Yoneda lemma (equation \eqref{eq:Yoneda}).
\end{rem}


\subsection{Thick morphisms with dualizable domain}
As mentioned in the introduction, there is a well-known trace for endomophisms of dualizable objects in a monoidal category (see e.g.\ \cite[\S 3]{JSV}). After recalling the definition of dualizable and the construction of that classical trace, we show in Theorem \ref{thm:dualizable} that our categorical trace is a generalization.
 
\begin{defn}\cite[Def.\ 7.1]{JS2}\label{def:dualizable} An object $X$ of a monoidal category $\sC$ is {\em (left) dualizable} if there is an object $X^{\spcheck}\in \sC$ (called the (left) dual of $X$) and 
morphisms $\ev\colon X^{\spcheck}\otimes X\to \one$ (called {\em evaluation map}) and $\coev\colon \one\to X\otimes X^{\spcheck}$ ({\em coevaluation map}) such that the following equations hold.
\begin{equation}\label{eq:dualizing_equations}
\begin{tikzpicture}
	\draw[thick] (0,0) -- +(0,-2) node[left]{$X$};
	\draw[thick] (1,0)  -- node[right]{$X^{\spcheck}$} +(0,-1.5) ;
	\draw[thick] (2,-1.5) -- +(0,2.5) node[right]{$X$} ;
\node[rectangle,rounded corners,fill=white,draw,minimum height=.5cm,minimum width=1.5cm] at (.5,0){$\coev$};
\node[rectangle,rounded corners,fill=white,draw,minimum height=.5cm,minimum width=1.5cm] at (1.5,-1.5){$\ev$};
\node at(3,-.5){$=$};
	\draw[thick]  (4,1) node[left]{$X$} to (4,-2) node[left]{$X$};
\node[rectangle,rounded corners,fill=white,draw,minimum height=.5cm,minimum width=.5cm] at (4,-.5){$\id_X$};
\end{tikzpicture}
\qquad\qquad
\begin{tikzpicture}
	\draw[thick] (2,.5)  -- +(0,-2.5) node[right]{$X^{\spcheck}$};
	\draw[thick] (1,.5)  -- node[right]{$X$} +(0,-1.5) ;
	\draw[thick] (0,-1)  --+(0,2) node[left]{$X^{\spcheck}$} ;
\node[rectangle,rounded corners,fill=white,draw,minimum height=.5cm,minimum width=1.5cm] at (1.5,.5){$\coev$};
\node[rectangle,rounded corners,fill=white,draw,minimum height=.5cm,minimum width=1.5cm] at (.5,-1){$\ev$};
\node at(3,-.5){$=$};
	\draw[thick]  (4,1) node[left]{$X^{\spcheck}$} to (4,-2) node[left]{$X^{\spcheck}$};
\node[rectangle,rounded corners,fill=white,draw,minimum height=.5cm,minimum width=.5cm] at (4,-.5){$\id_{X^{\spcheck}}$};
\end{tikzpicture}
\end{equation}
If $X$ is dualizable with dual $X^{\spcheck}$, then there is a family of bijections 
\begin{equation}\label{eq:dualizable}
\sC(Z,Y\otimes X^{\spcheck})\overset{\cong}\ra \sC(Z\otimes X,Y),
\end{equation}
 natural in $Y,Z\in \sC$, given by 
\[
\begin{tikzpicture}
[scale=.7,
big/.style={rectangle,rounded corners,fill=white,draw,minimum height=.3cm,minimum width=1cm},
small/.style={rectangle,rounded corners,fill=white,draw,minimum height=.3cm,minimum width=.3cm}]
\begin{scope}[shift={(0,0)}]
\draw[thick](0,-.5)--(0,-1.5)node[below]{$Y^{\quad}$};
\draw[thick](1,-.5)--(1,-1.5)node[below]{$X^{\spcheck}$};
\draw[thick](.5,-.5)--(.5,.5)node[above]{$Z$};
\node at (.5,-.5)[big]{$f$};
\end{scope}
\node at (2,-.5){$\longmapsto$};
\begin{scope}[shift={(3,0)}]
\draw[thick](0,-.5)--(0,-2.5)node[below]{$Y$};
\draw[thick](1,-.5)--(1,-2);
\node at(1.4,-1.25){$X^{\spcheck}$};
\draw[thick](2,-2)--(2,.5)node[above]{$X$};
\draw[thick](.5,-.5)--(.5,.5)node[above]{$Z$};
\node at (.5,-.5)[big]{$f$};
\node at (1.5,-2)[big]{$\ev$};
\end{scope}
\node at (7,-.5){$\text{with inverse}$};
\begin{scope}[shift={(9,0)}]
\draw[thick](.5,-2)--(.5,-3)node[below]{$Y$};
\draw[thick](2,-.5)--(2,-3)node[below]{$X^{\spcheck}$};
\draw[thick](1,-2)--(1,-.5);
\node at (.7,-1.25){$X$};
\draw[thick](0,-2)--(0,.5)node[above]{$Z$};
\node at (.5,-2)[big]{$g$};
\node at (1.5,-.5)[big]{$\coev$};
\end{scope}
\node at (12,-.5){$\longmapsfrom$};
\begin{scope}[shift={(13,0)}]
\draw[thick](.5,-.5)--(.5,-1.5)node[below]{$Y$};
\draw[thick](1,-.5)--(1,.5) node[above]{$X$};
\draw[thick](0,-.5)--(0,.5)node[above]{$Z$};
\node at (.5,-.5)[big]{$g$};
\end{scope}
\end{tikzpicture}
\] 
In particular, a dualizable object is semi-dualizable in the sense of Definition \ref{def:semi-dualizable}. 
\end{defn}

\begin{ex}\label{ex:dualizable_vectorspace} A finite dimensional vector space $X$ is a dualizable object in the category $\Vect_\F$:  we take $X^{\spcheck}$ to be the vector space dual to $X$, $\ev$ to be the usual evaluation map, and define 
\[
\coev\colon \F\ra X\otimes X^{\spcheck}\qquad\text{by}\qquad
1\mapsto \sum_i e_i\otimes e^i
\]
where $\{e_i\}$ is a basis of $X$ and $\{e^i\}$ is the dual basis of $X^{\spcheck}$. It is not hard to show that a vector space $X$ is dualizable \Iff it is finite dimensional.
\end{ex}

\begin{defn}\label{def:classtr} Let $\sC$ be a monoidal category with switching isomorphisms. Let $f\colon X\to X$ be an endomorphism of a dualizable object $X\in \sC$. Then the {\em classical trace} $\classtr(f)\in\sC(\one,\one)$ is defined by
\[
\begin{tikzpicture}
[scale=.7,
big/.style={rectangle,rounded corners,fill=white,draw,minimum height=.3cm,minimum width=1cm},
small/.style={rectangle,rounded corners,fill=white,draw,minimum height=.3cm,minimum width=.3cm}]
\node at (0,1){$\classtr(f):=$};
\begin{scope}[shift={(4,0)}]
\draw[thick, rounded corners](0,3.5)--(0,.5)--(1,-.5)--(1,-1.5);
\draw[thick, rounded corners](1,3.5)--(1,.5)--(0,-.5)--(0,-1.5);
\node[big] at (.5,3.5){$coev$};
\node[big] at (.5,-1.5){$ev$};
\node[small] at (0,1.5){$f$};
\node at (0,2.5)[left]{$X$};
\node at (1,2.5)[right]{$X^{\spcheck}$};
\node at (0,-.5)[left]{$X^{\spcheck}$};
\node at (1,-.5)[right]{$X$};
\end{scope}
\end{tikzpicture}
\]
This definition can be found for example in section 3 of \cite{JSV} for balanced monoidal categories (see Definition \ref{def:balanced} for the definition of a balanced monoidal category and how the switching isomorphism is determined by the braiding and the twist of the balanced monoidal category). In fact, the construction in \cite{JSV} is more general: they associate to a morphism $f\colon 
A\otimes X\to B\otimes X$ a trace in $\sC(A,B)$ if $X$ is dualizable; specializing to $A=B=\one$ gives the {\em classical trace} described above. 
\end{defn}

\begin{thm}\label{thm:dualizable} Let $\sC$ be a monoidal category with switching isomorphisms, and let  $X$ be a dualizable object of $\sC$. Then
\begin{enumerate}
\item The map $\Psi\colon \wh \sC(X,Y)\to \sC(X,Y)$ is a bijection. In particular, any morphism with domain $X$ is thick, and any endomorphism $f\colon X\to X$ has a well-defined categorical trace $\tr(f)\in\sC(\one,\one)$.
\item The categorical trace of $f$ is equal to its classical trace $\classtr(f)$ defined above.
\end{enumerate}
\end{thm}

\begin{rem}
Part (1) of the Theorem implies in particular that if $X$ is (left) dualizable, then the identity $\id_X$ is thick. 
The referee observed that the converse holds as well. To see this, assume that $\id_X$ is thick. Then $\id_X$ is in the image of $\Psi\colon \wh\sC(X,X)\to \sC(X,X)$ and hence by Lemma \ref{lem:Psi_Phi} in the image of $\Phi\colon \sC(\one,X\otimes X^{\spcheck})\to \sC(X,X)$. If $\coev\in \sC(\one,X\otimes X^{\spcheck})$ belongs to the preimage $\Phi^{-1}(\id_{X})$, then it is straightforward to check that Equations \eqref{eq:Phi} hold and hence $X$ is dualizable. The first equation holds by construction of $\Phi$; to check the second equation, we apply the bijection \eqref{eq:semi-dualizable} (for $Z=X^{\spcheck}$) to both sides and obtain $\ev$ for both.
\end{rem}

\begin{proof}
To prove part (1), it suffices by Lemma \ref{lem:Psi_Phi} to show that the map $\Phi\colon \sC(\one,Y\otimes X^{\spcheck})\to \sC(X,Y)$ (see Equation \eqref{eq:Phi}) is a bijection. Comparing $\Phi$ with the natural bijection \eqref{eq:dualizable} for dualizable objects $X\in \sC$, we see that this bijection is equal to $\Phi$ in the special case $Z=\one$.

To prove part (2) we recall that by definition 
$\tr(f)=\wh\tr(\wh f)\in \sC(\one,\one)$ for any $\wh f\in \Psi^{-1}(f)\subset \wh \sC(X,X)$ (see Equation \eqref{eq:trace_diagram}). In the situation at hand, $\Psi$ is invertible by part (1), and using the factorization $\Psi=\Phi\circ \alpha^{-1}$ provided by Lemma \ref{lem:Psi_Phi}, we have
\[
\wh f=\Psi^{-1}(f)=\alpha\Phi^{-1}(f)=\alpha((f\otimes\id_{X^{\spcheck}})\circ\coev)=
[X^{\spcheck},(f\otimes\id_{X^{\spcheck}})\circ\coev,\ev]
\]
Here the second equality follows from the explicit form of the inverse of $\Phi$ (which agrees with the map \eqref{eq:dualizable} for $Z=\one$). Comparing $\wh\tr(\wh f)$ (Equation \eqref{eq:whtr2}) and $\classtr(f)$ (Definition \ref{def:classtr}), we see 
\[
\tr(f)=\wh\tr(X^{\spcheck},(f\otimes \id_{X^{\spcheck}})\circ\coev,\ev)=\classtr(f)
\]
\end{proof}
\subsection{The Category $\Ban$ of Banach spaces}\label{subsec:Ban}
Let $X$ be a Banach space and $f\colon X\to X$ a continuous linear map. There are classical conditions ($f$ is {\em nuclear} and $X$ has the {\em approximation property}, see Definition \ref{def:nuclear_Ban}) which guarantee that $f$ has a well-defined (classical) trace which we again denote by $\classtr(f)\in \C$. For example any Hilbert space $H$ has the approximation property and a continuous linear map $f\colon H\to H$ is nuclear \Iff it is trace class. The main result of this subsection is Theorem \ref{thm:Ban} which shows that these classical conditions imply that $f$ has a well-defined categorical trace and that the categorical trace of $f$ agrees with its classical trace. 
Before stating this result, we review the (projective) tensor product of Banach spaces and define the notions {\em nuclear} and {\em approximation property}.

The category $\Ban$ of Banach spaces is a monoidal category whose monoidal structure is given by the {\em projective tensor product}, defined as follows. For  Banach spaces $X$, $Y$, the {\em projective norm} on the algebraic tensor product $X\otimes_{alg} Y$ is given by 
\[
||z||:=\inf\{\sum ||x_i||\cdot ||y_i||\mid z=\sum x_i\otimes y_i\in X\otimes Y\}
\]
where the infimum is taken over all ways of expressing $z\in X\otimes_{alg} Y$ as a finite sum of elementary tensors. Then the {\em projective tensor product} $X\otimes Y$ is defined to be the completion of $X\otimes_{alg} Y$ with respect to the projective norm. 

It is well-known that $\Ban$ is a closed monoidal category (see Equation \eqref{eq:adjunction}). For Banach spaces $X$, $Y$ the internal hom space $\underline {\Ban}(X,Y)$ is the Banach space of continuous linear maps $T\colon X\to Y$ equipped with the operator norm $||T||:=\sup{x\in X,||x||=1}||T(x)||$. In particular, every Banach space has a left semi-dual $X^{\spcheck}=\underline{\Ban}(X,I)$ in the sense of Definition \ref{def:dualizable} which is just the Banach space of continuous linear maps $X\to\C$. The categorically defined evaluation map $\ev\colon X^{\spcheck}\otimes X\to \C$ (see Equation\eqref{eq:evaluation}) agrees with the usual evaluation map defined by $f\otimes x\mapsto \wh f(x)$. The map
\[
Y\otimes X^{\spcheck}\cong \Ban(\C,Y\otimes X^{\spcheck})\overset{\Phi}\ra \Ban(X,Y)
\]
(see Equation \eqref{eq:Phi}) is determined by sending $y\in f\in Y\otimes X^{\spcheck}$ to the map $x\mapsto yf(x)$ (see the discussion after Theorem \ref{thm:vect}).

We note that the morphism set $\Ban(X\otimes Y,Z)$ is in bijective correspondence to the  continuous bilinear maps $X\times Y\to Z$. This bijection is given by sending $g\colon X\otimes Y\to Z$ to the composition $g\circ \chi$, where $\chi\colon X\times Y\to X\otimes Y$ is given by $(x,y)\mapsto x\otimes y$. In particular, if $X^{\spcheck}$ is the Banach space dual to $X$, we have a morphism
\begin{equation}\label{eq:ev}
\ev\colon X^{\spcheck}\otimes X\ra \C
\qquad\text{determined by}\qquad
g\otimes x\mapsto g(x)
\end{equation}
called the {\em evaluation map}.

\begin{defn}\cite[Ch.\ III, \S 7]{Sch}\label{def:nuclear_Ban}  A continuous linear map between Banach spaces is {\em nuclear} if $f$ is in the image of the map 
\[
\Phi\colon Y\otimes X^{\spcheck}\ra \Ban(X,Y)
\qquad\text{determined by}\qquad
y\otimes g\mapsto(x\mapsto y g(x))
\]
A Banach space $X$ has the {\em approximation property} if the identity of $X$ can be approximated by finite rank operators with respect to the compact-open topology.
\end{defn}

\begin{thm}\label{thm:Ban}
Let $X,Y\in \Ban$. 
\begin{enumerate}
\item A morphism $f\colon X\to Y$ is thick \Iff it is  nuclear.
\item If $X$ has the approximation property, it has the trace property.
\item If $X$ has the approximation property and $f\colon X\to X$ is nuclear, then the categorical trace of $f$ agrees with its classical trace. 
\end{enumerate}
\end{thm}

\begin{proof}
This result holds more generally in the category $\TV$  which we will prove in the following subsection (see Theorem \ref{thm:TV}). For the proofs of statements (2) and (3) we refer to that section. For the proof of statement (1) we recall that a morphism $f\colon X\to Y$ in the category $\sC=\Ban$ is thick \Iff if is in the image of the map 
$\Psi\colon \wh \sC(X,Y)\to \sC(X,Y)$ 
and that it is {\em nuclear} if it is in the image of 
$\Phi\colon Y\otimes X^{\spcheck}=\sC(\one ,Y\otimes X^{\spcheck})\to \sC(X,Y)$. 
Hence part (1) is a corollary of Lemma \ref{lem:Psi_Phi} according to which these two maps are equivalent for any semi-dualizable object $X$ of a monoidal category $\sC$. In particular, since $\Ban$ is a closed monoidal category, every Banach space is semi-dualizable. 
\end{proof}

\subsection{A Category of topological vector spaces}\label{subsec:TV}
In this subsection we extend Theorem \ref{thm:Ban} from Banach spaces to the category $\TV$ whose objects are locally convex topological vector spaces which are Hausdorff and complete. We recall that the topology on a vector space $X$ is required to be invariant under translations and dilations. In particular, it determines a uniform structure on $X$ which in turn allows us to speak of Cauchy nets and hence completeness; see \cite[section I.1]{Sch} for details.
The morphisms of $\TV$ are continuous linear maps, and the projective tensor product described below gives $\TV$ the structure of a symmetric monoidal category. It contains the category $\Ban$ of Banach spaces as a full subcategory.

\begin{thm}\label{thm:TV}
Let $X,Y$ be objects in the category $\TV$. 
\begin{enumerate}
\item A morphism $f\colon X\to Y$ is thick \Iff it is  nuclear.
\item If $X$ has the approximation property, then it has the trace property.
\item If $X$ has the approximation property, and $f\colon X\to X$ is nuclear, then the categorical trace of $f$ agrees with its classical trace. 
\end{enumerate}
\end{thm}

Before proving this theorem, we define nuclear morphisms in $\TV$ and the approximation property. Then we'll recall the classical trace of a nuclear endomorphism of a topological vector space with the approximation property as well as the projective tensor product. 

\begin{defn}\label{def:nuclear_TV}
A continuous linear map $f\in\TV(X,Y)$ is {\em nuclear} if it factors in the form
\begin{equation}\label{eq:nuclear_factorization}
X\overset{p}\ra X_0\overset{ f_0}\ra Y_0\overset{j}\ra Y,
\end{equation}
where $f_0$ is a nuclear map between Banach spaces (see Definition \ref{def:nuclear_Ban}) and $p$,  $j$ are continuous linear maps.
\end{defn}

 The definition of nuclearity in Schaefer's book (\cite[p.\ 98]{Sch}) is phrased differently. We give his more technical definition at the end of this section and show that a continuous linear map is nuclear in his sense \Iff it is nuclear in the sense of the above definition. 
 
 \medskip
 
\noindent{\bf Approximation property.} An object $X\in\TV$ has the {\em approximation property} if the identity of $X$ is in the closure of the subspace of finite rank operators with respect to the compact-open topology \cite[Chapter III, section 9]{Sch} (our completeness assumption for topological vector spaces implies that uniform convergence on compact subsets is the same as uniform convergence on pre-compact subsets). 

\medskip

\noindent{\bf The classical trace for nuclear endomorphisms.} Let $f$ be a nuclear endomorphism of $X\in \TV$, and let $I_\nu\colon X\to X$ be a net of finite rank morphisms which converges to the identity on $X$ in the compact-open topology. Then $f\circ I_\nu$ is a finite rank operator which has a classical trace $\classtr(f\circ I_\nu)$. It can be proved that the limit $\lim_{\nu}\classtr (f\circ I_\nu)$ exists \cite[Proof of Theorem 1]{Li} and is independent of the choice of the net $I_\nu$ (this also follows from our proof of Theorem \ref{thm:TV}). The classical trace of $f$ is {\em defined} by
\[
\classtr(f):=\lim_\nu\classtr(f\circ I_\nu).
\]

\noindent{\bf Projective tensor product.} The projective tensor product of Banach spaces defined in the previous section extends to topological vector spaces as follows. For $X,Y\in \TV$ the {\em projective topology} on the algebraic tensor product $X\otimes_{alg}Y$ is the finest locally convex topology such that the canonical bilinear map
\[
\chi\colon X\times Y\ra X\otimes_{alg} Y
\qquad
(x,y)\mapsto x\otimes y
\]
is continuous \cite[p.\ 93]{Sch}. The {\em projective tensor product} $X\otimes Y$ is the topological vector space obtained as the completion of $X\otimes_{alg}Y$ with respect to the projective topology. It can be shown that it is locally convex and Hausdorff and that the morphisms $\TV(X\otimes Y,Z)$ are in bijective correspondence to continuous bilinear maps $X\times Y\to Z$; this bijection is given by sending $f\colon X\otimes Y\to Z$ to $f\circ\chi$ \cite[Chapter III, section 6.2]{Sch}. 

\medskip

\noindent{\bf Semi-norms.} For checking the convergence of a sequence or continuity of a map between locally convex topological vector spaces, it is convenient to work with semi-norms.  For $X\in\TV$ and $U\subset X$ a convex circled $0$\nb-neighborhood  ($U$ is {\em circled} if $\lambda U\subset U$ for every $\lambda\in \C$ with $|\lambda|\le 1$) one gets a semi-norm
\begin{equation}\label{eq:gauge}
||x||_U:=\inf\{\lambda\in \R_{>0}\mid x\in \lambda U\}
\end{equation}
on $X$. Conversely, a collection of semi-norms determines a topology, namely the coarsest locally convex topology such that the given semi-norms are continuous maps. For example, if $X$ is a Banach space with norm $||\quad||$, we obtain the usual topology on $X$. As another example, the projective topology on the algebraic tensor product $X\otimes_{alg}Y$ is the topology determined by the family of semi-norms $||\quad||_{U,V}$ parametrized by convex circled $0$\nb-neighborhoods $U\subset X$, $V\subset Y$ defined by 
\[
||z||_{U,V}:=\inf\{\sum_{i=1}^n ||x_i||_U ||y_i||_V \mid z=\sum_{i=1}^n x_i\otimes y_i\},
\]
where the infimum is taken over all ways of writing $z\in X\otimes_{alg}Y$ as a finite sum of elementary tensors. It follows from this description that the projective tensor product defined above is compatible with the projective tensor product of Banach spaces defined earlier (see \cite[Chapter III, section 6.3]{Sch}).

Our next goal is the proof of Theorem \ref{thm:TV}, for which we will use  the following lemma.

\begin{lem}\label{lem:factorization}
\begin{enumerate}
\item Any morphism  $t\colon \C\to Y\otimes Z$ in $\TV$ factors in the form
\[
\begin{tikzpicture}
[scale=.7,
big/.style={rectangle,rounded corners,fill=white,draw,minimum height=.5cm,minimum width=1cm},
small/.style={rectangle,rounded corners,fill=white,draw,minimum height=.5cm,minimum width=.5cm}]
\begin{scope}[shift={(1,0)}]
\draw[thick, rounded corners](0,3)--(0,0)node[below]{$Y$};
\draw[thick, rounded corners](1,3)--(1,0)node[below]{$Z$};
\node[big] at (.5,3){$t$};
\end{scope}
\node at(4,1.5){$=$};
\begin{scope}[shift={(6,0)}]
\draw[thick, rounded corners](0,3)--(0,0)node[below]{$Y$};
\draw[thick, rounded corners](1,3)--(1,0)node[below]{$Z$};
\node[big] at (.5,3){$t_0$};
\node[small] at (0,1){$j$};
\node at (-.5,2){$Y_0$};
\end{scope}
\end{tikzpicture}
\]
where $Y_0$ is a Banach space.
\item Any morphism   $b\colon Z\otimes X\to \C$ factors in the form
\[
\begin{tikzpicture}
[scale=.7,
big/.style={rectangle,rounded corners,fill=white,draw,minimum height=.5cm,minimum width=1cm},
small/.style={rectangle,rounded corners,fill=white,draw,minimum height=.5cm,minimum width=.5cm}]
\begin{scope}[shift={(1,0)}]
\draw[thick, rounded corners](0,3)node[above]{$Z$}--(0,0);
\draw[thick, rounded corners](1,3)node[above]{$X$}--(1,0);
\node[big] at (.5,0){$b$};
\end{scope}
\node at(4,1.5){$=$};
\begin{scope}[shift={(6,0)}]
\draw[thick, rounded corners](0,3)node[above]{$Z$}--(0,0);
\draw[thick, rounded corners](1,3)node[above]{$X$}--(1,0);
\node[big] at (.5,0){$ b_0$};
\node[small] at (1,2){$p$};
\node at (1.5,1){$X_0$};
\end{scope}
\end{tikzpicture}
\]
where $X_0$ is a Banach space.
\end{enumerate}
\end{lem}

For the proof of this lemma, we will need the following two ways to construct Banach spaces from a topological vector space $X$:
\begin{enumerate}
\item Let $U$ be a convex, circled neighborhood of $0\in X$. Let $X_U$ be the Banach space obtained from $X$ by quotiening out the null space of the semi-norm $||\quad||_U$ and by completing the resulting normed vector space. Let  $p_U\colon X \to X_U$ be the evident map. 
\item Let $B$ be a convex, circled bounded subset of $X$. We recall that $B$ is {\em bounded} if for each neighborhood $U$ of $0\in X$ there is some $\lambda\in \C$ such that $B\subset\lambda U$. Let $X_B$ be the vector space $X_B:=\bigcup_{n=1}^\infty nB$ equipped with the norm $||x||_B:=\inf\{\lambda\in \R_{>0}\mid x\in \lambda B\}$. If $B$ is closed in $X$, then $X_B$ is complete (by our assumption that $X$ is complete), and hence $X_B$ is a Banach space (\cite[Ch.\ III, \S 7; p.\ 97]{Sch}). The inclusion map $j_B\colon X_B\to X$ is continuous thanks to the assumption that $B$ is bounded.
\end{enumerate}

\begin{proof}[Proof of Lemma~\ref{lem:factorization}]
To prove part (1) we use the fact (see e.g.\ Theorem 6.4 in Chapter III of \cite{Sch}) that any element of the completed projective tensor product $Y\otimes Z$, in particular the element $t(1)$, can be written in the form 
\begin{equation}\label{eq:tensorproductelement}
t(1)=\sum_{i=1}^\infty\lambda_iy_i\otimes z_i
\quad\text{with}\quad
y_i\to 0\in Y,\, z_i\to 0\in Z,\, \sum |\lambda_i|< \infty
\end{equation}
Let $B':=\{y_i\mid i=1,2\dots\}\cup\{0\}$, and let $B$ be the closure of the {\em convex, circled hull} of $B'$ (the convex circled hull of $B'$ is the intersection of all convex circled subsets of $W$ containing $B'$). We note that $B'$ is bounded, hence its convex, circled hull is bounded, and hence $B$ is bounded.  

We define $j\colon Y_0\to Y$ to be the map  $j_B\colon Y_B\to Y$. To finish the proof of part (1), it suffices to show that $t(1)$ is in the image of the inclusion map $j_B\otimes \id_Z\colon Y_B\otimes Z\into Y\otimes Z$. It is clear that each partial sum $\sum_{i=1}^n\lambda_iy_i\otimes z_i$ belongs to the algebraic tensor product $Y_B\otimes_{alg}Z$, and hence we need to show that the sequence of partial sums is a Cauchy sequence with respect to the semi-norms $||\quad||_{B,V}$ on $Y_B\otimes_{alg} Z$ that define the projective topology (here $V$ runs through all convex circled $0$\nb-neighborhoods $V\subset Z$). Since $y_i\in B$, it follows $||y_i||_B\le 1$ and hence we have the estimate 
\[
||\sum_{i=1}^n\lambda_iy_i\otimes z_i||_{B,V}
\le \sum_{i=1}^n|\lambda_i|\,||y_i||_B||z_i||_V
\le \sum_{i=1}^n|\lambda_i|\,||z_i||_V
\]
Since $z_i\to 0$, we have $||z_i||_V\le 1$ for all but finitely many $i$, and this implies that the partial sums form a Cauchy sequence.

To prove part (2) we recall that the morphism $b\colon Z\otimes X\to \C$ corresponds to a continuous bilinear map $b'\colon Z\times X\to \C$. The continuity of $b'$ implies that there are convex, circled $0$\nb-neighborhoods $V\subset Z$, $U\subset X$ such that $|b'(z,x)|<1$ for $z\in V$, $x\in U$. It follows that $|b'(z,x)|\le ||z||_V||x||_U$ for all $z\in Z$, $x\in X$. Hence $b$ extends to a continuous bilinear map
\[
\wt b'\colon Z\times X_U\ra \C
\]
for the completion $X_U$ of $X$. This corresponds to the desired morphism $b_0\colon Z\otimes X_U\to \C$; the property $b=b_0\circ (\id_Z\otimes p_U)$ is clear by construction. Defining the map $p\colon X\to X_0$ to be $p_U\colon X\to X_U$, we obtained the desired factorization of $b$.
\end{proof}

\begin{rem}\label{rem:set}
The proofs above and below imply that for fixed objects $X,Y\in\TV$, the thickened morphisms $\wh\TV(X,Y)$ actually form a {\em set}. 
Any triple $(Z,t,b)$ can be factored into Banach spaces $X_0, Y_0, Z_0$ as explained in the 3 pictures below. By the argument above, we may actually choose $X_0=X_U$ where $U$ runs over certain subsets of $X$. Since $X$ is fixed, it follows that the arising Banach spaces $X_U$ range over a certain set. Finally, by Lemma~\ref{lem:Psi_Phi}, the Banach space $Z_0$ may be replaced by $X_U^{\spcheck}$ without changing the equivalence class of the triple. Therefore, the given triple $(Z,t,b)$ is equivalent to a triple of the form $(X_U^{\spcheck},t',b')$. Since the collection of objects $X_U^{\spcheck}\in\TV$ forms a set, we see that $\wh\TV(X,Y)$ is a set as well. 

In this paper, we have not addressed the issue whether $\wh\sC(X,Y)$ is a set because this problem does not arise in the examples we discuss: The argument above for $\sC=\TV$ is the hardest one, in all other examples we actually identify $\wh\sC(X,Y)$ with some very concrete set. 

This problem is similar to the fact that presheaves on a given category do not always form a category (because natural transformations do not always form a set). So we are following the tradition of treating this problem only if forced to.
\end{rem}

\begin{proof}[Proof of Theorem \ref{thm:TV}]
The factorization \eqref{eq:nuclear_factorization} shows that a nuclear morphism $f\colon X\to Y$ factors through a nuclear map $f_0\colon X_0\to Y_0$ of Banach spaces. Then $f_0$ is thick by Theorem \ref{thm:Ban} and hence $f$ is thick, since pre- or post-composition of a thick morphism with an any morphism is thick. 
To prove the converse, assume that $f$ is thick, i.e., that it can be factored in the form 
\[
f=(\id_Y\otimes b)(t\otimes \id_X)
\qquad
t\colon \one\to Y\otimes Z\quad b\colon Z\otimes X\to \one.
\]
Then using Lemma \ref{lem:factorization} to factorize $t$ and $b$, we see that $f$ can be further factored in the form
\[
\begin{tikzpicture}
[scale=.7,
big/.style={rectangle,rounded corners,fill=white,draw,minimum height=.6cm,minimum width=2cm},
small/.style={rectangle,rounded corners,fill=white,draw,minimum height=.6cm,minimum width=.6cm}]
\node at(-1,2){$f=$};
\begin{scope}[shift={(1.5,0)}]
\draw[thick, rounded corners](0,3)--(0,0)node[below]{$Y$};
\draw[thick, rounded corners](2,3)--(2,1);
\draw[thick, rounded corners](4,4)node[above]{$X$}--(4,1);
\node[big] at (1,3){$t_0$};
\node[big] at (3,1){$b_0$};
\node[small] at (0,1){$j$};
\node[small] at (4,3){$p$};
\node at (-.5,2){$Y_0$};
\node at (4.5,2){$X_0$};
\node at (1.5,2){$Z$};
\end{scope}
\end{tikzpicture}
\]
This implies that $f$ has the desired factorization $X\overset{p}\ra X_0\overset{f_0}\ra Y_0\overset{j}\ra Y$, where
\[
\begin{tikzpicture}
[scale=.7,
big/.style={rectangle,rounded corners,fill=white,draw,minimum height=.6cm,minimum width=2cm},
small/.style={rectangle,rounded corners,fill=white,draw,minimum height=.6cm,minimum width=.6cm}]
\node at(-1,2){$f_0=$};
\begin{scope}[shift={(1.5,0)}]
\draw[thick, rounded corners](0,3)--(0,.5)node[below]{$Y_0$};
\draw[thick, rounded corners](2,3)--(2,1);
\draw[thick, rounded corners](4,3.5)node[above]{$X_0$}--(4,1);
\node[big] at (1,3){$t_0$};
\node[big] at (3,1){$b_0$};
\node at (1.5,2){$Z$};
\end{scope}
\end{tikzpicture}
\]
It remains to show that $f_0$ is a nuclear map between Banach spaces. In the category of Banach spaces a morphisms is nuclear \Iff it is thick by Theorem \ref{thm:Ban}. At first glance, it seems that the factorization of $f_0$ above shows that $f_0$ is thick. However, on second thought one realizes that we need to replace  $Z$ by a {\em Banach space} to make that argument. This can be done by using again our Lemma \ref{lem:factorization} to factorize $t_0$ and hence $f_0$ further in the form
\[
\begin{tikzpicture}
[scale=.7,
big/.style={rectangle,rounded corners,fill=white,draw,minimum height=.6cm,minimum width=2cm},
small/.style={rectangle,rounded corners,fill=white,draw,minimum height=.6cm,minimum width=.6cm}]
\node at(-1,2){$f_0=$};
\begin{scope}[shift={(1.5,0)}]
\draw[thick, rounded corners](0,3)--(0,-2.5)node[below]{$Y_0$};
\draw[thick, rounded corners](2,3)--(2,-1);
\draw[thick, rounded corners](4,3.5)node[above]{$X_0$}--(4,-1);
\node[big] at (1,3){$t'_0$};
\node[big] at (3,-1){$b_0$};
\node[small] at (2,1){$j'$};
\node at (1.5,2){$Z_0$};
\node at (1.5,0){$Z$};
\end{scope}
\end{tikzpicture}
\]
This shows that $f_0$ is a thick morphism in the category $\Ban$ and hence nuclear by Theorem \ref{thm:Ban}. 
The key for the proof of parts (2) and (3) of Theorem~\ref{thm:TV} will be the following lemma.

\begin{lem} \label{lem:continuity} For any $\wh f\in \wh\TV(Y,X)$ the  map 
\begin{equation}\label{eq:continuity}
\TV(X,Y)\ra \TV(\one,\one)=\C
\qquad
g\mapsto \wh\tr(\wh f\circ g)
\end{equation}
is continuous with respect to the compact open topology on $\TV(X,Y)$. 
\end{lem}

To prove part (2) of Theorem \ref{thm:TV}, assume that $X$ has the approximation property and let $I_\nu\colon X\to X$ be a net of finite rank operators converging to the identity of $X$ in the compact open topology. Then by the lemma,   for any  $\wh f\in \wh\TV(X,X)$, the net 
\[
\wh\tr(f\circ \wh I_\nu)=\wh\tr(\wh f\circ I_\nu)
\qquad\text{converges to}\qquad
\wh\tr(\wh f\circ \id_X)=\wh\tr(\wh f). 
\]
Here $\wh I_\nu\in \wh\TV(X,X)$ are thickeners of $I_\nu$. They exist since every finite rank morphisms is nuclear and hence thick by part (1). This implies the {\em trace condition} for $X$, since $\wh\tr(\wh f)=\lim_\nu \wh \tr(f\circ\wh I_\nu)$ depends only on $f$.

To prove part (3) let $f\in\sC(X,X)$ be a nuclear endomorphism of $X\in\TV$ and let $I_\nu$ be a net of finite rank operators converging to the identity of $X$ in the compact open topology. By part (1) $f$ is thick, i.e., there is a thickener $\wh f\in \sC(X,X)$. Then as discussed above, we have
\[
\classtr(f)=\lim_{\nu}\classtr(f\circ I_\nu)
\qquad\text{and}\qquad
\tr(f)=\wh \tr(\wh f)=\lim_\nu \wh \tr(f\circ \wh I_\nu)=\lim_\nu \tr( {f\circ  I_\nu})
\]
For the finite rank operator $f\circ I_\nu$ its classical trace $\classtr(f\circ I_\nu)$ and its categorical trace $ \tr( {f\circ  I_\nu})$ agree by part (3) of Theorem \ref{thm:vect}. 
\end{proof}

\begin{proof}[Proof of Lemma \ref{lem:continuity}]
Let $\wh f=[Z,t,b]\in \wh \TV(Y,X)$. Then $\wh f\circ g=[Z,t,b\circ (\id_Z\otimes g)]$ and hence $\wh\tr(\wh f\circ g)$ is given by the composition
\begin{equation}\label{eq:composition}
\C\overset{t}\ra X\otimes Z\overset{s_{X,Z}}\ra
Z\otimes X\overset{\id_Z\otimes g}\ra
Z\otimes Y\overset{b}\ra\C
\end{equation}
As in Equation \eqref{eq:tensorproductelement} we write $t(1)\in X\otimes Z$ in the form 
\[
t(1)=\sum_{i=1}^\infty\lambda_ix_i\otimes z_i
\quad\text{with}\quad
x_i\to 0\in X,\, z_i\to 0\in Z,\, \sum_{i=1}^\infty |\lambda_i|< \infty
\]
This implies
\[
\wh\tr(\wh f\circ g)=\sum_{i=1}^\infty\lambda_ib'(z_i,g(x_i)),
\]
where $b'\colon Z\times Y\to \C$ is the continuous bilinear map corresponding to $b\colon Z\otimes Y\to \C$. 

To show that the map $g\mapsto \wh\tr(\wh f\circ g)$ is continuous, we will construct for given $\epsilon>0$ a compact subset $K\subset X$ and an open subset $U\subset Y$ such that for
\[
g\in\scr U_{K,U}:=\{g\in \TV(X,Y)\mid g(K)\subset U\}\subset\TV(X,Y)
\]
we have $|\wh\tr(\wh f\circ g)|<\epsilon$. To construct $U$, we note that
the continuity of $b'$ implies that there are $0$\nb-neighborhoods $U\subset Y$, $V\subset Z$ such that $y\in U$, $z\in V$ implies $| b'(z,y)|<\epsilon$. Without loss of generality we can assume $z_i\in V$ for all $i$ (by replacing $z_i$ by $cz_i$ and $x_i$ by $x_i/c$ for a sufficiently small number $c$) and $\sum|\lambda_i|=1$ (by replacing $\lambda_i$ by $\lambda_i/s$ and $x_i$ by $sx_i$ for $s=\sum|\lambda_i|$). We define $K:=\{x_i\mid i\in \N\}\cup \{0\}\subset X$. Then for $g\in \scr U_{K,U}$ we have
\[
|\wh \tr(\wh f\circ g)|=|\sum \lambda_i  b'(z_i,g(x_i)|
\le \sum |\lambda_i| |b'(z_i,g(x_i)|
<(\sum |\lambda_i|)\epsilon=\epsilon.
\]
\end{proof}

Finally we compare our definition of a nuclear map between topological vector spaces (see Definition \ref{def:nuclear_Ban}) with the more classical definition which can be found e.g.\ in \cite[p.\ 98]{Sch}. A continuous linear map $f\in\TV(X,Y)$ is {\em nuclear in the classical sense} if there is a convex circled $0$\nb-neighborhood $U\subset X$, a closed, convex, circled, bounded subset $B\subset Y$ such that $f(U)\subset B$ and the induced map of Banach spaces $X_U\to Y_B$ is nuclear. 

\begin{lem} A morphism in $\TV$ is nuclear \Iff it is nuclear in the classical sense.
\end{lem}

\begin{proof}
If $f\colon X\to Y$ is nuclear in the classical sense, it factors in the form
\[
X\overset{p_U}\ra X_U\overset{ f_0}\ra Y_B\overset{j_B}\ra Y,
\]
where $f_0$ is nuclear. Hence $f$ is nuclear.

Conversely, let us assume that $f$ is nuclear, i.e., that it factors in the form
\[
X\overset{p}\ra X_0\overset{ f_0}\ra Y_0\overset{j}\ra Y,
\]
where $f_0$ is a nuclear map between Banach spaces. To show that $f$ is nuclear in the classical sense, we will construct a convex circled $0$\nb-neighborhood $U\subset X$, and a closed, convex, circled, bounded subset $B\subset Y$ such that $f(U)\subset B$ and we have a commutative diagram
\[
\xymatrix{
X\ar[r]^p\ar[rd]_{p_U}&X_0\ar[r]^{f_0}&Y_0\ar[d]\ar[r]^j&Y\\
&X_U\ar[u]\ar[r]_{f'}&Y_B\ar[ru]_{j_B}&
}
\]
Then the induced map $f'$ factors through the nuclear map $f_0$, hence $f'$ is nuclear and $f$ is nuclear in the classical sense.

We define $U:=p^{-1}(\mathring B_\delta)$, where $\mathring B_\delta\subset X_0$ is the open ball of radius $\delta$ around the origin in the Banach space $X_0$, and $\delta>0$ is chosen such that $\mathring B_\delta\subset f_0^{-1}(\mathring B_1)$. The continuity of $p$ implies that $U$ is open. Moreover, $p(U)\subset \mathring B_\delta\=\delta\mathring B_1$ implies $||p(x)||<\delta$ for $x\in U$ which in turn implies the estimate $||p(x)||<\delta||x||_U$ for all $x\in X$. It follows that the map $p$ factors through $p_U$. 

We define $B:=j(B_1)\subset Y$, where $B_1$ is the closed unit ball in $Y_0$. This is a bounded subset of $Y$, since for any open subset $U\subset Y$ the preimage $j^{-1}(U)$ is an open $0$\nb-neighborhood of $Y_0$ and hence there is some $\epsilon >0$ such that $ B_\epsilon\subset j^{-1}(U)$. Then $\epsilon j(B_1)=j(B_\epsilon)$ is a subset of $U$. This implies $j(B_1)\subset \frac 1\epsilon U$ and hence $B$ is a bounded subset of $Y$ (it is clear that $B$ is closed, convex and circled). By construction of $B$ we have the inequality $||j(y)||_B\le ||y||$ which implies that $j$ factors through $j_B$. Also by construction, $f(U)$ is contained in $B$ and hence $f$ induces a morphism $f'\colon X_U\to Y_B$. It follows that the outer edges of the diagram form a commutative square. The fact that $j_B$ is a monomorphism and that the image of $p_U$ is dense in $X_U$ then imply that the middle square of the diagram above is commutative. 
\end{proof}
\subsection{The Riemannian bordism category}
We recall from section \ref{sec:motivation} that the objects of the $d$\nb-dimensional Riemannian bordism category $\RBord{d}$ are closed $(d-1)$\nb-dimensional Riemannian manifolds. Given $X,Y\in \RBord{d}$,   the set $\RBord{d}(X,Y)$ of morphisms from $X$ to $Y$ is the disjoint union of the set of isometries from $X$ to $Y$ and the set of Riemannian bordisms from $X$ to $Y$ (modulo isometry relative boundary), except that for $X=Y=\emptyset$, the identity isometry equals the empty  bordism. 

\begin{thm}\label{thm:bord}
\begin{enumerate}
\item A morphism in $\RBord{d}$ is {\em thick} \Iff it is a bordism.
\item For any $X,Y\in \RBord{d}$ the map 
\[
\Psi\colon \wh {\RBord{d}}(X,Y)\ra \RBord{d}(X,Y)
\]
 is injective. In particular, every object $X\in \RBord{d}$ has the trace property and every bordism $\Sigma$ from $X$ to $X$ has a well-defined trace $\tr(\Sigma)\in \RBord{d}(\emptyset,\emptyset)$.
 \item If $\Sigma$ is a Riemannian bordism from $X$ to $X$, then $\tr(\Sigma)=\Sigma_{gl}$, the closed Riemannian manifold obtained by gluing the two copies of $X$ in the boundary of $\Sigma$.
\end{enumerate}
\end{thm}

\begin{proof}
To prove part (1) suppose that $f\colon X \to Y$ is a thick morphism, i.e., it can be factored in the form 
\begin{equation}\label{eq:factorization2}
\xymatrix{
X\cong \emptyset\amalg X\ar[rr]^{t\,\amalg \,\id_{X}}&&Y\amalg Z\amalg X\ar[rr]^{\id_{Y}\,\amalg \,b}&&Y\amalg\emptyset\cong Y
}
\end{equation}
We note that the morphisms $t\colon \emptyset\to Y\amalg Z$ and $b\colon Z\amalg X\to\emptyset$ must both be {\em bordisms} (the only case where say $t$ could possibly be an isometry is $Y=Z=\emptyset$; however that isometry is the same morphism as the empty bordism). Hence the composition $f$ is a bordism. 

Conversely, assume that $\Sigma$ is a bordism from $X$ to $Y$. Then $\Sigma$ can be decomposed as in the following picture:
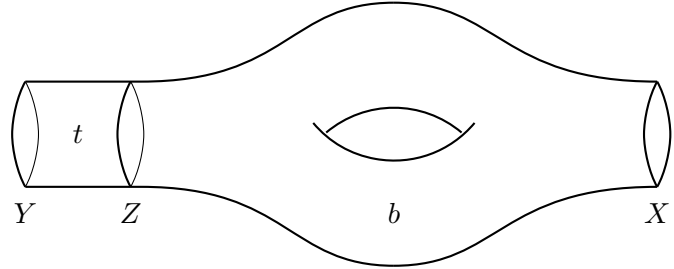
\[
\begin{tikzpicture}[scale=.7]
\draw[thick,rounded corners](-7,1)--(-5,1).. controls +(right:3cm) and +(left:2cm)..(0,2.5);
\draw[thick,rounded corners](0,2.5).. controls +(right:2cm) and +(left:3cm)..(5,1);
\draw[thick,rounded corners](-7,-1)--(-5,-1).. controls +(right:3cm) and +(left:2cm)..(0,-2.5);
\draw[thick,rounded corners](0,-2.5).. controls +(right:2cm) and +(left:3cm)..(5,-1);
\begin{scope}[shift={(-7,0)}]
\draw[thick](-0,1) .. controls +(-120:.3cm) and +(90:.3cm).. (-0.25,0).. controls +(-90:.3cm) and +(120:.3cm).. (0,-1);
\draw[thin](0,1) .. controls +(-60:.3cm) and +(90:.3cm).. (.25,0).. controls +(-90:.3cm) and +(60:.3cm).. (0,-1);
\end{scope}
\begin{scope}[shift={(-5,0)}]
\draw[thick](-0,1) .. controls +(-120:.3cm) and +(90:.3cm).. (-0.25,0).. controls +(-90:.3cm) and +(120:.3cm).. (0,-1);
\draw[thin](0,1) .. controls +(-60:.3cm) and +(90:.3cm).. (.25,0).. controls +(-90:.3cm) and +(60:.3cm).. (0,-1);
\end{scope}
\begin{scope}[shift={(5,0)}]
\draw[thick](-0,1) .. controls +(-120:.3cm) and +(90:.3cm).. (-0.25,0).. controls +(-90:.3cm) and +(120:.3cm).. (0,-1);
\draw[thick](0,1) .. controls +(-60:.3cm) and +(90:.3cm).. (.25,0).. controls +(-90:.3cm) and +(60:.3cm).. (0,-1);
\end{scope}
\draw[thick]($(0,-1.5)+(50:2)$) arc (50:130:2);
\draw[thick]($(0,1.5)+(-40:2)$) arc (-40:-140:2);
\node at (-5,-1.5){$Z$};
\node at (5,-1.5){$X$};
\node at (-7,-1.5){$Y$};
\node at (0,-1.5){$b$};
\node at (-6,0){$t$};
\end{tikzpicture}
\]
Here $t=Y\times [0,\epsilon]\subset \Sigma$, $Z=Y\times \{\epsilon\}$ and $b=\Sigma\setminus (Y\times [0,\epsilon))$ are bordisms, where $\epsilon>0$ is chosen suitably so that $Y\subset \Sigma$ has a neighborhood isometric to $Y\times [0,2\epsilon)$ equipped with the product metric. Regarding $t$  as a Riemannian bordism from $\emptyset$ to $Y\amalg Z$, and similarly $b$ as a Riemannian bordism  from $Z\amalg X$ to $\emptyset$,  it is clear from the construction that the composition \eqref{eq:factorization2} is $\Sigma$.

To show that $\Psi$ is injective, let $[Z',t',b']\in \wh{\RBord{d}}(X,Y)$, and let $\Sigma=\Psi([Z',t',b'])\in \RBord{d}$ be the composition \eqref{eq:factorization}. In other words, we have a decomposition of the bordism $\Sigma$ into two pieces $t'$, $b'$ which intersect along $Z'$. Now let $(Z,t,b)$ be the triple constructed in the proof of part (1) above. By choosing $\epsilon$ small enough, we can assume that $Z=Y\times \epsilon\subset \Sigma$ is in the interior of the bordism $t'$, and we obtain a decomposition of the bordism $\Sigma$ as shown in the picture below.
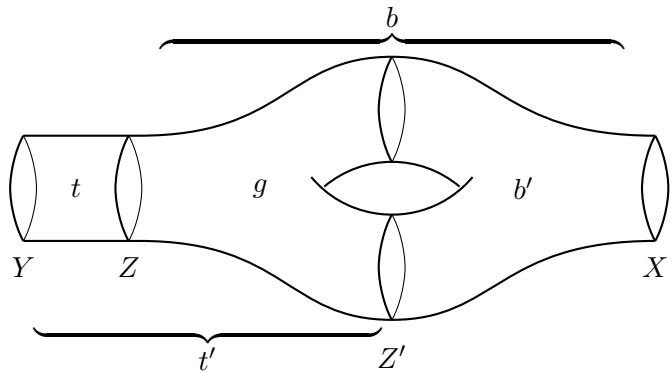
\[
\begin{tikzpicture}[scale=.7]
\draw[thick,rounded corners](-7,1)--(-5,1).. controls +(right:3cm) and +(left:2cm)..(0,2.5);
\draw[thick,rounded corners](0,2.5).. controls +(right:2cm) and +(left:3cm)..(5,1);
\draw[thick,rounded corners](-7,-1)--(-5,-1).. controls +(right:3cm) and +(left:2cm)..(0,-2.5);
\draw[thick,rounded corners](0,-2.5).. controls +(right:2cm) and +(left:3cm)..(5,-1);
\begin{scope}[shift={(-7,0)}]
\draw[thick](-0,1) .. controls +(-120:.3cm) and +(90:.3cm).. (-0.25,0).. controls +(-90:.3cm) and +(120:.3cm).. (0,-1);
\draw[thin](0,1) .. controls +(-60:.3cm) and +(90:.3cm).. (.25,0).. controls +(-90:.3cm) and +(60:.3cm).. (0,-1);
\end{scope}
\begin{scope}[shift={(-5,0)}]
\draw[thick](-0,1) .. controls +(-120:.3cm) and +(90:.3cm).. (-0.25,0).. controls +(-90:.3cm) and +(120:.3cm).. (0,-1);
\draw[thin](0,1) .. controls +(-60:.3cm) and +(90:.3cm).. (.25,0).. controls +(-90:.3cm) and +(60:.3cm).. (0,-1);
\end{scope}
\begin{scope}[shift={(5,0)}]
\draw[thick](-0,1) .. controls +(-120:.3cm) and +(90:.3cm).. (-0.25,0).. controls +(-90:.3cm) and +(120:.3cm).. (0,-1);
\draw[thick](0,1) .. controls +(-60:.3cm) and +(90:.3cm).. (.25,0).. controls +(-90:.3cm) and +(60:.3cm).. (0,-1);
\end{scope}
\begin{scope}[shift={(0,+1.5)}]
\draw[thick](-0,1) .. controls +(-120:.3cm) and +(90:.3cm).. (-0.25,0).. controls +(-90:.3cm) and +(120:.3cm).. (0,-1);
\draw[thin](0,1) .. controls +(-60:.3cm) and +(90:.3cm).. (.25,0).. controls +(-90:.3cm) and +(60:.3cm).. (0,-1);
\end{scope}
\begin{scope}[shift={(0,-1.5)}]
\draw[thick](-0,1) .. controls +(-120:.3cm) and +(90:.3cm).. (-0.25,0).. controls +(-90:.3cm) and +(120:.3cm).. (0,-1);
\draw[thin](0,1) .. controls +(-60:.3cm) and +(90:.3cm).. (.25,0).. controls +(-90:.3cm) and +(60:.3cm).. (0,-1);
\end{scope}
\draw[thick]($(0,-1.5)+(50:2)$) arc (50:130:2);
\draw[thick]($(0,1.5)+(-40:2)$) arc (-40:-140:2);
\node at (-5,-1.5){$Z$};
\node at (5,-1.5){$X$};
\node at (-7,-1.5){$Y$};
\node at (0,-3.2){$Z'$};
\node at (2.5,0){$b'$};
\node at (-6,0){$t$};
\node at (-2.5,0){$g$};
\node at (-3.5,-3){$\overset{\underbrace{\qquad\qquad\qquad\qquad\qquad\qquad}}{t'}$};
\node at (0,3){$\underset{\overbrace{\qquad\qquad\qquad\qquad\qquad\qquad\qquad\qquad}}{b}$};
\end{tikzpicture}
\]
Regarding $g$ as a bordism from $Z$ to $Z'$ we see that 
\[
t'=(\id_{Y}\amalg g)\circ t
\qquad \text{and}\qquad 
b=b'\circ (g\amalg \id_{X}),
\]
which implies that the triple $(Z,t,b)$ and $(Z',t',b')$ are equivalent in the sense of  Definition \ref{def:thickened}. If $[Z'',t'',b'']\in\wh{\RBord{d}}$ is another element with $\Psi([Z'',t'',b''])=\Sigma$, then by choosing $\epsilon$ small enough we conclude $[Z'',t'',b'']=[Z,t,b]=[Z',t',b']$.

For the proof of part (3), we decompose as in the proof of (1) the bordism $\Sigma$ into two pieces $t$ and $b$ by cutting it along the one-codimensional submanifold $Z=X\times\{\epsilon\}$ (here $Y=X$ since $\Sigma$ is an {\em endomorphism}). We regard $t$ as a bordism from $\emptyset$ to $X\amalg Z$ and $b$ as a bordism from $Z\amalg X$ to $\emptyset$. Then $\tr(\Sigma)$ is given by the composition 
\[
\xymatrix{
\emptyset\ar[r]^<>(.5)t&X\amalg Z\ar[rr]^{s_{X,Z}}&&Z\amalg X\ar[r]^<>(.5)b&\emptyset
}.
\]
which geometrically means to glue the two bordisms along $X$ and $Z$. Gluing first along $Z$ we obtain the Riemannian manifold $\Sigma$, then gluing along $X$ we get the closed Riemannian manifold $\Sigma_{gl}$.
\end{proof}
\section{Properties of the trace pairing}
The goal of this section is the proof of our main theorem \ref{thm:main2} according to which our trace pairing is symmetric, additive and multiplicative. There are three subsections devoted to the proof of these three properties, plus a subsection on braided monoidal and balanced monoidal categories needed for the multiplicative property. Each proof will be based on first proving the following analogous properties for the  trace $\wh\tr(\wh f)$ of thickened endomorphisms $\wh f\in \wh \sC(X,X)$:

\begin{thm}\label{thm:whtr} Let $\sC$ be a monoidal category, equipped with a natural family of  isomorphisms $s=s_{X,Y}\colon X\otimes Y\to Y\otimes X$. Then the trace map 
\[
\wh\tr\colon \wh\sC(X,X)\ra \sC(\one,\one)
\]
has the following properties:
\begin{enumerate}
\item (symmetry) $\wh\tr(\wh f\circ g)=\wh\tr(g\circ\wh f)$ for $\wh f\in \wh\sC(X,Y)$, $g\in \sC(Y,X)$;
\item (additivity) If $\sC$ is an additive category with distributive monoidal structure (see Definition \ref{def:distributive}), then $\wh\tr$ is a linear map;
\item \label{mult} (multiplicativity) $\wh\tr(\wh f_1\otimes \wh f_2)=\wh\tr(\wh f_1)\otimes\wh\tr(\wh f_2)$
for $\wh f_1\in \wh\sC(X_1,X_1)$, $\wh f_2\in \wh\sC(X_2,X_2)$, provided $\sC$ is  a symmetric monoidal category with braiding $s$. More generally, this property holds if $\sC$ is a {\em balanced monoidal category} (see Definition \ref{def:balanced}). 
\end{enumerate}
\end{thm}

For the tensor product $\wh f_1\otimes \wh f_2\in \wh\sC(X_1\otimes X_2,X_1\otimes X_2)$ of the thickened morphisms $\wh f_i$ see Definition \ref{def:wh_tensor}.

\subsection{The symmetry property of the trace pairing}
\begin{proof}[Proof of part (1) of Theorem \ref{thm:whtr}] Let $\wh f=[Z,t,b]\in \wh\sC(X,Y)$. Then 
\[
\wh f\circ g=[Z,t,b\circ (\id_Z\otimes g)]
\qquad\text{and}\qquad
g\circ\wh f=[Z,(g\otimes \id_Z)\circ t,b]
\]
and hence
\[
\begin{tikzpicture}
[scale=.7,
big/.style={rectangle,rounded corners,fill=white,draw,minimum height=.5cm,minimum width=1cm},
small/.style={rectangle,rounded corners,fill=white,draw,minimum height=.5cm,minimum width=.5cm}]
\node at (-1,0){$\wh\tr(\wh f\circ g)=$};
\begin{scope}[shift={(2,0)}]
\draw[thick, rounded corners](0,3.5)--(0,.5)--(1,-.5)--(1,-3.5);
\draw[thick, rounded corners](1,3.5)--(1,.5)--(0,-.5)--(0,-3.5);
\node[big] at (.5,3.5){$t$};
\node[big] at (.5,-3.5){$b$};
\node[small] at (1,-1.5){$g$};
\node at (1.4,-2.5){$X$};
\node at (1.4,2.5){$Z$};
\node at (-.4,2.5){$Y$};
\end{scope}
\node at(4.5,0){$=$};
\begin{scope}[shift={(6,0)}]
\draw[thick, rounded corners](0,3.5)--(0,.5)--(1,-.5)--(1,-3.5);
\draw[thick, rounded corners](1,3.5)--(1,.5)--(0,-.5)--(0,-3.5);
\node[big] at (.5,3.5){$t$};
\node[big] at (.5,-3.5){$b$};
\node[small] at (0,1.5){$g$};
\node at (1.4,-2.5){$X$};
\node at (1.4,2.5){$Z$};
\node at (-.4,2.5){$Y$};
\end{scope}
\node at (10,0){$=\wh\tr(g\circ \wh f)$};
\end{tikzpicture}
\]
Here the second equality follows from the naturality of the switching isomorphism (see Picture \eqref{eq:s_naturality}).
\end{proof}
\begin{proof}[Proof of part (1) of Theorem \ref{thm:main2}]
Let $\wh g\in \wh\sC(Y,X)$ with $\Psi(\wh g)=g$. Then
\[
\tr(f,g)=\wh\tr(\wh f\circ g)=\wh\tr(g\circ \wh f)=\wh\tr(\wh g\circ f)=\tr(f,g).
\]
Here the second equation is part (1) of Theorem \ref{thm:whtr}, while the third is a consequence of Lemma \ref{lem:hat_composition2}. 
\end{proof}
\subsection{Additivity of the trace pairing}\label{subsec:additivity}

Throughout this subsection we will assume that the category $\sC$ is an {\em additive category with distributive monoidal structure} (see Definitions \ref{def:additive} and \ref{def:distributive} below). Often an additive category  is defined as a category enriched over abelian groups with finite products (or equally coproducts). However, the abelian group structure on the morphism sets is actually {\em determined} by the underlying category $\sC$, and hence  a better point of view is to think of `additive' as a {\em property} of a category $\sC$, rather than specifying additional data. 

\begin{defn}\label{def:additive}
 A category $\sC$ is {\em additive} if 
\begin{enumerate}
\item There is a zero object $0\in \sC$ (an object which is terminal and initial);
\item finite products and coproducts exist;
\item Given a finite collection of objects, the canonical morphism from their coproduct to their product (given by identity maps on coordinates) is an isomorphism. 
\item Any morphism from $X\to Y$ has an additive inverse under the canonical addition (defined below) on $\sC(X,Y)$. 
\end{enumerate}
We remark that the requirement (1) is redundant since an initial object is a coproduct of the empty family of objects, and a terminal object is the product of the same family; the map from the initial to the terminal object is an isomorphism by (3), thus making these objects zero objects. 
\end{defn}

If $\sC$ is an additive category, we will use the notation $X_1\oplus \dots\oplus X_n$ for the coproduct of objects $X_i\in \sC$ (the term {\em direct sum} is used for finite coproducts in additive categories).  Any map
\[
f\colon X=X_1\oplus \dots\oplus X_n\ra Y=Y_1\oplus \dots\oplus Y_m,
\]
amounts to an $m\times n$ matrix of morphisms $f_{ij}\in \sC(X_j,Y_i)$, $1\le i\le m$, $1\le j\le n$, by identifying  $Y$ via property (3) with the {\em product} of the $Y_i$'s. In particular, there are morphisms
\[
\Delta:=\left(\begin{smallmatrix}\id\\\id\end{smallmatrix}\right)
\colon X\ra X\oplus X
\qquad
\nabla:=\left(\begin{smallmatrix}\id&\id\end{smallmatrix}\right)
\colon Y\oplus Y\ra Y
\]
referred to as {\em diagonal map} and {\em fold map}, respectively.
For $f,g\in \sC(X,Y)$ their sum $f+g$ is defined to be the composition
\[
X\overset{\Delta} \ra X\oplus X\overset {f\oplus g}\ra Y\oplus Y\overset\nabla\ra Y
\]
This addition gives $\sC(X,Y)$ the structure of an abelian group. Abusing notation we write $0\colon X\to Y$ for the additive unit which is given by the unique morphism that factors through the zero object. 
Identifying morphisms between direct sums with matrices, their composition corresponds to multiplication of the corresponding matrices.

\begin{defn}\label{def:distributive} A monoidal structure $\otimes$ on an additive category $\sC$ is {\em distributive} if for any objects $X,Y\in \sC$  the functors
\begin{alignat*}{3}
\sC&\ra \sC\qquad &\qquad\qquad \sC&\ra \sC\\
Z&\mapsto Z\otimes X &Z&\mapsto Y\otimes Z
\end{alignat*}
preserve coproducts. In particular, for objects $Z_1,Z_2\in \sC$ we have canonical distributivity isomorphisms
\[
(Z_1\otimes X)\oplus (Z_1\otimes X)\cong (Z_1\oplus Z_2)\otimes X
\quad\text{and}\quad
(Y\otimes Z_1)\oplus (Y\otimes Z_2)\cong Y\otimes (Z_1\oplus Z_2).
\]
In addition, these functors send the initial object $0\in \sC$ (thought of as the coproduct of the empty family of objects of $\sC$) to an initial object of $\sC$; in others words, there are canonical isomorphisms $0\otimes X\cong 0$ and $Y\otimes 0\cong 0$.
\end{defn}

\begin{defn}\label{def:wh_addition}
Given triples $(Z_1,t_1,b_1)$ and $(Z_2,t_2,b_2)$ representing elements of $\wh \sC(X,Y)$ we define their {\em sum} by
\[
(Z_1,t_1,b_1)+(Z_2,t_2,b_2):=\left(Z_1\oplus Z_2, 
\left(\begin{smallmatrix}
t_1\\t_2
\end{smallmatrix}\right),
\left(\begin{smallmatrix}
b_1&b_2
\end{smallmatrix}\right)\right)
\]
Here  $\left(\begin{smallmatrix}
t_1\\t_2
\end{smallmatrix}\right)$
is a morphism from $\one$ to $(Y\otimes Z_1)\oplus (Y\otimes Z_2)\cong Y\otimes (Z_1\oplus Z_1)$ and 
$\left(\begin{smallmatrix}
b_1&b_2
\end{smallmatrix}\right)
$
is a morphism from 
$(Z_1\otimes X)\oplus (Z_2\otimes X)\cong  (Z_1\oplus Z_1)\otimes X$ to $\one$. 
The addition of triples induces a well-defined addition
\[
\wh\sC(X,Y)\times  \wh\sC(X,Y)\overset{+}\ra \wh\sC(X,Y)
\]
since, if $g_i\colon Z_i\to Z'_i$ is an equivalence from $(Z_i,t_i,b_i)$ to $(Z'_i,t'_i,b'_i)$, then $g_1\oplus g_2\colon Z_1\oplus Z_2\to Z'_1\oplus Z'_2$ is an equivalence from $(Z_1,t_1,b_1)+(Z_2,t_2,b_2)$ to $(Z'_1,t'_1,b'_1)+(Z'_2,t'_2,b'_2)$.
\end{defn}

\begin{lem} The above addition gives $\wh\sC(X,Y)$ the structure of an abelian group. 
\end{lem}

\begin{proof}
We claim:
\begin{enumerate}
\item The additive unit is represented by any triple $(Z,t,b)$ with $t=0$ or $b=0$;
\item the additive inverse of $[Z,t,b]$ is represented by $(Z,-t,b)$ or $(Z,t,-b)$ (here $-t$, $-b$ are the additive inverses to the morphisms $t$ respectively $b$ whose existence is guaranteed by axiom (4) in the definition of an additive category). 
\end{enumerate}
It follows from the description of addition that $[0,0,0]$ is an additive unit in $\wh\sC(X,Y)$. It is easy to check that $0\colon Z\to 0$ is an equivalence from $(Z,t,0)$ to $(0,0,0)$ and that $0\colon 0\to Z$ is an equivalence from $(0,0,0)$ to $(Z,0,b)$, which proves the first claim.

The diagonal map $\Delta\colon Z\to Z\oplus Z$ is an equivalence from $(Z,t,0)$ to 
$(Z,t,b)+(Z,t,-b)=(Z\oplus Z, 
\left(\begin{smallmatrix} t\\t\end{smallmatrix}\right),
\left(\begin{smallmatrix} b&-b\end{smallmatrix}\right))
$.
Similarly, the fold map $\nabla\colon Z\oplus Z\to Z$ is an equivalence from 
$(Z,t,b)+(Z,-t,b)=(Z\oplus Z, 
\left(\begin{smallmatrix} t\\-t\end{smallmatrix}\right),
\left(\begin{smallmatrix} b&b\end{smallmatrix}\right))
$
to $(Z,0,b)$. This proves the second claim.
\end{proof}

\begin{lem}\label{lem:Psi_linear} The map $\Psi\colon\wh\sC(X,Y)\to\sC(X,Y)$ is a homomorphism.
\end{lem}
\begin{proof} Using the notation from the definition above, we want to show that $\Psi(Z,t,b)=\Psi(Z_1,t_1,b_1)+\Psi(Z_1,t_1,b_1)$. We recall (see Equation \eqref{eq:factorization} of the introduction and the paragraph following it) that $\Psi(Z,t,b)\in \sC(X,Y)$ is given by the composition
\begin{equation*}
Y\cong Y\otimes \one\overset{\id_Y\otimes b}\longleftarrow Y\otimes Z\otimes X\overset{t\otimes \id_X}\longleftarrow \one\otimes X\cong X
\end{equation*}
Here we write the arrows from right to left in order that composition corresponds to matrix multiplication. Identifying $Y\otimes Z\otimes X$ with $(Y\otimes Z_1\otimes X)\oplus (Y\otimes Z_2\otimes X)$, these maps are given by the following matrices:
\[
\id_Y\otimes b=
\left(\begin{matrix}
\id_Y\otimes b_1&\id_Y\otimes b_2
\end{matrix}\right)
\qquad
t\otimes \id_X=
\left(\begin{matrix}
t_1\otimes \id_X\\t_2\otimes \id_X
\end{matrix}\right)
\]
Hence $\Psi(Z,t,b)$ is given by the matrix product
\begin{align*}
&\left(\begin{matrix}
\id_Y\otimes b_1&\id_Y\otimes b_2
\end{matrix}\right)
\left(\begin{matrix}
t_1\otimes \id_X\\t_2\otimes \id_X
\end{matrix}\right)\\
=&(\id_Y\otimes b_1)\circ (t_1\otimes \id_X)+(\id_Y\otimes b_2)\circ (t_2\otimes \id_X)\\
=&\Psi(Z_1,t_1,b_1)+\Psi(Z_2,t_2,b_2)
\end{align*}
\end{proof}
\begin{proof}[Proof of part (2) of Theorem \ref{thm:whtr}]
We recall from the definition of $\wh\tr$ (see Equation \eqref{eq:whtr}) that $\wh\tr(Z,t,b)\in \sC(\one,\one)$ is given by the composition
\begin{equation*}
\one\overset{b}\longleftarrow Z\otimes X\overset{s_{X,Z}}\longleftarrow X\otimes Z\overset t\longleftarrow \one
\end{equation*}
Identifying $ Z\otimes X$ with $(Z_1\otimes X)\oplus (Z_2\otimes X)$, and $ X\otimes Z$ with $(X\otimes Z_1)\oplus (X\otimes Z_2)$, this composition is given by the following matrix product:
\begin{align*}
&\left(\begin{matrix}
b_1&b_2
\end{matrix}\right)
\left(\begin{matrix}
s_{X,Z_1}&0\\
0&s_{X,Z_2}
\end{matrix}\right)
\left(\begin{matrix}
t_1\\t_2
\end{matrix}\right)\\
=&b_1\circ s_{X,Z_1}\circ t_1+
b_2\circ s_{X,Z_2}\circ t_2\\
=&\wh\tr(Z_1,t_1,b_1)+\wh\tr(Z_2,t_2,b_2)
\end{align*}
\end{proof}
To show that the additivity of $\wh\tr$ implies that the pairing $\tr(f,g)$ is linear in $f$ and $g$, we will need the following fact.

\begin{lem} The composition map
$\wh\sC(X,Y)\times \wh\sC(Y,X)\ra \wh \sC(Y,Y)$ is bilinear.
\end{lem}
The proof of this lemma is analogous to the previous two proofs, and so we leave it to the reader.

\begin{proof}[Proof of part (2) of Theorem \ref{thm:main2}]
Let  $f_1,f_2\in \sC(X,Y)$, $\wh f_1,\wh f_2\in \wh\sC(X,Y)$ with $\Psi(\wh f_i)=f_i$, and $g\in \sC^{tk}(Y,X)$. Then by Lemma \ref{lem:Psi_linear} $\Psi(\wh f_1+\wh f_2)=f_1+f_2$, and hence 
\begin{align*}
\tr(f_1+f_2,g)=&\wh \tr((\wh f_1+\wh f_2)\circ g)=\tr(\wh f_1\circ g+\wh f_2\circ g)\\
=&\wh\tr(\wh f_1\circ g)+\wh\tr( \wh f_2\circ g)
=\tr(f_1,g)+\tr(f_2,g)
\end{align*}
\end{proof}
\subsection{Braided and balanced monoidal categories}
We recall from \cite[Definition 2.1]{JS2} that a braided monoidal category is a monoidal category equipped with a natural family of isomorphisms $c=c_{X,Y}\colon X\otimes Y\to Y\otimes X$ such that 
\begin{equation}\label{eq:relations}
\begin{tikzpicture}
[scale=.7,
big/.style={rectangle,rounded corners,fill=white,draw,minimum height=.5cm,minimum width=1cm},
small/.style={rectangle,rounded corners,fill=white,draw,minimum height=.5cm,minimum width=.5cm}]
\begin{scope}[shift={(0,0)}]
\draw[thick, rounded corners](-1,2)node[above]{$X$}--(-1,-2)node[below]{$Y$};
\draw[thick, rounded corners](0,2)node[above]{$Y$}--(0,-2)node[below]{$Z$};
\draw[thick, rounded corners](1,2)node[above]{$Z$}--(1,-2)node[below]{$X$};
\node[rectangle,rounded corners,fill=white,draw,minimum height=.5cm,minimum width=2cm] at (0,0){$c_{X,Y\otimes Z}$};
\end{scope}
\node at (2,0){$=$};
\begin{scope}[shift={(4,0)}]
\draw[thick, rounded corners](-1,2)node[above]{$X$}--(-1,-2)node[below]{$Y$};
\draw[thick, rounded corners](0,2)node[above]{$Y$}--(0,-2)node[below]{$Z$};
\draw[thick, rounded corners](1,2)node[above]{$Z$}--(1,-2)node[below]{$X$};
\node[big] at (-.5,1){$c_{X,Y}$};
\node[big] at (.5,-1){$c_{X,Z}$};
\end{scope}
\node at (7,0){and};
\begin{scope}[shift={(10,0)}]
\draw[thick, rounded corners](-1,2)node[above]{$X$}--(-1,-2)node[below]{$Z$};
\draw[thick, rounded corners](0,2)node[above]{$Y$}--(0,-2)node[below]{$X$};
\draw[thick, rounded corners](1,2)node[above]{$Z$}--(1,-2)node[below]{$Y$};
\node[rectangle,rounded corners,fill=white,draw,minimum height=.5cm,minimum width=2cm] at (0,0){$c_{X\otimes Y,Z}$};
\end{scope}
\node at (12,0){$=$};
\begin{scope}[shift={(14,0)}]
\draw[thick, rounded corners](-1,2)node[above]{$X$}--(-1,-2)node[below]{$Z$};
\draw[thick, rounded corners](0,2)node[above]{$Y$}--(0,-2)node[below]{$X$};
\draw[thick, rounded corners](1,2)node[above]{$Z$}--(1,-2)node[below]{$Y$};
\node[big] at (-.5,-1){$c_{X,Z}$};
\node[big] at (.5,1){$c_{Y,Z}$};
\end{scope}
\end{tikzpicture}
\end{equation}
There is a close relationship between braided monoidal categories and the braid groups $B_n$. We recall that the elements of $B_n$ are {\em braids with  $n$ strands}, consisting of isotopy classes of $n$ non-intersecting piecewise smooth curves $\gamma_i\colon [0,1]\to \R^3$, $i=1,\dots,n$  with endpoints at $\{1,\dots,n\}\times\{0\}\times \{0,1\}$. One requires that the $z$\nb-coordinate of each curve is strictly decreasing (so that strands are ``going down"). Composition of braids is defined by concatenation of their strands; e.g., in $B_2$  we have the composition
\[
\begin{tikzpicture}
[scale=.7,
big/.style={rectangle,rounded corners,fill=white,draw,minimum height=.6cm,minimum width=2cm},
small/.style={rectangle,rounded corners,fill=white,draw,minimum height=.6cm,minimum width=.6cm}]
\begin{scope}[shift={(0,.5)}]
\draw[thick, rounded corners](0,.5)--(0,0)--(1,-1)--(1,-1.5);
\node[circle,fill=white] at (.5,-.5){} ;
\draw[thick, rounded corners](1,.5)--(1,0)--(0,-1)--(0,-1.5);
\end{scope}
\node at (1.5,0){$\circ$};
\begin{scope}[shift={(2,.5)}]
\draw[thick, rounded corners](1,.5)--(1,0)--(0,-1)--(0,-1.5);
\node[circle,fill=white] at (.5,-.5){} ;
\draw[thick, rounded corners](0,.5)--(0,0)--(1,-1)--(1,-1.5);
\end{scope}
\node at (4,0){$=$};
\begin{scope}[shift={(5,0)}]
\draw[thick, rounded corners](1,1.5)--(1,1)--(0,0)--(1,-1)--(1,-1.5);
\node[circle,fill=white] at (.5,.5){} ;
\node[circle,fill=white] at (.5,-.5){} ;
\draw[thick, rounded corners](0,1.5)--(0,1)--(1,0)--(0,-1)--(0,-1.5);
\end{scope}
\node at (7,0){$=$};
\begin{scope}[shift={(8,0)}]
\draw[thick, rounded corners](0,1.5)--(0,-1.5);
\draw[thick, rounded corners](1,1.5)--(1,-1.5);
\end{scope}
\end{tikzpicture}
\]
The last equality follows from the obvious isotopy in $\R^3$, also known as the second {\em Reidemeister move}. The three Reidemeister moves are used to understand isotopies of arcs in $\R^3$ via their projections to the plane $\R^2$. Generically, the isotopy projects to an isotopy in the plane but there are three codimension one singularities where this does not happen: a cusp, a tangency and a triple point (of immersed arcs in the plane). The corresponding `moves' on the planar projection are the Reidemeister moves; for example, in the above figure one can see a tangency in the middle of moving the braid in the center to the (trivial) braid on the right. A cusp singularity cannot arise for braids but a triple point can: the resulting (third Reidemeister move) isotopies are also known as {\em braid relations}.  A typical example would be the following isotopy (with a triple point in the middle):
\begin{equation}\label{eq:YB}
\begin{tikzpicture}
[scale=.7,
big/.style={rectangle,rounded corners,fill=white,draw,minimum height=.4cm,minimum width=1.5cm},
small/.style={rectangle,rounded corners,fill=white,draw,minimum height=.6cm,minimum width=.6cm}]
\begin{scope}[shift={(0,0)}]
\draw[thick,rounded corners](2,2)node[above]{$$}--(2,.5)--(1,-.5)--(0,-1.5)--(0,-2)node[below]{$$};
\node[circle,fill=white] at (1.5,0){} ;
\node[circle,fill=white] at (.5,-1){} ;
\draw[thick,rounded corners](0,2)node[above]{$$}--(0,1.5)--(1,.5)--(2,-.5)--(2,-2)node[below]{$$};
\node[circle,fill=white] at (.5,1){} ;
\draw[thick,rounded corners](1,2)node[above]{$$}--(1,1.5)--(0,.5)--(0,-.5)--(1,-1.5)--(1,-2)node[below]{$$};
\end{scope}
\node at (4,0){$=$};
\begin{scope}[shift={(6.5,0)}]
\draw[thick,rounded corners](2,2)node[above]{$$}--(2,1.5)--(1,.5)--(0,-.5)--(0,-2)node[below]{$$};
\node[circle,fill=white] at (.5,0){} ;
\draw[thick,rounded corners](0,2)node[above]{$$}--(0,.5)--(1,-.5)--(2,-1.5)--(2,-2)node[below]{$$};
\node[circle,fill=white] at (1.5,-1){} ;
\node[circle,fill=white] at (1.5,1){} ;
\draw[thick,rounded corners](1,2)node[above]{$$}--(1,1.5)--(2,.5)--(2,-.5)--(1,-1.5)--(1,-2)node[below]{$$};
\end{scope}
\end{tikzpicture}
\end{equation}

Thinking of the groups $B_n$ as categories with one object, we can form their disjoint union to obtain the {\em braid category} $\sB:=\coprod_{n=0}^\infty B_n$. The category $\sB$ can be equipped with the structure of a braided monoidal category \cite[Example 2.1]{JS2}; in fact, $\sB$ is equivalent to the {\em free braided monoidal category} generated by one object (this is a special case of Theorem 2.5 of \cite{JS2}). 

More generally, if $\sC_k$ is the free braided monoidal category generated by $k$ objects, there is a braided monoidal functor $F\colon \sC_k\to \sB$ which sends $n$\nb-fold tensor products of the generating objects to the object $n\in \sB$ (whose automorphism group is $B_n$); on morphisms, $F$ sends the structure maps $\alpha_{X,Y,Z}$, $\ell_X$ and $r_X$ (see beginning of Section~\ref{sec:thick}) to the identity, and the braiding isomorphisms $c_{X_i,X_j}$ for generating objects $X_i,X_j$ to the braid
\[
\begin{tikzpicture}
[scale=.7,
big/.style={rectangle,rounded corners,fill=white,draw,minimum height=.6cm,minimum width=2cm},
small/.style={rectangle,rounded corners,fill=white,draw,minimum height=.6cm,minimum width=.6cm}]
\begin{scope}[shift={(0,.5)}]
\draw[thick, rounded corners](0,.5)--(0,0)--(1,-1)--(1,-1.5);
\node[circle,fill=white] at (.5,-.5){} ;
\draw[thick, rounded corners](1,.5)--(1,0)--(0,-1)--(0,-1.5);
\end{scope}
\node at (2,0){$\in B_2$};
\end{tikzpicture}
\]
This suggests to represent $c_{X,Y}$ by an overcrossing and $c^{-1}_{Y,X}$ by an undercrossing in the pictorial representations of morphisms in braided monoidal categories, a convention broadly used in the literature \cite{JS1} that we will adopt. In other words, one defines:
\[
\begin{tikzpicture}
[scale=.7,
big/.style={rectangle,rounded corners,fill=white,draw,minimum height=.8cm,minimum width=1.3cm},
small/.style={rectangle,rounded corners,fill=white,draw,minimum height=.3cm,minimum width=.3cm}]
\begin{scope}[shift={(0,0)}]
\draw[thick,rounded corners](1,.5)node[above]{$Y$}--(1,0)--(0,-1)--(0,-1.5)node[below]{$Y$};
\node[fill=white] at (.5,-.5){};
\draw[thick,rounded corners](0,.5)node[above]{$X$}--(0,0)--(1,-1)--(1,-1.5)node[below]{$X$};
\node at (2,-.5){$:=$};
\end{scope}
\begin{scope}[shift={(3,0)}]
\draw[thick,rounded corners](1,.5)node[above]{$Y$}--(1,-1.5)node[below]{$X$};
\draw[thick,rounded corners](0,.5)node[above]{$X$}--(0,-1.5)node[below]{$Y$};
\node[big] at (.5,-.5){$c_{X,Y}$};
\end{scope}
\node at (6,-.5){$\text{and}$};
\begin{scope}[shift={(8,0)}]
\draw[thick,rounded corners](0,.5)node[above]{$X$}--(0,0)--(1,-1)--(1,-1.5)node[below]{$X$};
\node[fill=white] at (.5,-.5){};
\draw[thick,rounded corners](1,.5)node[above]{$Y$}--(1,0)--(0,-1)--(0,-1.5)node[below]{$Y$};
\node at (2,-.5){$:=$};
\end{scope}
\begin{scope}[shift={(11,0)}]
\draw[thick,rounded corners](1,.5)node[above]{$Y$}--(1,-1.5)node[below]{$X$};
\draw[thick,rounded corners](0,.5)node[above]{$X$}--(0,-1.5)node[below]{$Y$};
\node[big] at (.5,-.5){$c^{-1}_{Y,X}$};
\end{scope}
\end{tikzpicture}
\]

The key result,  \cite[Cor.\ 2.6]{JS2}, is that for any objects $X,Y\in \sC_k$ the map 
\[
F\colon\sC_k(X,Y)\to \sB(F(X),F(Y))
\]
is a bijection. One step in the argument is to realize that the Yang-Baxter relations ~\ref{eq:YB} follow from the relations \ref{eq:relations} together with the naturality of the braiding isomorphism $c$. An immediate consequence is the following statement which we will refer to below. 

\begin{prop}{\bf (Joyal-Street)}\label{prop:isotopy} Let 
$f,g\colon X_1\otimes \dots \otimes X_k\to X_{\sigma(1)}\otimes \dots \otimes X_{\sigma(k)}$
be morphisms of a braided monoidal category $\sC$ which are in the image of the tautological functor $T\colon \sC_k\to \sC$ which sends the $i$\nb-th generating object of the free braided monoidal category $\sC_k$ to $X_i$. If $f=T(\wt f)$, $g=T(\wt g)$ and the associated braids $F(\wt f), F(\wt g)\in B_k$ agree, then $f=g$.
\end{prop}

We note that a morphism $f$ is in the image of $T$ \Iff it can be written as a composition of the isomorphisms $\alpha$, $\ell$, $r$, $c$ and their inverses. For such a composition it is easy to read off the braid $F(\wt f)$: in the pictorial representation of $f$ we simply ignore all associators and units and replace each occurrence of the braiding isomorphism $c$ (respectively its inverse) by an overcrossing (respectively and undercrossing); then $F(\wh f)$ is the resulting braid.  For example, for

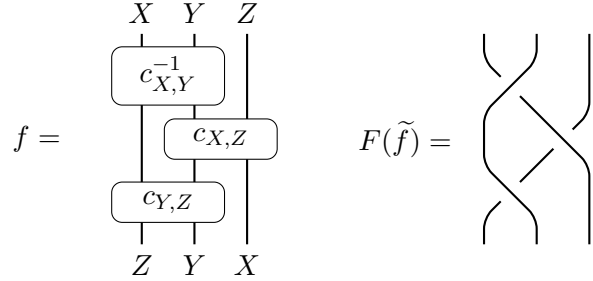
\[
\begin{tikzpicture}
[scale=.7,
big/.style={rectangle,rounded corners,fill=white,draw,minimum height=.4cm,minimum width=1.5cm},
small/.style={rectangle,rounded corners,fill=white,draw,minimum height=.6cm,minimum width=.6cm}]
\node at (-2,0){$f=$};
\begin{scope}[shift={(0,0)}]
\draw[thick](0,2)node[above]{$X$}--(0,-2)node[below]{$Z$};
\draw[thick](1,2)node[above]{$Y$}--(1,-2)node[below]{$Y$};
\draw[thick](2,2)node[above]{$Z$}--(2,-2)node[below]{$X$};
\node[big] at (.5,1.2){$c_{X,Y}^{-1}$};
\node[big] at (1.5,0){$c_{X,Z}$};
\node[big] at (.5,-1.2){$c_{Y,Z}$};
\end{scope}
\node at (5,0){$F(\wt f)=$};
\begin{scope}[shift={(6.5,0)}]
\draw[thick,rounded corners](2,2)node[above]{$$}--(2,.5)--(1,-.5)--(0,-1.5)--(0,-2)node[below]{$$};
\node[circle,fill=white] at (1.5,0){} ;
\node[circle,fill=white] at (.5,-1){} ;
\draw[thick,rounded corners](0,2)node[above]{$$}--(0,1.5)--(1,.5)--(2,-.5)--(2,-2)node[below]{$$};
\node[circle,fill=white] at (.5,1){} ;
\draw[thick,rounded corners](1,2)node[above]{$$}--(1,1.5)--(0,.5)--(0,-.5)--(1,-1.5)--(1,-2)node[below]{$$};
\end{scope}
\end{tikzpicture}
\]

Another result that we will need for the proof of the multiplicativity property are the following isotopy relations. Roughly speaking, they say that if the unit $\one\in \sC$ is involved, then the isotopy does not need to be `relative boundary' as in the previous pictures.

\begin{lem}\label{lem:over/under_crossing}
Let $V$, $W$ be objects of a braided monoidal category $\sC$, and let $f\colon V\to \one$ and $g\colon \one\to V$ be morphisms in $\sC$. Then there are the following relations:
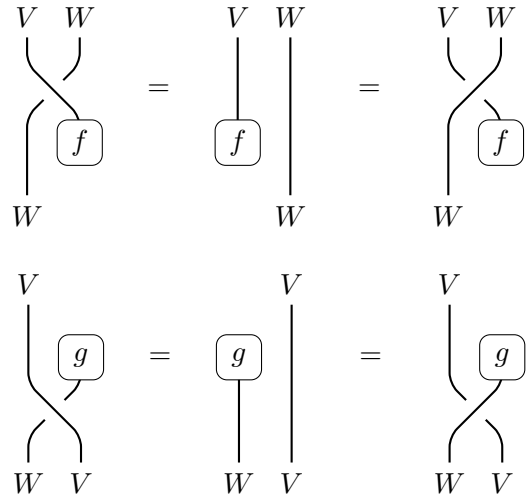
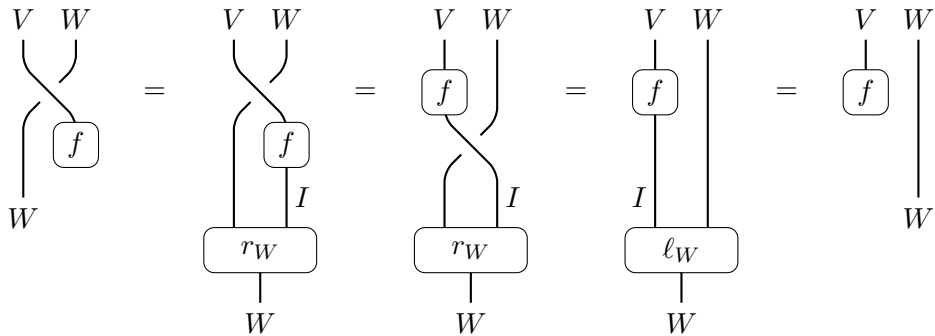
\[
\begin{tikzpicture}
[scale=.7,
big/.style={rectangle,rounded corners,fill=white,draw,minimum height=.6cm,minimum width=1.5cm},
small/.style={rectangle,rounded corners,fill=white,draw,minimum height=.6cm,minimum width=.6cm}]
\begin{scope}[shift={(0,0)}]
\draw[thick, rounded corners](1,.5)node[above]{$W$}--(1,0)--(0,-1)--(0,-2.5)node[below]{$W$};
\node[circle,fill=white] at (.5,-.5){};
\draw[thick, rounded corners](0,.5)node[above]{$V$}--(0,0)--(1,-1)--(1,-1.5);
\node[small] at (1,-1.5){$f$};
\node at (2.5,-.5) {$=$};
\end{scope}
\begin{scope}[shift={(4,0)}]
\draw[thick, rounded corners](1,.5)node[above]{$W$}--(1,-2.5)node[below]{$W$};
\draw[thick, rounded corners](0,.5)node[above]{$V$}--(0,-1.5);
\node[small] at (0,-1.5){$f$};
\node at (2.5,-.5) {$=$};
\end{scope}
\begin{scope}[shift={(8,0)}]
\draw[thick, rounded corners](0,.5)node[above]{$V$}--(0,0)--(1,-1)--(1,-1.5);
\node[small] at (1,-1.5){$f$};
\node[circle,fill=white] at (.5,-.5){};
\draw[thick, rounded corners](1,.5)node[above]{$W$}--(1,0)--(0,-1)--(0,-2.5)node[below]{$W$};
\end{scope}
\end{tikzpicture}
\]
\[
\begin{tikzpicture}
[scale=.7,
big/.style={rectangle,rounded corners,fill=white,draw,minimum height=.6cm,minimum width=1.5cm},
small/.style={rectangle,rounded corners,fill=white,draw,minimum height=.6cm,minimum width=.6cm}]
\begin{scope}[shift={(0,0)}]
\draw[thick, rounded corners](1,-.5)--(1,-1)--(0,-2)--(0,-2.5)node[below]{$W$};
\node[circle,fill=white] at (.5,-1.5){};
\draw[thick, rounded corners](0,.5)node[above]{$V$}--(0,-1)--(1,-2)--(1,-2.5)node[below]{$V$};
\node[small] at (1,-.5){$g$};
\node at (2.5,-.5) {$=$};
\end{scope}
\begin{scope}[shift={(4,0)}]
\draw[thick, rounded corners](1,.5)node[above]{$V$}--(1,-2.5)node[below]{$V$};
\draw[thick, rounded corners](0,-.5)--(0,-2.5)node[below]{$W$};
\node[small] at (0,-.5){$g$};
\node at (2.5,-.5) {$=$};
\end{scope}
\begin{scope}[shift={(8,0)}]
\draw[thick, rounded corners](0,.5)node[above]{$V$}--(0,-1)--(1,-2)--(1,-2.5)node[below]{$V$};
\node[circle,fill=white] at (.5,-1.5){};
\draw[thick, rounded corners](1,-.5)--(1,-1)--(0,-2)--(0,-2.5)node[below]{$W$};
\node[small] at (1,-.5){$g$};
\end{scope}
\end{tikzpicture}
\]
\end{lem}
\begin{proof}
\[
\begin{tikzpicture}
[scale=.7,
big/.style={rectangle,rounded corners,fill=white,draw,minimum height=.6cm,minimum width=1.5cm},
small/.style={rectangle,rounded corners,fill=white,draw,minimum height=.6cm,minimum width=.6cm}]
\begin{scope}[shift={(0,0)}]
\draw[thick, rounded corners](1,.5)node[above]{$W$}--(1,0)--(0,-1)--(0,-2.5)node[below]{$W$};
\node[circle,fill=white] at (.5,-.5){};
\draw[thick, rounded corners](0,.5)node[above]{$V$}--(0,0)--(1,-1)--(1,-1.5);
\node[small] at (1,-1.5){$f$};
\node at (2.5,-.5) {$=$};
\end{scope}
\begin{scope}[shift={(4,0)}]
\draw[thick, rounded corners](1,.5)node[above]{$W$}--(1,0)--(0,-1)--(0,-3.5);
\node[circle,fill=white] at (.5,-.5){};
\draw[thick, rounded corners](0,.5)node[above]{$V$}--(0,0)--(1,-1)--(1,-3.5);
\draw[thick, rounded corners](.5,-3.5)--(.5,-4.5)node[below]{$W$};
\node[small] at (1,-1.5){$f$};
\node at (1.3,-2.5){$\one$};
\node[big] at (.5,-3.5){$r_W$};
\node at (2.5,-.5) {$=$};
\end{scope}
\begin{scope}[shift={(8,0)}]
\draw[thick, rounded corners](1,.5)node[above]{$W$}--(1,-1)--(0,-2)--(0,-3.5);
\node[circle,fill=white] at (.5,-1.5){};
\draw[thick, rounded corners](0,.5)node[above]{$V$}--(0,-1)--(1,-2)--(1,-3.5);
\draw[thick, rounded corners](.5,-3.5)--(.5,-4.5)node[below]{$W$};
\node[small] at (0,-.5){$f$};
\node at (1.3,-2.5){$\one$};
\node[big] at (.5,-3.5){$r_W$};
\node at (2.5,-.5) {$=$};
\end{scope}
\begin{scope}[shift={(12,0)}]
\draw[thick, rounded corners](1,.5)node[above]{$W$}--(1,-3.5);
\node[circle,fill=white] at (.5,-1.5){};
\draw[thick, rounded corners](0,.5)node[above]{$V$}--(0,-3.5);
\draw[thick, rounded corners](.5,-3.5)--(.5,-4.5)node[below]{$W$};
\node[small] at (0,-.5){$f$};
\node at (-.3,-2.5){$\one$};
\node[big] at (.5,-3.5){$\ell_W$};
\node at (2.5,-.5) {$=$};
\end{scope}
\begin{scope}[shift={(16,0)}]
\draw[thick, rounded corners](1,.5)node[above]{$W$}--(1,-2.5)node[below]{$W$};
\draw[thick, rounded corners](0,.5)node[above]{$V$}--(0,-.5);
\node[small] at (0,-.5){$f$};
\end{scope}
\end{tikzpicture}
\]
Here the first and last equality come from interpreting compositions involving the box $f$ with no output ($=$ morphism with range $\one$). The second equality is the naturality of the braiding isomorphism, and the third equality is a compatibility between the unit constraints $r_W$, $\ell_W$ and the braiding isomorphism which is a consequence of the relations \eqref{eq:relations} \cite[Prop.\ 2.1]{JS2}.

This proves the first equality; the proofs of the other equalities are analogous. 
\end{proof}

\begin{defn}\cite[Def.\ 6.1]{JS2}\label{def:balanced} A {\em balanced monoidal category} is a braided monoidal category $\sC$ together with a natural family of isomorphisms $\theta_X\colon X\overset \cong\to X$, called {\em twists}, parametrized by the objects $X\in\sC$. We note that the naturality means that for any morphism $g\colon X\to Y$ we have $g\circ \theta_X=\theta_Y\circ g$. In addition it is required that $\theta_X=\id_X$ for $X=\one$, and for any objects $X,Y\in\sC$ the twist $\theta_{X\otimes Y}$ is determined by the following equation:

\begin{equation}\label{eq:twist}
\begin{tikzpicture}
[scale=.9,
big/.style={rectangle,rounded corners,fill=white,draw,minimum height=.5cm,minimum width=1cm},
small/.style={rectangle,rounded corners,fill=white,draw,minimum height=.5cm,minimum width=.5cm}]
\begin{scope}[shift={(0,0)}]
\draw[thick, rounded corners](-.5,2)node[above]{$X$}--(-.5,-2)node[below]{$X$};
\draw[thick, rounded corners](.5,2)node[above]{$Y$}--(.5,-2)node[below]{$Y$};
\node[big] at (0,0){$\theta_{X\otimes Y}$};
\end{scope}
\node at (2,0){$=$};
\begin{scope}[shift={(4,0)}]
\draw[thick, rounded corners](.5,2)node[above]{$Y$}--(.5,1.5)--(-.5,.5)--(-.5,0);
\node[circle,fill=white] at (0,1){} ;
\draw[thick, rounded corners](-.5,2)node[above]{$X$}--(-.5,1.5)--(.5,.5)--(.5,0);
\draw[thick, rounded corners](.5,0)--(.5,-.5)--(-.5,-1.5)--(-.5,-2)node[below]{$X$};
\node[circle,fill=white] at (0,-1){} ;
\draw[thick, rounded corners](-.5,0)--(-.5,-.5)--(.5,-1.5)--(.5,-2)node[below]{$Y$};
\node[small] at (-.5,0){$\theta_Y$};
\node[small] at (.5,0){$\theta_X$};
\end{scope}
\end{tikzpicture}
\end{equation}

If $\theta_X=\id_X$ for {\em all objects} $X\in \sC$, this equality reduces to the requirement $c_{Y,X}\circ c_{X,Y}=\id_{X\otimes Y}$ for the braiding isomorphism. This is the additional requirement in a {\em symmetric monoidal category}, so symmetric monoidal categories can be thought of as balanced  monoidal categories with $\theta_X=\id_X$ for all objects $X$. 

For a balanced monoidal category $\sC$ with braiding isomorphism $c$ and twist $\theta$, we define the switching isomorphism $s=s_{X,Y}\colon X\otimes Y\to Y\otimes X$ by $s_{X,Y}:=(\id_Y\otimes\theta_X)\circ c_{X,Y}$; pictorially:
\begin{equation}\label{eq:switching}
\begin{tikzpicture}
[scale=.9,
big/.style={rectangle,rounded corners,fill=white,draw,minimum height=.5cm,minimum width=1cm},
small/.style={rectangle,rounded corners,fill=white,draw,minimum height=.5cm,minimum width=.5cm}]
\begin{scope}[shift={(0,0)}]
\draw[thick, rounded corners](-.5,1.5)node[above]{$X$}--(-.5,.5)--(.5,-.5)--(.5,-1.5)node[below]{$X$};
\draw[thick, rounded corners](.5,1.5)node[above]{$Y$}--(.5,.5)--(-.5,-.5)--(-.5,-1.5)node[below]{$Y$};
\end{scope}
\node at (2,0){$:=$};
\begin{scope}[shift={(4,-.5)}]
\draw[thick, rounded corners](.5,2)node[above]{$Y$}--(.5,1.5)--(-.5,.5)--(-.5,0)--(-.5,-1)node[below]{$Y$};
\node[circle,fill=white] at (0,1){} ;
\draw[thick, rounded corners](-.5,2)node[above]{$X$}--(-.5,1.5)--(.5,.5)--(.5,0)--(.5,-1)node[below]{$X$};
\node[small] at (.5,0){$\theta_X$};
\end{scope}
\end{tikzpicture}
\end{equation}
\end{defn}
\subsection{Multiplicativity of the trace pairing}\label{subsec:multiplicativity}
In this subsection we will prove the multiplicativity of $\wh\tr$  and derive from it the multiplicativity of the trace pairing (part (3) of Theorem \ref{thm:main2}). Our first task is to define a tensor product of thickened morphisms.

\begin{defn}\label{def:wh_tensor} Let $\wh f_i=[Z_i,t_i,b_i]\in \wh\sC(X_i,Y_i)$ for $i=1,2$. If $\sC$ is a braided monoidal category, we define 
\[
\wh f_1\otimes\wh f_2:=[Z,t,b]\in \wh\sC(X_1\otimes X_2,Y_1\otimes Y_2)
\]
where $Z=Z_1\otimes Z_2$, and 

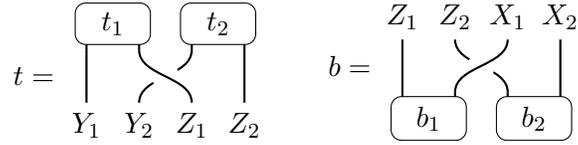
\[
\begin{tikzpicture}
[scale=.7,
big/.style={rectangle,rounded corners,fill=white,draw,minimum height=.3cm,minimum width=1cm},
small/.style={rectangle,rounded corners,fill=white,draw,minimum height=.3cm,minimum width=.3cm}]
\begin{scope}[shift={(0,0)}]
\node at (-1,-1.5){$t=$};
\begin{scope}[shift={(0,0)}]
\draw[thick](0,-.5)--+(0,-.5);
\draw[thick](1,-.5)--+(0,-.5);
\node at (.5,-.5)[big]{$t_1$};
\end{scope}
\begin{scope}[shift={(2,0)}]
\draw[thick](0,-.5)--+(0,-.5);
\draw[thick](1,-.5)--+(0,-.5);
\node at (.5,-.5)[big]{$t_2$};
\end{scope}
\begin{scope}[shift={(1,-1)}]
	\draw[thick] (1,0).. controls +(down:.5cm) and +(up:.5cm)..+(-1,-1)node[below]{$Y_2$};
		\node[circle,fill=white] at (.5,-.5){} ;
	\draw[thick] (0,0).. controls +(down:.5cm) and +(up:.5cm)..+(+1,-1)node[below]{$Z_1$};
\end{scope}
\draw[thick](0,-1)--+(0,-1)node[below]{$Y_1$};
\draw[thick](3,-1)--+(0,-1)node[below]{$Z_2$};
\end{scope}
\begin{scope}[shift={(6,-.8)}]
\node at (-1,-.5){$b=$};
\draw[thick](0,0)node[above]{$Z_1$}--+(0,-1);
\draw[thick](3,0)node[above]{$X_2$}--+(0,-1);
\begin{scope}[shift={(1,0)}]
	\draw[thick] (0,0)node[above]{$Z_2$}.. controls +(down:.5cm) and +(up:.5cm)..+(+1,-1);
			\node[circle,fill=white] at (.5,-.5){} ;
		\draw[thick] (1,0)node[above]{$X_1$}.. controls +(down:.5cm) and +(up:.5cm)..+(-1,-1);
\end{scope}
\begin{scope}[shift={(0,-1)}]
\draw[thick](0,0)--(0,-.5);
\draw[thick](1,0)--(1,-.5);
\node at (.5,-.5)[big]{$b_1$};
\end{scope}
\begin{scope}[shift={(2,-1)}]
\draw[thick](0,0)--(0,-.5);
\draw[thick](1,0)--(1,-.5);
\node at (.5,-.5)[big]{$b_2$};
\end{scope}
\end{scope}
\end{tikzpicture}
\]

\end{defn}
We will leave it to the reader to check that this tensor product is well-defined. A crucial property of this tensor product for thickened morphisms is its compatibility with the tensor product for morphisms:
\begin{lem}\label{lem:Psi_multiplicative} The map $\Psi\colon \wh\sC(X,Y)\to \sC(X,Y)$ is multiplicative, \ie, $\Psi(\wh f_1\otimes \wh f_2)=\Psi(\wh f_1)\otimes \Psi (\wh f_2)$. 
\end{lem}
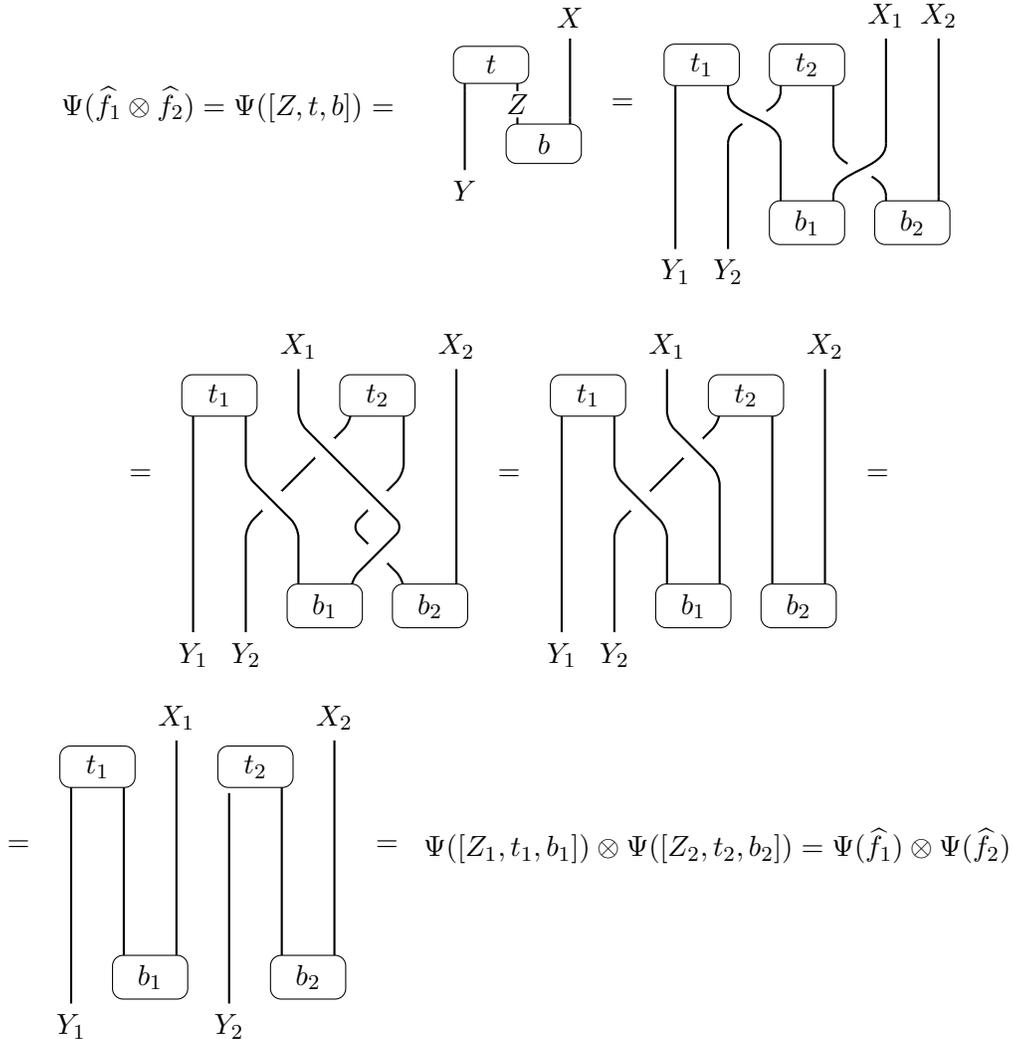
\begin{proof}
\[
\begin{tikzpicture}
[scale=.7,
big/.style={rectangle,rounded corners,fill=white,draw,minimum height=.3cm,minimum width=1cm},
small/.style={rectangle,rounded corners,fill=white,draw,minimum height=.3cm,minimum width=.3cm}]
\begin{scope}[shift={(0,0)}];
\node at (-1.5,-1.25){$\Psi(\wh f_1\otimes\wh f_2)=\Psi([Z,t,b])=$};
\begin{scope}[shift={(3,0)}]
\begin{scope}[shift={(0,0)}]
\draw[thick](0,-.5)--+(0,-.5);
\draw[thick](1,-.5)--+(0,-.5);
\node at (.5,-.5)[big]{$t$};
\end{scope}
\begin{scope}[shift={(1,-1.5)}]
\draw[thick](0,0)--(0,-.5);
\draw[thick](1,0)--(1,-.5);
\node at (.5,-.5)[big]{$b$};
\end{scope}
\node at (1,-1.25){$Z$};
\draw[thick](0,-1)--+(0,-1.5)node[below]{$Y$};
\draw[thick](2,0)node[above]{$X$}--+(0,-1.5);
\end{scope}
\node at (6,-1.25){$=$};
\begin{scope}[shift={(7,0)}]
\begin{scope}[shift={(0,0)}]
\draw[thick](0,-.5)--+(0,-.5);
\draw[thick](1,-.5)--+(0,-.5);
\node at (.5,-.5)[big]{$t_1$};
\end{scope}
\begin{scope}[shift={(2,0)}]
\draw[thick](0,-.5)--+(0,-.5);
\draw[thick](1,-.5)--+(0,-.5);
\node at (.5,-.5)[big]{$t_2$};
\end{scope}
\begin{scope}[shift={(1,-1)}]
	\draw[thick] (1,0).. controls +(down:.5cm) and +(up:.5cm)..+(-1,-1);
		\node[circle,fill=white] at (.5,-.5){} ;
	\draw[thick] (0,0).. controls +(down:.5cm) and +(up:.5cm)..+(+1,-1);
\end{scope}
\begin{scope}[shift={(3,-2)}]
	\draw[thick] (0,0).. controls +(down:.5cm) and +(up:.5cm)..+(+1,-1);
			\node[circle,fill=white] at (.5,-.5){} ;
		\draw[thick] (1,0).. controls +(down:.5cm) and +(up:.5cm)..+(-1,-1);
\end{scope}
\begin{scope}[shift={(2,-3)}]
\draw[thick](0,0)--(0,-.5);
\draw[thick](1,0)--(1,-.5);
\node at (.5,-.5)[big]{$b_1$};
\end{scope}
\begin{scope}[shift={(4,-3)}]
\draw[thick](0,0)--(0,-.5);
\draw[thick](1,0)--(1,-.5);
\node at (.5,-.5)[big]{$b_2$};
\end{scope}
\draw[thick](0,-1)--+(0,-3)node[below]{$Y_1$};
\draw[thick](1,-2)--+(0,-2)node[below]{$Y_2$};
\draw[thick](2,-2)--+(0,-1);
\draw[thick](3,-1)--+(0,-1);
\draw[thick](4,0)node[above]{$X_1$}--+(0,-2);
\draw[thick](5,0)node[above]{$X_2$}--+(0,-3);
\end{scope}
\end{scope}
\end{tikzpicture}
\]
\[
\begin{tikzpicture}
[scale=.7,
big/.style={rectangle,rounded corners,fill=white,draw,minimum height=.3cm,minimum width=1cm},
small/.style={rectangle,rounded corners,fill=white,draw,minimum height=.3cm,minimum width=.3cm}]
\node at (-1,-2){$=$};
\begin{scope}[shift={(0,0)}]
\draw[thick,rounded corners](3,-.5)--(3,-1)--(1,-3)--(1,-5)node[below]{$Y_2$};
\node[fill=white] at (2.5,-1.5){};
\node[fill=white] at (1.5,-2.5){};
\draw[thick](0,-.5)--(0,-5)node[below]{$Y_1$};
\draw[thick,rounded corners](1,-.5)--(1,-2)--(2,-3)--(2,-4.5);
\draw[thick,rounded corners](4,-.5)--(4,-2)--(3,-3)--(4,-4)--(4,-4.5);
\node[fill=white] at (3.5,-2.5){};
\node[fill=white] at (3.5,-3.5){};
\draw[thick,rounded corners](2,0)node[above]{$X_1$}--(2,-1)--(4,-3)--(3,-4)--(3,-4.5);
\draw[thick,rounded corners](5,0)node[above]{$X_2$}--(5,-4.5);
\node at (.5,-.5)[big]{$t_1$};
\node at (3.5,-.5)[big]{$t_2$};
\node at (2.5,-4.5)[big]{$b_1$};
\node at (4.5,-4.5)[big]{$b_2$};
\node at (6,-2){$=$};
\end{scope}
\begin{scope}[shift={(7,0)}]
\draw[thick,rounded corners](3,-.5)--(3,-1)--(1,-3)--(1,-5)node[below]{$Y_2$};
\node[fill=white] at (2.5,-1.5){};
\node[fill=white] at (1.5,-2.5){};
\draw[thick](0,-.5)--(0,-5)node[below]{$Y_1$};
\draw[thick,rounded corners](1,-.5)--(1,-2)--(2,-3)--(2,-4.5);
\draw[thick,rounded corners](2,0)node[above]{$X_1$}--(2,-1)--(3,-2)--(3,-4.5);
\draw[thick,rounded corners](4,-.5)--(4,-4.5);
\draw[thick,rounded corners](5,0)node[above]{$X_2$}--(5,-4.5);
\node at (.5,-.5)[big]{$t_1$};
\node at (3.5,-.5)[big]{$t_2$};
\node at (2.5,-4.5)[big]{$b_1$};
\node at (4.5,-4.5)[big]{$b_2$};
\node at (6,-2){$=$};
\end{scope}
\end{tikzpicture}
\]
\[
\begin{tikzpicture}
[scale=.7,
big/.style={rectangle,rounded corners,fill=white,draw,minimum height=.3cm,minimum width=1cm},
small/.style={rectangle,rounded corners,fill=white,draw,minimum height=.3cm,minimum width=.3cm}]
\node at (-1,-2){$=$};
\begin{scope}[shift={(0,0)}]
\draw[thick](0,-.5)--(0,-5)node[below]{$Y_1$};
\draw[thick](1,-.5)--(1,-4.5);
\draw[thick](3,-1)--(3,-5)node[below]{$Y_2$};
\draw[thick](2,0)node[above]{$X_1$}--(2,-4.5);
\draw[thick](4,-.5)--(4,-4.5);
\draw[thick](5,0)node[above]{$X_2$}--(5,-4.5);
\node at (.5,-.5)[big]{$t_1$};
\node at (3.5,-.5)[big]{$t_2$};
\node at (1.5,-4.5)[big]{$b_1$};
\node at (4.5,-4.5)[big]{$b_2$};
\node at (6,-2){$=$};
\end{scope}
\node[right] at (6.5,-2){$\Psi([Z_1,t_1,b_1])\otimes \Psi([Z_2,t_2,b_2])=\Psi(\wh f_1)\otimes \Psi(\wh f_2)$};
\end{tikzpicture}
\]

Here the fourth equality is a consequence of Lemma \ref{lem:over/under_crossing} (applied to $g=t_2\colon \one\to W=Y_2\otimes Z_2$, $V=X_1$). The fifth equality follows from Proposition \ref{prop:isotopy} and the obvious isotopy. The sixth equality is  again a consequence of Lemma \ref{lem:over/under_crossing} (applied to $f=b_1\colon W=Z_1\otimes X_1\to \one$, $V=Y_2$).
\end{proof} 
\begin{proof}[Proof of part (3) of Theorem \ref{thm:whtr}]
Let $\wh f_i=[Z_i,t_i,b_i]\in \wh\sC(X_i,X_i)$, $i=1,2$.  Then 

\begin{equation}\label{eq:whtr_tensorproduct}
\begin{tikzpicture}
[scale=.7,
big/.style={rectangle,rounded corners,fill=white,draw,minimum height=.3cm,minimum width=1cm},
small/.style={rectangle,rounded corners,fill=white,draw,minimum height=.3cm,minimum width=.3cm}]
\node at (0,-2){$\wh\tr(\wh f_1\otimes \wh f_2)=\wh\tr(Z,t,b)=$};
\begin{scope}[shift={(4,0)}];
\draw[thick](0,-.5)--(0,-5);
\draw[thick](1,-.5)--(1,-5);
\node at (.5,-.5)[big]{$t$};
\node at (.5,-2)[big]{$c_{X,Z}$};
\node at (.5,-5)[big]{$b$};
\node at (1,-3.5)[small]{$\theta_X$};
\node at (-.3,-1.25){$X$} ;
\node at (1.3,-1.25){$Z$} ;
\node at (-.3,-2.75){$Z$} ;
\node at (1.3,-2.75){$X$} ;
\node at (1.3,-4.25){$X$} ;
\end{scope}
\end{tikzpicture}
\end{equation}
where $Z:=Z_1\otimes Z_2$, $X:=X_1\otimes X_2$, and $t$, $b$ are as in Definition \ref{def:wh_tensor}. We note that $\theta_X=\theta_{X_1\otimes X_2}$ can be expressed in terms of $\theta_1:=\theta_{X_1}$ and $\theta_2:=\theta_{X_2}$ as in Equation  \ref{eq:twist} and that the braiding isomorphism $c_{X,Z}$ can be expressed graphically as 
\[
\begin{tikzpicture}[scale=.5]
\node at (-5.5,-2.5){$c_{X,Z}=c_{X_1\otimes X_2,Z_1\otimes Z_2}=$};
\draw[thick](2,0)node[above]{$Z_1$}--(2,-1)--(0,-3)--(0,-5);
\draw[thick](3,0)node[above]{$Z_2$}--(3,-2)--(1,-4)--(1,-5);
\node[circle, fill=white] at (1.5,-1.5){};
\node[circle, fill=white] at  (2.5,-2.5){};
\node[circle, fill=white] at (.5,-2.5){};
\node[circle, fill=white] at (1.5,-3.5){};
\draw[thick](0,0)node[above]{$X_1$}--(0,-2)--(2,-4)--(2,-5);
\draw[thick](1,0)node[above]{$X_2$}--(1,-1)--(3,-3)--(3,-5);
\end{tikzpicture}
\]
It follows that
\[
\begin{tikzpicture}
[scale=.7,
big/.style={rectangle,rounded corners,fill=white,draw,minimum height=.3cm,minimum width=1cm},
small/.style={rectangle,rounded corners,fill=white,draw,minimum height=.3cm,minimum width=.3cm}]
\node at (-3,-5){$\wh\tr(\wh f_1\otimes \wh f_2)=$};
\begin{scope}[shift={(0,0)}]
\draw[thick,rounded corners](3,0)--(3,-3)--(1,-5)--(1,-8)--(2,-9)--(2,-10);
\node[circle,fill=white] at (2.5,-3.5){} ;
\node[circle,fill=white] at (1.5,-4.5){} ;
\node[circle,fill=white] at (1.5,-8.5){} ;
\draw[thick,rounded corners](1,0)--(1,-1)--(2,-2)--(0,-4)--(0,-10);
\node[circle,fill=white] at (1.5,-2.5){} ;
\node[circle,fill=white] at (.5,-3.5){} ;
\draw[thick,rounded corners](2,0)--(2,-1)--(1,-2)--(3,-4)--(3,-5)--(2,-6)--(2,-7)--(3,-8)--(3,-10);
\node[circle,fill=white] at (1.5,-1.5){} ;
\draw[thick](1.3,-1.3)--(1.7,-1.7);
\node[circle,fill=white] at (2.5,-5.5){} ;
\node at (2,-6.5)[small]{$\theta_{{2}}$};
\draw[thick,rounded corners](0,0)--(0,-3)--(3,-6)--(3,-7)--(1,-9)--(1,-10);
\node at (3,-6.5)[small]{$\theta_{{1}}$};
\node[circle,fill=white] at (2.5,-7.5){} ;
\draw[thick] (2.3,-7.3)--(2.7,-7.7);
\node at (.5,0)[big]{$t_1$};
\node at (2.5,0)[big]{$t_2$};
\node at (.5,-10)[big]{$b_1$};
\node at (2.5,-10)[big]{$b_2$};
\node at (.3,-.75){$X_1$} ;
\node at (1.3,-.75){$Z_1$} ;
\node at (2.3,-.75){$X_2$} ;
\node at (3.3,-.75){$Z_2$} ;
\node at (.3,-9.25){$Z_1$} ;
\node at (1.3,-9.25){$X_1$} ;
\node at (2.3,-9.25){$Z_2$} ;
\node at (3.3,-9.25){$X_2$} ;
\node at (4,-5){$=$};
\end{scope}
\begin{scope}[shift={(5,0)}]
\draw[thick,rounded corners](3,0)--(3,-3)--(1,-5)--(1,-7)--(2,-8)--(2,-10);
\node[circle,fill=white] at (2.5,-3.5){} ;
\node[circle,fill=white] at (1.5,-4.5){} ;
\node[circle,fill=white] at (1.5,-7.5){} ;
\draw[thick,rounded corners](1,0)--(1,-1)--(2,-2)--(0,-4)--(0,-10);
\node[circle,fill=white] at (1.5,-2.5){} ;
\node[circle,fill=white] at (.5,-3.5){} ;
\draw[thick,rounded corners](2,0)--(2,-1)--(1,-2)--(3,-4)--(3,-5)--(2,-6)--(3,-7)--(3,-10);
\node[circle,fill=white] at (1.5,-1.5){} ;
\draw[thick](1.3,-1.3)--(1.7,-1.7);
\node[circle,fill=white] at (2.5,-5.5){} ;
\draw[thick,rounded corners](0,0)--(0,-3)--(3,-6)--(1,-8)--(1,-10);
\node[circle,fill=white] at (2.5,-6.5){} ;
\draw[thick] (2.3,-6.3)--(2.7,-6.7);
\node at (1,-8.5)[small]{$\theta_{{1}}$};
\node at (3,-8.5)[small]{$\theta_{{2}}$};
\node at (.5,0)[big]{$t_1$};
\node at (2.5,0)[big]{$t_2$};
\node at (.5,-10)[big]{$b_1$};
\node at (2.5,-10)[big]{$b_2$};
\node at (4,-5){$=$};
\end{scope}
\begin{scope}[shift={(10,0)}]
\draw[thick,rounded corners](3,0)--(3,-1)--(4,-2)--(4,-6)--(2,-8)--(2,-10);
\node[circle,fill=white] at (3,-7){} ;
\draw[thick,rounded corners](1,0)--(1,-1)--(2,-2)--(0,-4)--(0,-10);
\node[circle,fill=white] at (1.5,-2.5){} ;
\node[circle,fill=white] at (.5,-3.5){} ;
\draw[thick,rounded corners](2,0)--(2,-1)--(1,-2)--(3,-4)--(3,-5)--(2,-6)--(3,-7)--(3,-10);
\node[circle,fill=white] at (1.5,-1.5){} ;
\draw[thick](1.3,-1.3)--(1.7,-1.7);
\node[circle,fill=white] at (2.5,-5.5){} ;
\draw[thick,rounded corners](0,0)--(0,-3)--(3,-6)--(1,-8)--(1,-10);
\node[circle,fill=white] at (2.5,-6.5){} ;
\draw[thick] (2.3,-6.3)--(2.7,-6.7);
\node at (1,-8.5)[small]{$\theta_{{1}}$};
\node at (3,-8.5)[small]{$\theta_{{2}}$};
\node at (.5,0)[big]{$t_1$};
\node at (2.5,0)[big]{$t_2$};
\node at (.5,-10)[big]{$b_1$};
\node at (2.5,-10)[big]{$b_2$};
\node at (5,-5){$=$};
\end{scope}
\end{tikzpicture}
\]
\[
\begin{tikzpicture}
[scale=.7,
big/.style={rectangle,rounded corners,fill=white,draw,minimum height=.3cm,minimum width=1cm},
small/.style={rectangle,rounded corners,fill=white,draw,minimum height=.3cm,minimum width=.3cm}]
\begin{scope}[shift={(0,0)}]
\draw[thick,rounded corners](3,0)--(3,-1)--(4,-2)--(4,-5)--(2,-7)--(2,-9);
\node[circle,fill=white] at (3,-6){} ;
\draw[thick,rounded corners](1,0)--(1,-1)--(2,-2)--(0,-4)--(0,-9);
\node[circle,fill=white] at (1.5,-2.5){} ;
\node[circle,fill=white] at (.5,-3.5){} ;
\draw[thick,rounded corners](2,0)--(2,-1)--(0,-3)--(2,-5)--(3,-6)--(3,-9);
\node[circle,fill=white] at (1.5,-1.5){} ;
\draw[thick](1.3,-1.3)--(1.7,-1.7);
\node[circle,fill=white] at (1,-2){} ;
\draw[thick,rounded corners](0,0)--(0,-1)--(3,-4)--(3,-5)--(1,-7)--(1,-9);
\node[circle,fill=white] at (2.5,-5.5){} ;
\draw[thick] (2.3,-5.3)--(2.7,-5.7);
\node at (1,-7.5)[small]{$\theta_{{1}}$};
\node at (3,-7.5)[small]{$\theta_{{2}}$};
\node at (.5,0)[big]{$t_1$};
\node at (2.5,0)[big]{$t_2$};
\node at (.5,-9)[big]{$b_1$};
\node at (2.5,-9)[big]{$b_2$};
\node at (5,-5){$=$};
\end{scope}

\begin{scope}[shift={(6,0)}]
\draw[thick,rounded corners](3,0)--(3,-1)--(4,-2)--(4,-5)--(2,-7)--(2,-9);
\node[circle,fill=white] at (3,-6){} ;
\draw[thick,rounded corners](1,0)--(1,-1)--(2,-2)--(0,-4)--(0,-9);
\node[circle,fill=white] at (1.5,-2.5){} ;
\node[circle,fill=white] at (.5,-3.5){} ;
\node[circle,fill=white] at (1.5,-1.5){} ;
\draw[thick,rounded corners](2,0)--(2,-1)--(0,-3)--(3,-6)--(3,-9);
\draw[thick,rounded corners](0,0)--(0,-1)--(3,-4)--(3,-5)--(1,-7)--(1,-9);
\node[circle,fill=white] at (1,-2){} ;
\draw[thick](1.2,-1.8)--(.8,-2.2);
\node[circle,fill=white] at (2.5,-5.5){} ;
\draw[thick] (2.3,-5.3)--(2.7,-5.7);
\node at (1,-7.5)[small]{$\theta_{{1}}$};
\node at (3,-7.5)[small]{$\theta_{{2}}$};
\node at (.5,0)[big]{$t_1$};
\node at (2.5,0)[big]{$t_2$};
\node at (.5,-9)[big]{$b_1$};
\node at (2.5,-9)[big]{$b_2$};
\node at (5,-5){$=$};
\end{scope}
\begin{scope}[shift={(12,0)}]
\draw[thick,rounded corners](1,0)--(1,-3)--(0,-4)--(0,-9);
\node[circle,fill=white] at (.5,-3.5){} ;
\draw[thick,rounded corners](0,0)--(0,-3)--(1,-4)--(1,-9);
\draw[thick,rounded corners](3,0)--(3,-5.5)--(2,-6.5)--(2,-9);
\node[circle,fill=white] at (2.5,-6){} ;
\draw[thick,rounded corners](2,0)--(2,-5.5)--(3,-6.5)--(3,-9);
\node at (1,-7.5)[small]{$\theta_{{1}}$};
\node at (3,-7.5)[small]{$\theta_{{2}}$};
\node at (.5,0)[big]{$t_1$};
\node at (2.5,0)[big]{$t_2$};
\node at (.5,-9)[big]{$b_1$};
\node at (2.5,-9)[big]{$b_2$};
\end{scope}
\node at (18,-5){$=\wh\tr(\wh f_1)\otimes\wh\tr(\wh f_2)$};
\end{tikzpicture}
\]
Here the first equality is obtained from equation \eqref{eq:whtr_tensorproduct} by combining the pictures for the morphisms $t$, $c_{X,Z}$, $\theta_X$ and $b$ that $\wh\tr(\wh f_1\otimes \wh f_2)$ is a composition of. 
The second equality follows from the naturality of the braiding isomorphism (see picture \eqref{eq:s_naturality}).
The third equality is a consequence of Proposition \ref{prop:isotopy} and an isotopy moving the strand with label $Z_2$ to the right.
The fourth equality is a consequence of an isotopy moving the strand with label $X_2$ left and down.
The fifth equality follows from Lemma \ref{lem:over/under_crossing} with $g=t_1$ and the relations \eqref{eq:relations}.
The sixth equality is a consequence of Proposition \ref{prop:isotopy} and an isotopy which moves the strands labeled $X_1$ and $Z_1$ to the left, and the strand labeled $X_2$ to the right.

\end{proof}
\begin{proof}[Proof of part (3) of Theorem \ref{thm:main2}]
Let $f_i\in \sC^{tk}(X_i,Y_i)$, $g_i\in \sC^{tk}(Y_i,X_i)$ for $i=1,2$, and let $\wh f_i\in \wh\sC(X_i,Y_i)$ with $\Psi(\wh f_i)=f_i$. Then according to Lemma \ref{lem:Psi_multiplicative}, the map
\[
\Psi\colon \wh\sC(X_1\otimes X_2,Y_1\otimes Y_2)\ra \sC(X_1\otimes X_2,Y_1\otimes Y_2)
\]
sends $\wh f_1\otimes \wh f_2$ to $f_1\otimes f_2$ and hence
\begin{align*}
\tr(f_1\otimes f_2,g_1\otimes g_2)
&=\wh\tr\left((\wh f_1\otimes \wh f_2)\circ( g_1\otimes  g_2)\right)
=\wh\tr\left((\wh f_1\circ g_1)\otimes (\wh f_2\circ g_2)\right)\\
&=\wh \tr(\wh f_1\circ g_1)\wh \tr(\wh f_2\circ g_2)
=\tr(f_1,g_1)\tr(f_2,g_2)
\end{align*}
Here the third equality uses the multiplicative property of $\wh\tr$, i.e., statement (3) of Theorem \ref{thm:whtr}.
\end{proof}

\end{document}